%% file: BLF-PLF-localization-inequalities.tex
\pdfoutput= 1 
\RequirePackage{fixltx2e}
\documentclass[11pt]{article}  
\usepackage[utf8]{inputenc}
\usepackage[T1]{fontenc}
\usepackage{amsmath, amsfonts, amsthm, amssymb, amscd, aliascnt}
\usepackage{a4wide}
\usepackage[normalem]{ulem}     
\usepackage{relsize}
\usepackage{cancel}
\usepackage{slashed}
\usepackage{stmaryrd}
\usepackage{esint}
\usepackage{amsmath,amssymb,slashed}
\usepackage{bbm}
\usepackage{float} 
\usepackage{setspace}
\usepackage{titlesec}
\titlespacing*{\paragraph}{0pt}{1.25ex plus 1ex minus .2ex}{.5em}
\usepackage{subcaption} 
\usepackage{comment}
\usepackage{longtable}
\usepackage{perpage}
\MakePerPage{footnote}
\numberwithin{table}{section}
\usepackage[mathcal]{euscript} 
\usepackage{mathrsfs} 
\usepackage{accents, verbatim}
\DeclareMathAlphabet\mathbfcal{OMS}{cmsy}{b}{n}  
\makeatletter
\@ifundefined{c@localmathalphabets}{}{\setcounter{localmathalphabets}{0}}
\makeatother

\input{epsf}
\usepackage{pdfpages, enumerate, graphicx, xcolor, array, booktabs, microtype}
\usepackage{fancyhdr}
\usepackage[nottoc]{tocbibind} 
\usepackage[pdfusetitle]{hyperref}
\hypersetup{linktoc = all}
\hypersetup{hidelinks}
\hypersetup{bookmarksnumbered,hypertexnames=false}

\numberwithin{equation}{section}
\numberwithin{table}{section}
\numberwithin{figure}{section}
\usepackage{footnotebackref}
\input{macros-localization}

\newcounter{fig}[section]

\usepackage{etoolbox}
\AtBeginEnvironment{table}{\refstepcounter{fig}}
\AtBeginEnvironment{figure}{\refstepcounter{fig}} 
\newcommand{\compactpart}[1]{%
\vskip1.2cm 
  \refstepcounter{part}%
  \phantomsection
  \addcontentsline{toc}{part}{\thepart\texorpdfstring{\quad}{ }#1}%
  \vspace{1.5em}
  \noindent{\Large\bfseries Part \thepart.\quad #1\par}
  \vspace{1em}
}
\begin{document} 

\phantomsection\label{firstpage}
 
\title{Harmonic, radial, and shell stability of the 
\\
weighted Einstein constraints on the sphere at infinity} 

\author{Bruno Le Floch\texorpdfstring{$^{1}$}{} and Philippe G. LeFloch\texorpdfstring{$^2$}{}}

\date{}

\maketitle 
\footnotetext[1]{Laboratoire de Physique Théorique et Hautes Énergies, Sorbonne Universit\'e \& Centre National de la Recherche Scientifique, 4 Place Jussieu, 75252 Paris, France. Email: {\tt blefloch@lpthe.jussieu.fr}.
}
\footnotetext[2]{Laboratoire Jacques-Louis Lions, Sorbonne Universit\'e \& Centre National de la Recherche Scientifique, 4 Place Jussieu, 75252 Paris, France. Email: {\tt contact@philippelefloch.org}.
\hfill  This version: July 2026.\!%
}

\begin{abstract}
We consider fourth-order and second-order partial differential operators localized on domains of the sphere in arbitrary dimension. These operators arise as weighted compositions of the linearized Einstein constraint operators and their adjoints, and played a key role in our resolution of the optimal localization problem in general relativity, also referred to as the gravitational shielding problem. To control the asymptotic behavior of solutions to Einstein's constraints in our companion paper (B.~Le Floch and P.G.~LeFloch, \emph{Optimal localization for the Einstein constraints}, preprint arXiv:2312.17706), we introduced the notions of \emph{harmonic, radial, and shell stability}. Harmonic stability controls the borderline harmonic modes, radial stability governs the radial evolution of spherical averages, and shell stability controls the coupled radial--angular evolution of solutions. In the present paper, we establish that these stability properties follow from weighted Poincar\'e, Korn, and Hardy inequalities. Furthermore, in arbitrary dimension, we investigate the behavior of the associated geometric constants, and conclude that the stability conditions hold for a broad class of localization functions. Crucially, our theory applies to \emph{arbitrarily small localization domains}, corresponding, in our resolution of the Einstein constraints, to gluing cones with \emph{arbitrarily small aperture}. At the opposite extreme, our stability conditions also hold on the entire sphere, corresponding to the absence of localization, provided that the space dimension is at most \(17\). This completes, for gluing cones of arbitrarily small aperture in every dimension, the program initiated by A.~Carlotto and R.~Schoen on gravitational shielding and the construction of solutions enjoying super-harmonic decay estimates. A natural further problem is to determine the optimal constants in these inequalities.
In the course of our proofs, we introduce Hamiltonian and momentum functionals, which we call \emph{shell functionals}, and show that they enjoy monotonicity and semi-coercivity properties; their structure also suggests possible analogies with functionals arising in other curvature-related geometric problems.
\end{abstract}


\clearpage 

{\small

\setcounter{secnumdepth}{2} 
\setcounter{tocdepth}{2}
\tableofcontents

} 


\section{Introduction}
\label{section-1}

\subsection{Main objective}
\label{sectionN1-1}

We consider two elliptic operators on conical domains of~$\RR^n$: a fourth-order scalar operator and a second-order vector-valued  operator, referred to here as the \emph{localized Hamiltonian and momentum operators}.
These operators depend on a radial exponent $p\in(0,n-2)$ and an angular localization function~$\lambda$.
They arise when gluing asymptotically Euclidean scalar-flat Riemannian manifolds, and more generally solutions of Einstein's constraint equations, using a variational method pioneered by Carlotto and Schoen~\cite{CarlottoSchoen}.
These variational and linearization procedures belong to a deformation theory for the Einstein equations, developed first in~\cite{FischerMarsden-1973,Moncrief-1975}.
The present paper is a companion to~\cite{LL-optimal-main}, where geometric inequalities involving these two operators were used to obtain an optimal radial decay of the resulting glued solutions at each asymptotic end.
It can, however, be read \emph{independently}. A short account of the gravitational-shielding aspect of this theory appeared earlier in~\cite{LL-Letter}.

In~\cite{LL-optimal-main}, we introduced three structural conditions governing the asymptotic behavior of localized solutions: \emph{harmonic stability} controls angular profiles with borderline radial decay, \emph{radial stability} governs the radial decay of spherical averages of solutions, and \emph{shell stability} concerns fluctuations away from these averages on each sphere. The main objective of the present paper is to provide convenient sufficient conditions for these three stability properties to hold. 

The localized Hamiltonian and momentum operators, applied to functions with a specific power of the radius~$r$ on~$\RR^n$, reduce to a pair of elliptic operators $\ssA,\ssB$ on a subdomain of the (unit) sphere $\Sbb^{n-1}\subset\RR^n$.
For the reader's convenience, the first part of the paper (Sections~\ref{section-2} to~\ref{section-4}) is formulated in terms of $\ssA,\ssB$, without any reference to the Einstein equations.
This suffices to obtain harmonic and radial stability, before moving on to shell stability in the second part of the paper (Sections~\ref{section-6} to~\ref{section-9}).
A central ingredient there is the construction of Hamiltonian and momentum \emph{shell functionals} enjoying certain \emph{monotonicity and semi-coercivity properties.}

The pivotal \autoref{section-5} explains how our present results complement those of our companion paper~\cite{LL-optimal-main} by providing broad classes of localization functions~$\lambda$ that satisfy all three stability conditions.
This paper therefore completes the localization program initiated by Carlotto and Schoen~\cite{CarlottoSchoen} and continued in~\cite{LL-optimal-main}. Explicit examples and complementary consequences of this localization theory will be developed in~\cite{LL-next}.

More precisely, we prove \emph{quantitative sufficient conditions} for harmonic, radial, and shell stability in terms of weighted Poincar\'e, Korn, Korn--Poincar\'e, and Hardy constants associated with the localization function $\lambda:\Sbb^{n-1}\to[0,+\infty)$. In particular, for every function $\lambda$, the Hamiltonian harmonic and radial conditions, as well as all three momentum conditions, hold when the parameter $p$ is taken to be sufficiently close to $n- 2$ (a minor restriction in our construction scheme in \cite{LL-optimal-main}); on the other hand, the Hamiltonian shell condition is established under an explicit smallness condition on Poincaré constants. Furthermore, the scaling results of \autoref{section-5} then show that all these functional inequalities are compatible with spherical localization domains of \emph{arbitrarily small diameter} ---corresponding, for asymptotically conical gluing, to \emph{arbitrarily small aperture}. At the opposite extreme, we also prove stability when $\lambda$ is constant on the entire sphere, corresponding to the absence of localization, in dimensions \(3\leq n\leq17\).

As a by-product, our method yields super-harmonic estimates for gluing, for instance, scalar-flat metrics on conical localization domains, a problem of interest in Riemannian geometry. We expect the shell functionals and weighted inequalities developed here to be relevant to other geometric problems and elliptic systems. Indeed, the weighted estimates established here suggest a broader use for the analytic mechanism of this paper.  On a foliation by spherical slices, they convert radial identities into effective control of tangential fluctuations even when the underlying functionals are only semi-coercive.  A comparable combination of slice-wise propagation and weighted angular control may be useful for geometric variational and evolution equations with degenerate modes.  The effectiveness of monotonicity methods in the analysis of concentration and singular limits provides further motivation for pursuing this connection; see, for instance,~\cite{BrendleHuisken:2017,CDHS,ColdingMinicozzi}.

From a broader viewpoint, the present results suggest two natural directions for further study. The first concerns the determination of sharp weighted Poincar\'e, Korn, Korn--Poincar\'e, and Hardy constants, for which optimal-transport and symmetrization methods may be relevant. The second concerns the extension of the gluing construction to localization domains substantially more general than asymptotic cones, once suitable analogues of the present weighted inequalities and shell estimates have been established.


\subsection{Spherical localization framework}
\label{sectionN1-spherical}

Fix an integer $n\ge 3$. Let $\Sbb^{n-1}=\Sbb^{n-1}_1\subset\RR^n$ be the unit sphere endowed with its round metric~$\gslash$ and Levi--Civita connection~$\nablaslash$.

\begin{definition}\label{def:localization-function}
A \underline{localization function} is a non-negative function \(\lambda:\Sbb^{n-1}\to[0,\lambda_0]\), for some \(\lambda_0>0\), whose positivity domain
\be
\Lambda\coloneqq\operatorname{int}(\operatorname{supp}\lambda)=\{\lambda>0\}
\ee
is connected and has smooth boundary, subject to the following conditions. 
\bei

\item If \(\Lambda=\Sbb^{n-1}\), then \(\lambda\) is smooth and strictly positive on the whole sphere.

\item If \(\del\Lambda\neq\emptyset\), then \(\lambda\) restricted to $\overline\Lambda$ is positive in \(\Lambda\) and vanishes with a non-zero gradient \(|\nablaslash\lambda|>0\) on \(\del\Lambda\).  Its extension by zero on \(\Sbb^{n-1}\setminus\overline\Lambda\) is only piecewise smooth.
\eei
\noindent As \(\lambda\) determines \(\Lambda\), the latter will often be omitted from notation.
\end{definition}

Fix $p\in(0,n- 2)$ and an integer $\expoP\ge 2$. To any localization function $\lambda$, we associate the \textbf{localization measure} on its domain $\Lambda$ defined by 
\bel{equa--09}
d\chi \coloneqq \lambda^{2\expoP} \, d\Vol_{\Sbb^{n-1}}.
\ee
Here $d\Vol_{\Sbb^{n-1}}$, also denoted by $d\xh$ in the following, is the standard volume form on $\Sbb^{n-1}$ restricted to~$\Lambda$. 

We recall some notation from~\cite{LL-optimal-main}; see also \autoref{section-5}, below. Further notation can be found in \autoref{section-2}, below, while all numerical constants such as $a_{n,p}, b_{n,p}, \ldots$ are defined in \autoref{section-2} and collected in \autoref{sec-appendix-A}, below. The  (spherical part of the) \emph{localized Hamiltonian operator} $\ssA^\lambda[\nu]$ is
\bel{equa--11}
\aligned
\ssA^\lambda[\nu] \coloneqq \lambda^{- 2\expoP} \Bigl( (n- 2) \, \Deltaslash ( \lambda^{2\expoP} \Deltaslash \nu) 
& + \nablaslash^a\nablaslash^b\bigl(  \lambda^{2\expoP} \nablaslash_a\nablaslash_b \nu\bigr)
\\
& - 2 (a_{n,p}+1) \nablaslash \cdot (  \lambda^{2\expoP} \nablaslash \nu) - c_{n,p} \Deltaslash( \lambda^{2\expoP} \nu) \Bigr),
\endaligned
\ee
where the unknown $\nu: \Lambda \to \RR$ is a scalar-valued field.  On the other hand, the  (spherical part of the) \emph{localized momentum operator} $\ssB^\lambda[\xi]$ has components 
\bel{equa--12}
\aligned
\ssB^{\lambda\perp\perp}[\xiperp]
& = (n-1) \xiperp - \frac{1}{2} \lambda^{- 2\expoP} \nablaslash \cdot \bigl( \lambda^{2\expoP} \nablaslash \xiperp \bigr), 
\\
\ssB^{\lambda\parallel\perp}[\xiperp]_a & = - \frac{1}{2} \nablaslash_a\xiperp - \lambda^{- 2\expoP} \nablaslash_a\bigl(\lambda^{2\expoP} \xiperp\bigr),
\\
\ssB^{\lambda\perp\parallel}[\xipar]
& = \frac{1+a_{n,p}}{2} \lambda^{- 2\expoP} \nablaslash \cdot \bigl(\lambda^{2\expoP} \xipar \bigr) + \nablaslash \cdot \xipar, 
\\
\ssB^{\lambda\parallel\parallel}[\xipar]_a  & = \frac{1+a_{n,p}}{2} \xipar_a
-  \lambda^{- 2\expoP} \nablaslash^b\Bigl(\lambda^{2\expoP} \Sym(\nablaslash \xipar)_{ab} \Bigr).
\endaligned
\ee
Here the unknown is a scalar field $\xiperp$ and a vector field $\xipar$ on $\Lambda$, which can arise as the restriction of a vector field $\xi=(\xiperp,\xipar)$ on~$\RR^n$, decomposed into components normal and tangent to the sphere.  We can either regard $\xipar$ as a vector in $\RR^n$ with components $\xipar_i$ (with $i=1,\dots,n$) or, more geometrically, as a tangent vector to the sphere with components denoted by $\xipar_a$, using abstract Penrose indices $a,b,\dots$ on the unit sphere~$\Sphe^{n-1}$. In our calculations we found it convenient to use both standpoints. In \eqref{equa--12}, we also used the notation
$\Sym(\nablaslash \xipar)_{ab} \coloneqq \frac{1}{2} \bigl( \nablaslash_a \xipar_b + \nablaslash_b \xipar_a\bigr)$.


\subsection{Outline of the main results}
\label{sectionN1-3}

We aim at connecting the stability conditions introduced in~\cite{LL-optimal-main} to concrete geometric inequalities associated with the localization function. We give here a description of the results and the organization of this paper.  

\bei

\item In \autoref{section-2}, we state first some definitions; more precisely, in Definitions~\ref{art2 -def-easy-Hstab}, \ref{def-harmonic-Mstab}, and \ref{def-radial-Mstab}, 
we state the harmonic and radial conditions, while postponing to later sections the presentation of the shell stability conditions. In \autoref{art2 -thm:informal-suff-stab}, we then state our main conclusion that all of our stability conditions hold under standard functional inequalities of Poincar\'e, Korn, and Hardy type.  

\item In \autoref{section-3}, we consider the harmonic Hamiltonian operator~$\ssA^\lambda$ and begin the main analytical part of the paper with the proof of Hamiltonian harmonic and radial stability. We state our conclusion as \autoref{art2 -prop-harm-rad} (harmonic stability) and \autoref{art2 -prop--radial} (radial stability). Furthermore, in \autoref{art2 -lem-nut-small} we also establish a key estimate on the fluctuations (the difference from the average) of the Hamiltonian silhouette function. The silhouette is defined as an element of the kernel of~$\ssA^\lambda$.

\item In \autoref{section-4}, we then turn our attention to the harmonic momentum operator~$\ssB^{\lambda}$. In \autoref{art2 -equa--weighted-Korn}, we establish a weighted Korn inequality for vector fields on the sphere, and next in \autoref{art2 -lem-poinckorn} a weighted Korn--Poincar\'e inequality adapted to the momentum operator $\ssB^{\lambda}$ in the harmonic limit $p \to (n- 2)$ (corresponding to $a_{n,p}\to 0$). These preliminary inequalities eventually lead us to \autoref{art2 -lem-39harmmom} (harmonic stability) and \autoref{art2 -coro-radial-momen} (radial stability).

\item In \autoref{section-5}, we return to the Einstein constraint equations and derive the localized Hamiltonian and momentum operators from the weighted-adjoint construction.  We give their harmonic--spherical decompositions, prove the three weighted Poincar\'e inequalities used in the Hamiltonian analysis, and establish the complete small-domain scaling result.

\item The weighted geometric inequalities are introduced where they enter the analysis.  The Poincar\'e inequalities and their small-domain scaling are discussed in \autoref{sectionN3- 2} and proven in \autoref{sectionN3- 2 -complete}; the weighted Korn and Korn--Poincar\'e inequalities are established in \autoref{art2 -equa--weighted-Korn} and \autoref{art2 -lem-poinckorn}; and the weighted Hardy inequality used for momentum shell stability is stated in \autoref{art2 -equainequ49}.

\item In \autoref{section-6}, we turn our attention to the shell stability condition for the localized Hamiltonian operator. After introducing a quadratic functional on the shells $\Lambda_r$ and its radial identity, we are in a position to present the notion of Hamiltonian shell stability in \autoref{def-shellH} and, in~\autoref{art2 -prop-dj39}, we state the desired stability property. 

\item In \autoref{section-7}, the notion of shell stability for the localized momentum operator is presented in \autoref{def-shell-Mstab} and consists of continuity and semi-coercivity conditions. We then analyze the momentum fluctuation operator in \autoref{art2 -prop-k-fluctuations-m}, while the corresponding stability theorem is postponed to a later section; cf.~\autoref{art2 -prop-semi-mdiss}.

\item In \autoref{section-8}, we introduce a general Ansatz of Hamiltonian shell functionals depending on several coefficients, and establish monotonicity and semi-coercivity properties under suitable inequalities on its defining coefficients. We then complete the proof of~\autoref{art2 -prop-dj39} on the shell stability of the Hamiltonian operator. 

\item In \autoref{section-9}, we consider the proposed momentum shell functional and investigate the positivity of the dissipation functional. We then conclude with the momentum shell stability in \autoref{art2 -prop-semi-mdiss}. 

\eei

Given these results, it would be interesting to extend the proposed method to the setting in which the sphere $\Sbb^{n-1}$ is replaced by an arbitrary compact Riemannian manifold. It would also be of interest to determine sharp constants in our inequalities: while $a_{n,p}, b_{n,p}, \ldots$ are \emph{$\lambda$-independent structure constants}, our \emph{$\lambda$-dependent constants} are likely not optimal. It would be interesting to seek sharper bounds via optimal-transport techniques. 
	

\subsection{Implications for general relativity} 
\label{sectionN1-4}

\subsubsection{Localization with sub-harmonic decay}

We briefly explain how \autoref{art2 -thm:informal-suff-stab} applies to the Einstein constraint equations; the precise analytic statements are given in \autoref{section-5}. Let $\Mbf$ be an $n$-dimensional spacelike hypersurface embedded in a Ricci-flat Lorentzian manifold of dimension $n+1$. Its induced metric $g$ and second fundamental form $k$ satisfy the vacuum Einstein constraint equations, namely\footnote{See for instance~\cite[Chap.~VII]{Choquet-book}.}:
\bel{eq:ee11}
\aligned
R_g + (\Tr_g k)^2 - |k|_g^2 &= 0,
\qquad
\bfDiv_g\bigl(k-(\Tr_g k)g\bigr)=0.
\endaligned
\ee
Here $R_g$ denotes the scalar curvature of $g$, $\Tr_g k$ and $|k|_g$ are respectively the trace and norm of $k$, and $\bfDiv_g$ is the divergence operator. Carlotto and Schoen~\cite{CarlottoSchoen} first discovered the \emph{gravity shielding} phenomenon for a class of asymptotically Euclidean initial data sets, while Chru\'sciel and Delay~\cite{ChruscielDelay-2021} later extended it to substantially more general domains and asymptotic ends. 
In essence, the constraint equations admit families of solutions that can be glued in distinct asymptotically conical regions near infinity. For instance, they can be exactly Euclidean in all angular directions except within a conical sector whose aperture can be made arbitrarily small. More generally, we may interpolate between an exactly Euclidean end and a nontrivial end, such as a Schwarzschild end, across a gluing region whose boundary is conical at infinity. These constructions rely on, and broaden, the earlier gluing techniques developed in compact settings~\cite{ChruscielDelay-memoir,Corvino-2000,CorvinoSchoen}. Related scalar-curvature deformation and localized gluing procedures were developed in~\cite{CorvinoEM,Delay}. Further localized deformations for initial data satisfying the dominant energy condition were obtained in~\cite{CorvinoHuang}. We also refer to the surveys~\cite{Carlotto-Review,Chrusciel-bourbaki,GallowayMiaoSchoen}.


\subsubsection{Localization with (super-)harmonic decay} 

Quantitatively, the Carlotto--Schoen construction provides control in the gluing region with \emph{sub-harmonic} decay $r^{-n+2-\eta}$, where $\eta>0$ is small. Carlotto and Schoen~\cite{CarlottoSchoen,Carlotto-Review} conjectured that the borderline case $\eta=0$, corresponding to harmonic decay, should also be attainable.
In~\cite{LeFlochNguyen}, this construction was revisited from the new viewpoint of \emph{asymptotic localization}: instead of requiring exact agreement with a model solution throughout the complement of the gluing cone, we allowed super-harmonic perturbations there.

In contrast, in~\cite{LL-optimal-main}, we solved the original (exact) localization problem and introduced a new methodology to establish metric estimates at super-harmonic rates $r^{-n+2+\eta}$, for $\eta>0$, with respect to a radial distance function. More precisely, we defined a \emph{seed-to-solution map}, viewed as a nonlinear projection from approximate solutions to exact solutions of the Einstein constraint equations. This yielded an \emph{optimal localization} theory producing solutions whose asymptotic behavior is controlled beyond the ADM energy--momentum term. The resulting super-harmonic decay theorem is conditional on harmonic, radial, and shell stability for both the Hamiltonian and momentum operators. Many other developments are also reviewed in~\cite{LL-optimal-main}.


\subsection{Notation}
\label{sectionN1-5}

Throughout the paper, $n\ge 3$ denotes the dimension of the Euclidean space $(\RR^n,\delta)$ and $\delta$ denotes the Euclidean metric. We write $\Deltaslash$ and $\nablaslash$ for the Laplace--Beltrami operator and Levi-Civita connection on the sphere $(\Sphe^{n-1}, \gslash)$ endowed with its standard round metric, and $d\xh$ for the standard homogeneous measure on the unit sphere.
Weighted Sobolev spaces are defined with standard conventions; cf.~\cite{LL-optimal-main}. In particular, we use the notation  
$\| v\|^2_{L^2_{- \expoP}(\Lambda)}$ for the norm of $L^2$ functions $v:\Lambda \subset \Sphe^{n-1} \to \RR$ defined on the $(n-1)$-sphere endowed with the measure $d\chi = \lambda^{2\expoP} d\xh$ (a real number $\expoP \geq 2$ being fixed), and similarly for Sobolev spaces $H^k_{- \expoP}(\Lambda)$. For convenience we normalize the norm as
\bel{equa-norm-poids} 
\| v\|^2_{H^k_{- \expoP}(\Lambda)} \coloneqq \sum_{0\leq j\leq k} \fint_{\Lambda} |\nablaslash^j v|_{\gslash}^2 \, d\chi
= \frac{1}{\aire[\Lambda,\lambda]} \sum_{0\leq j\leq k} \int_{\Lambda} |\nablaslash^j v|_{\gslash}^2 \, d\chi. 
\ee
Furthermore, the weighted average of a function $f\colon \Lambda \to \RR$ defined on an open set $\Lambda \subset \Sphe^{n-1}$ of the $(n-1)$-dimensional, unit sphere $\Sphe^{n-1} \subset \RR^n$ is denoted by 
\bel{equa-average} 
\la f \ra \coloneqq \fint_{\Lambda} f \, d\chi 
\coloneqq {1 \over \aire[\Lambda,\lambda]} \int_{\Lambda} f \, d\chi,   
\qquad  
\aire[\Lambda,\lambda] \coloneqq \int_{\Lambda} \, d\chi.
\ee

For a list of relevant constants, we refer to \autoref{sec-appendix-A}.


\compactpart{Asymptotic structure on the sphere at infinity}

\section{Definitions and main stability statements}
\label{section-2}

\subsection{Hamiltonian stability}
\label{sectionN2 - 2}

Fix a localization function $\lambda$.  We introduce harmonic and radial stability for the spherical operators $\ssA^\lambda$ and $\ssB^\lambda$ defined in \eqref{equa--11} and \eqref{equa--12}.  These notions are formulated entirely on the weighted spherical domain $(\Lambda,d\chi)$, where $\Lambda=\{\lambda>0\}$ and $d\chi=\lambda^{2\expoP}d\xh$ is the weighted measure defined in \eqref{equa--09}, for a fixed sufficiently large integer~$\expoP$.
We shall assume that $p\in(p_n^\flat,n-2)$, where
\be
p_n^\flat = \frac{(n-1)(n-3)}{2(n-2)} < \frac{n-2}{2} ,
\ee
as this ensures that various operator coefficients $a_{n,p}, b_{n,p}, c_{n,p}, d_{n,p}$ given in~\eqref{equa-our-parame-00} are positive.

The harmonic Hamiltonian operator~$\ssA^{\lambda}$ is self-adjoint when the parameter $c_{n,p}$~vanishes (which occurs in the limits $p\to n-2$ and $p\to p_n^\flat$), but is generally not, and we encode its coercivity through a quadratic functional obtained by a formal integration by parts. Specifically, we set
\be
\ssrmA^{\lambda}[\nu]\coloneqq \fint_{\Lambda} \nu\,\ssA{}^{\lambda}[\nu]\,d\chi
\ee
and, consequently,  
\bel{ssAalpha-quaform-0}
\aligned
\ssrmA^{\lambda}[\nu]
&= \fint_{\Lambda}\Bigl(
(n- 2)(\Deltaslash\nu)^2
+|\nablaslash^2\nu|^2
+2(1+a_{n,p})|\nablaslash\nu|^2
-c_{n,p}\,\nu\Deltaslash\nu
\Bigr)\,d\chi.
\endaligned
\ee
The functional $\ssrmA^{\lambda}$ serves as the basic energy controlling harmonic modes of the localized Hamiltonian operator. We also use the corresponding bilinear expression and retain the same notation for its diagonal quadratic functional:
\be
\ssrmA^\lambda[\nu,\mu]
\coloneqq \fint_\Lambda \mu\,\ssA^\lambda[\nu]\,d\chi,
\qquad
\ssrmA^\lambda[\nu]\equiv\ssrmA^\lambda[\nu,\nu].
\ee
Here and below, $\ker(\ssA^\lambda)$ is understood weakly in $H^2_{- \expoP}(\Lambda)$, that is, $\ssrmA^\lambda[\nu,\mu]=0$ for every $\mu\in H^2_{- \expoP}(\Lambda)$.

Besides harmonic coercivity, we introduce a radial stability condition which involves the kernel of $\ssA^\lambda$. In \eqref{equa-H211} and \eqref{equa-H-normalization}, below, the numerical constants $d^{\min}_{n,p}\leq d_{n,p} \leq d^{\max}_{n,p}$, as well as the normalization constant $\theta^\lambda$ are defined explicitly in \autoref{sec-appendix-A} below. The constant $\theta^\lambda$ selects a distinguished generator of the one-dimensional kernel, compatible with an ADM mass formula.

\begin{definition}
\label{art2 -def-easy-Hstab}
Fix $p<n- 2$ and consider the Hamiltonian operator $\ssA^\lambda$. A localization function $\lambda:\Lambda\to(0,\lambda_0]$ is said to satisfy
\bei
\item the Hamiltonian \underline{harmonic stability condition} provided that 
\bel{equa-H- 210}
\aligned
\ssrmA^{\lambda}[\nu]\gtrsim \|\nu\|_{H^2_{- \expoP}(\Lambda)}^2,
\qquad
\nu\in H^2_{- \expoP}(\Lambda)
\text{ with }
\la\nu\ra =0, 
\endaligned
\ee
\item the Hamiltonian \underline{radial stability condition}, provided that harmonic stability holds and, for every non-zero $\nu\in\ker(\ssA^{\lambda})$,
\bel{equa-H211}
\frac{\la\Deltaslash\nu\ra}{\la\nu\ra} < d^{\min}_{n,p}
\quad \text{ or } \quad  
\frac{\la\Deltaslash\nu\ra}{\la\nu\ra} > d^{\max}_{n,p}. 
\ee
\eei

\noindent Moreover, when both stability conditions hold, the Hamiltonian \underline{silhouette function} $\nu^{\normal}$ is the unique function on $\Lambda$ satisfying 
\bel{equa-H-normalization} 
\aligned
& \nu^{\normal} \in \ker(\ssA^{\lambda}),  
&
&& \la - \Deltaslash\nu^{\normal} + d_{n,p} \, \nu^{\normal} \ra 
&= \theta^{\lambda}. 
\endaligned
\ee
\end{definition}

Observe that $1\in\ker((\ssA^\lambda)^\ast)$, namely $\ssA^\lambda[\nu]$ has a vanishing average.
An argument based on applying the Banach--Nečas--Babuška theorem to the zero-average subspace shows that the kernel and cokernel are one-dimensional (cf.~\cite[Section~6.2]{LL-optimal-main}).
Concretely, harmonic stability makes $\nu\mapsto\la\nu\ra$ injective on $\ker(\ssA^\lambda)$, hence that space is at most one-dimensional.  Conversely, harmonic stability ensures that $\ssA^\lambda[\nu]=\ssA^\lambda[1]$ has a unique solution~$\nu$ with zero average, from which one learns that $1-\nu$ belongs to the kernel.  As a result, $\ssA^\lambda$~has a one-dimensional kernel whose elements have a non-zero average, and \eqref{equa-H211} is well-defined.

To complement these two stability conditions, the Hamiltonian \emph{shell stability condition}, which concerns the radial evolution of solutions (to the linearized Einstein constraints), will be introduced in \autoref{section-6}.


\subsection{Momentum stability}
\label{sectionN2 -3}

We write $\xi=(\xi^\perp,\xi^\parallel)$, where $\xi^\perp$ is a scalar function on $\Lambda$ and $\xi^\parallel$ is a vector field tangent to $\Sphe^{n-1}$ along $\Lambda$. We consider the non-self-adjoint spherical operator $\xi\mapsto\ssB^\lambda[\xi]$ (which is self-adjoint in the limit $p\to n-2$) and associate with it the quadratic functional
\bel{ssBalpha-quaform-0}
\aligned
\ssrmB^{\lambda}[\xi]
= \fint_{\Lambda} \hskip-.1cm \Bigl( (n-1) (\xiperp)^2 + \frac{1}{2} |\nablaslash\xiperp|^2
& - {a_{n,p}+2 \over 2} \xipar\cdot\nablaslash\xiperp + 2 \xiperp \nablaslash\cdot\xipar
\\
&  + {a_{n,p}+1\over 2}|\xipar|^2
+ \bigl|\Sym(\nablaslash\xipar) \bigr|^2 \Bigr) \, d\chi, 
\endaligned
\ee
where $\Sym(\nablaslash \xipar)_{ab} \coloneqq \frac{1}{2} \bigl( \nablaslash_a \xipar_b + \nablaslash_b \xipar_a\bigr)$. As above, we set
\be
\ssrmB^\lambda[\xi,\rho]
\coloneqq\fint_\Lambda \rho\cdot\ssB^\lambda[\xi]\,d\chi,
\qquad
\ssrmB^\lambda[\xi]\equiv\ssrmB^\lambda[\xi,\xi].
\ee
The kernel is understood weakly in $H^1_{- \expoP}(\Lambda,\RR^n)$, and \eqref{equa-stable-M-414} is a weighted Korn inequality for the measure~$d\chi$. 

\begin{definition}
\label{def-harmonic-Mstab}
Fix $p<n- 2$ and consider the momentum operator $\ssB^\lambda$.
A localization function $\lambda\colon\Lambda\to(0,\lambda_0]$ is said to satisfy the momentum \underline{harmonic stability condition} provided that
\bel{equa-stable-M-414}
\aligned
& \ssrmB^{\lambda}[\xi] \gtrsim 
\| \xiperp \|_{H^1_{- \expoP}(\Lambda)}^2
+ \| \xipar \|_{L^2_{- \expoP}(\Lambda)}^2
+ \| \Sym(\nablaslash \xipar) \|_{L^2_{- \expoP}(\Lambda)}^2,
\\
& \xi \in H^1_{- \expoP}(\Lambda)
\, \text{ with } \, 
 \bigl\la 2 \, \xh_l\, \xi^{\perp} + \xi_{l}^{\parallel} \bigr \ra  = 0.
\endaligned
\ee
\end{definition}

Observe that the $n$ fields 
\be
\xi_l^\ast \coloneqq (\xh_l,\nablaslash\xh_l),
\ee
lie in $\ker((\ssB^\lambda)^\ast)$; the corresponding Banach--Nečas--Babuška argument on the codimension-$n$ subspace in \eqref{equa-stable-M-414} shows that the kernel and cokernel have dimension~$n$ (cf.~\cite[Section~9.2]{LL-optimal-main}).

Assume henceforth that harmonic stability holds, and let $\{\xi^{(j)}\}_{j=1}^n$ be any basis of $\ker(\ssB^\lambda)$. Its associated \emph{structure matrix} is
\bel{equa-the-matrix}
(\Xi^\notreM)^{(j)}{}_k
 \coloneqq
 \bigl\la - \nablaslash_k \xi^{(j) \perp} + 2 a_{n,p} \xh_k\, \xi^{(j) \perp} \bigr \ra
 +(1+a_{n,p}) \la \xi^{(j) \parallel}{}_k \ra. 
 \ee

The normalization constant $\eta^\lambda$ in \eqref{equa-M-normalization} below merely selects a distinguished basis of the $n$-dimensional kernel compatible with an ADM momentum formula; it is defined explicitly in \autoref{sec-appendix-A}.

\begin{definition}
\label{def-radial-Mstab}
Fix $p<n- 2$ and consider the momentum operator $\ssB^\lambda$. A localization function $\lambda$ satisfies the momentum \underline{radial stability condition} provided that harmonic stability in \autoref{def-harmonic-Mstab} holds and that the structure matrix associated with one (equivalently, every) basis of $\ker(\ssB^\lambda)$ is invertible, namely    
\bel{equa-Xi-invertible}
\detbf (\Xi^\notreM) \neq 0. 
\ee 
When both stability conditions hold, the momentum \underline{silhouette vector fields} $\xi^{\normal(j)}$, $j=1,\ldots,n$, are the unique basis of the kernel whose associated structure matrix satisfies
\bel{equa-M-normalization}
\aligned
\Xi^\notreM &= \eta^{\lambda} I.
\endaligned
\ee 
\end{definition}

In addition, the momentum \emph{shell stability condition}, which concerns the radial evolution of solutions (to the linearized Einstein constraints), will be introduced later in \autoref{section-7}.


\subsection{Main results of this paper}
\label{sectionN2 -4}

\subsubsection{Definitions of the geometric constants}

To a localization domain $\Lambda\subset\Sphe^{n-1}$, we associate its diameter with respect to the geodesic distance of the round metric $\gslash$,
\bel{diamLambdadef}
\diam(\Lambda) = \sup_{x,y\in\Lambda} \dbf_{\gslash}(x,y) ,
\ee
Given a localization function $\lambda$ and the measure $d\chi=\lambda^{2\expoP}d\xh$, the \emph{weighted Poincaré constant} $\CPoin$ is the optimal constant for which
\bel{equa-Poin1}
\| \nu - \la\nu\ra \|_{L^2_{- \expoP}(\Lambda)}
\leq \CPoin\,\| \nablaslash\nu \|_{L^2_{- \expoP}(\Lambda)} 
\ee 
holds for every $\nu\in H^1_{- \expoP}(\Lambda)$. On the other hand, the \emph{second-order weighted Poincaré constant} $\CPoinTwo$ is the optimal constant for which
\bel{equa-Poin2}
\| \nablaslash\nu \|_{L^2_{- \expoP}(\Lambda)}^2
\leq (\CPoinTwo)^2\,\| \nablaslash^2\nu \|_{L^2_{- \expoP}(\Lambda)}^2 + \sum_{k,l} \matG_{kl} \la\nablaslash_k\nu\ra \la\nablaslash_l\nu\ra ,
\ee
holds for every $\nu\in H^2_{- \expoP}(\Lambda)$, where the (positive definite) matrix $\matG$ is defined by
\bel{equa-defineG}
\matG = \Id + \la\xh \otimes \xh\ra - 2 \, \langle\xh\rangle \otimes\langle\xh\rangle.
\ee
This inequality enters the proof of the Hamiltonian shell stability in Sections~\ref{section=6.3}--\ref{art2 -section=2.3}.
For instance, in a small domain, $\matG$~is close to a projector orthogonal to~$\la\xh\ra$, so that~\eqref{equa-Poin2} reduces essentially to a component-wise Poincaré inequality~\eqref{equa-Poin1} on the tangent vector~$\nablaslash\nu$.
Importantly, there is no constant in front of the average term in~\eqref{equa-Poin2}.


\subsubsection{Main stability theorem}

The theorem below gives sufficient geometric conditions for harmonic and radial stability of the spherical operators, together with shell stability for their radial extensions. The shell conditions are defined later in \autoref{def-shell-Hstab} for the Hamiltonian operator and in \autoref{def-shell-Mstab} for the momentum operator.

\begin{theorem}[Harmonic, radial, and shell stability under functional inequalities]
\label{art2 -thm:informal-suff-stab} 
Fix $p\in(p_n^\flat,n- 2)$ and consider the Hamiltonian operator $\ssA^\lambda$ and the momentum operator $\ssB^\lambda$ associated with a given localization function $\lambda$.

\bei

\item[1a.] There exists a universal constant $c_1^\notreH(n)$ (depending only on the dimension~$n$) such that,
if the Poincar\'e constant $\CPoin$ satisfies
\be
(n- 2 -p) \, \CPoin < c_1^\notreH(n),
\ee
then the Hamiltonian \underline{harmonic stability} and \underline{radial stability} conditions hold.

\item[1b.] There exists a universal constant $c_2^\notreH(n)$ such that, if
\be
\diam(\Lambda) + \CPoin + \CPoinTwo < c_2^\notreH(n),
\ee
then the Hamiltonian \underline{shell stability} condition holds. 

\item[2.] There exists a geometric constant $c^{\notreM \lambda}>0$, depending\footnote{More precisely, as we show in the proof, $c^{\notreM \lambda}>0$ depends explicitly on suitably weighted Poincaré, Korn, Korn--Poincaré, and Hardy constants, and on the moment matrices $S^2$ and $\matV$ associated with~$\lambda$.}  on the localization function and the dimension $n$ such that, if
\be
0<n- 2 -p< c^{\notreM \lambda},
\ee
the \underline{harmonic, radial, and shell momentum stability} conditions hold.
\eei
\end{theorem}

The Hamiltonian assertions follow from Theorems~\ref{art2 -prop-harm-rad}, \ref{art2 -prop--radial}, and~\ref{art2 -prop-dj39} stated in the following sections; the momentum assertion follows from Theorems~\ref{art2 -lem-39harmmom}, \ref{art2 -coro-radial-momen}, and~\ref{art2 -prop-semi-mdiss} below. More explicit bounds are established in those results, with auxiliary constants collected in \autoref{sec-appendix-A}.


\subsubsection{The almost-harmonic and small-domain regimes}

The next proposition summarizes several settings in which all stability conditions hold.  Item (i) below is a direct consequence of the main theorem, (ii)~is proven as \cite[Theorem~4.10]{LL-optimal-main}, and (iii)~arises from the Poincar\'e scaling results in \autoref{lem:Poincare-scaling-complete}.

\begin{proposition}[Sufficient conditions for harmonic, radial, and shell stability]
\label{prop:Poincare-scaling}
The following consequences of \autoref{art2 -thm:informal-suff-stab} hold.
\bei

\item[\textup{(i)}] 
For every fixed localization function $\lambda$, the Hamiltonian harmonic and radial stability conditions and all three momentum stability conditions hold when $p\in(p_n^\flat,n-2)$ is sufficiently close to $n-2$. If, in addition, the smallness condition in item~\textup{1b} of \autoref{art2 -thm:informal-suff-stab} holds, then the Hamiltonian shell stability condition also holds.

\item[\textup{(ii)}] In dimension $3\leq n\leq 17$ and for $p\in(p_n^\flat,n-2)$ sufficiently close to $n-2$, the localization function $\lambda\equiv 1$, with $\Lambda=\Sphe^{n-1}$, satisfies all six stability conditions.

\item[\textup{(iii)}] Let $\breve\Lambda$ be a bounded, connected domain with smooth boundary and compact closure in the unit ball of $\RR^{n-1}$, and let $\breve\lambda\in C^\infty(\overline{\breve\Lambda},[0,\lambda_0])$ be positive in $\breve\Lambda$ and vanish linearly on $\del\breve\Lambda$. For $\alpha\in(0,1]$, let $\Lambda_{(\alpha)}\subset\Sbb^{n-1}$ consist of the points $\xh$ such that $\xh_n>0$ and $(\xh_1,\dots,\xh_{n-1})/\alpha\in\breve\Lambda$, and set
\be
\lambda_{(\alpha)}(\xh)=\breve\lambda(\xh_1/\alpha,\dots,\xh_{n-1}/\alpha).
\ee
Then one has
\be
\diam(\Lambda_{(\alpha)})+C_\Poin^{\lambda_{(\alpha)}}+C_{\Poin 2}^{\lambda_{(\alpha)}}\lesssim\alpha,
\qquad \alpha \to 0.
\ee
Here, the implicit constant depends only on $n$, $\expoP$, and the reference localization function~$\breve\lambda$.   Consequently, we have the following conclusions: all three Hamiltonian stability conditions hold for sufficiently small $\alpha$, uniformly in $p\in(p_n^\flat,n- 2)$. Moreover, for a given~$\alpha$, all three momentum stability conditions hold whenever $p\in(p_n^\flat,n- 2)$ is sufficiently close to $n- 2$. In this joint small-domain and almost-harmonic regime, all six stability conditions hold.
\eei
\end{proposition}

For the estimates in item~(iii), the average term involving the matrix $\matG$ in \eqref{equa-Poin2} is essential for the above small-domain statements;  see Sections~\ref{sectionN3- 2} and~\ref{sectionN3- 2 -complete}. For comparison, when $\Lambda =\Sphe^{n-1}$ and $\lambda\equiv 1$, the present normalization gives (with obvious notation)
\bel{equa-round-sphere-constants}
\diam(\Sphe^{n-1})=\pi,
\quad C_\Poin(\Sphe^{n-1}, \lambda \equiv 1)=\frac{1}{\sqrt{n-1}}, 
\quad C_{\Poin 2}(\Sphe^{n-1},\lambda\equiv 1) = \frac{1}{\sqrt{n+2}}.
\ee
The proof proceeds by decomposing functions into spherical harmonics and applying the Bochner identity (given by an integration by parts) to decouple the modes.
The standard Poincaré inequality is saturated by degree-one modes.
For the second-order estimate, the degree-one modes obey the inequality (with a constant $1/n<1/\sqrt{n+2}$) thanks to the matrix~$\matG$, while degree-two modes saturate the inequality.


\section{Harmonic and radial stability for the localized Hamiltonian operator}
\label{section-3}

\subsection{Harmonic stability under a Poincaré inequality}
\label{sectionN4-1}

\subsubsection{Main statement of harmonic stability}

We now study the Hamiltonian operator $\ssA^\lambda$ and derive conditions on the weighted Poincar\'e constant that imply harmonic and radial stability. We impose, step by step, increasingly restrictive geometric conditions on the localization function $\lambda$. Throughout this section, $p\in(p_n^\flat,n- 2)$, so in particular $c_{n,p}>0$. The following statement involves the constant $\alpha_{n,p}$ defined explicitly in~\eqref{equa-const-alpha}.

\begin{theorem}[Harmonic stability of the Hamiltonian operator $\ssA^\lambda$]
\label{art2 -prop-harm-rad}
Fix an exponent $p\in(p_n^\flat,n- 2)$ and a localization function $\lambda$, with $d\chi=\lambda^{2\expoP}d\xh$. If the weighted Poincar\'e constant satisfies
\bel{art2 -equa-4178}
\CPoin < \alpha_{n,p}, 
\ee
then the \underline{Hamiltonian harmonic stability} \eqref{equa-H- 210} holds. Moreover, the upper bound in~\eqref{art2 -equa-4178} tends to $+\infty$ in the harmonic limit $p\to n-2$, so that \eqref{art2 -equa-4178} holds for any localization function provided $p$ is sufficiently close to $n- 2$. Alternatively, uniformly in the exponent~$p$, it holds on any family of localization pairs for which $C_\Poin^\lambda$ is sufficiently small, in particular on the small domains constructed in \autoref{sectionN3- 2 -complete}, below.
\end{theorem}


\subsubsection{Proof of harmonic stability}

Before proving \autoref{art2 -prop-harm-rad}, we need a technical estimate on the harmonic functional~$\ssrmA^\lambda$.

\begin{proposition}[Semi-coercivity of the functional $\ssrmA^{\lambda}$ on $\la\nut\ra =0$]
\label{art2 -lem-ssrmAcoer}
With the notation in \autoref{art2 -prop-harm-rad}, assume that the Poincar\'e constant obeys
\be
\rho(\gamma)
\coloneqq
2(1+a_{n,p})(\CPoin)^{- 2}
-
\frac{c_{n,p}^2}{4\gamma}
>0
\ee
for some constant $\gamma>0$. Then, for any function $\nut$ with $\la\nut\ra =0$ and any $0\le \rho' < \rho(\gamma)$, one has the
semi-coercivity bound
\be
\aligned
& \ssrmA^{\lambda}[\nut,\nut]
- \Bigl(\frac{n^2 -3n+3}{n-1}- \gamma\Bigr)
\|\Deltaslash \nut\|_{L^2_{- \expoP}(\Lambda)}^2
- \|(\nablaslash^2\nut)^\circ\|_{L^2_{- \expoP}(\Lambda)}^2
- \rho'\|\nut\|_{L^2_{- \expoP}(\Lambda)}^2
\\
&\ge
\frac{(\rho(\gamma)- \rho')}{1+(\CPoin)^{- 2}}
\|\nut\|_{H^1_{- \expoP}(\Lambda)}^2 .
\endaligned
\ee
\end{proposition}

\begin{proof}[Proof of \autoref{art2 -lem-ssrmAcoer}]
Starting from the quadratic functional~\eqref{ssAalpha-quaform-0}, we decompose the Hessian into its trace and traceless parts, and we complete the square in the $\Deltaslash\nut$--$\nut$ coupling term.
This yields
\bel{art2 -ssrmAcoer-ineq}
\aligned
& \ssrmA^{\lambda}[\nut,\nut]
-
\Bigl(\frac{n^2 -3n+3}{n-1}- \gamma\Bigr)\fint_{\Lambda}(\Deltaslash\nut)^2\,d\chi
-
\fint_{\Lambda}\bigl|(\nablaslash^2\nut)^\circ\bigr|^2\,d\chi
-
\rho'\fint_{\Lambda}\nut^2\,d\chi
\\
&=
\gamma \fint_{\Lambda}\Bigl(\Deltaslash\nut- \frac{c_{n,p}}{2\gamma}\nut\Bigr)^2\,d\chi
+
\fint_{\Lambda}\Bigl(
2(1+a_{n,p})|\nablaslash\nut|^2
-
\Bigl(\frac{c_{n,p}^2}{4\gamma}+\rho'\Bigr)\nut^2
\Bigr)\,d\chi.
\endaligned
\ee
The first integral on the right-hand side is non-negative. For the second one, we use the weighted Poincar\'e inequality (under the condition $\la\nut\ra =0$) to estimate
\be
\fint_{\Lambda}\nut^2\,d\chi
\le (\CPoin)^2 \fint_{\Lambda}|\nablaslash\nut|^2\,d\chi,
\ee
and therefore the second integral is coercive as soon as $\frac{c_{n,p}^2}{4\gamma}+\rho' < 2(1+a_{n,p})(\CPoin)^{- 2}$, that is, $\rho'<\rho(\gamma)$. Combining this with the definition of the weighted $H^1_{- \expoP}(\Lambda)$ norm and the Poincar\'e estimate above gives the claimed bound with the explicit factor $\frac{1}{1+(\CPoin)^{- 2}}$.
\end{proof}

\begin{proof}[Proof of \autoref{art2 -prop-harm-rad}]
Harmonic stability requires coercivity of the quadratic form $\ssrmA^\lambda[\nu]$ for functions $\nu$ with vanishing average $\la\nu\ra =0$, at the $H^2_{- \expoP}(\Lambda)$ level. We apply \autoref{art2 -lem-ssrmAcoer} with a parameter $\gamma$ chosen so that
\be
0<\gamma<\frac{n^2 -3n+3}{n-1}, \qquad \quad \rho(\gamma)>0.
\ee
Under~\eqref{art2 -equa-4178} such a choice is possible, and in fact one may take $\gamma$ arbitrarily close to $\frac{n^2 -3n+3}{n-1}$ while
preserving $\rho(\gamma)>0$.

Since both coefficients subtracted in \autoref{art2 -lem-ssrmAcoer} are positive, that proposition shows that $\ssrmA^\lambda[\nu,\nu]$ controls $\|\Deltaslash\nu\|_{L^2_{- \expoP}(\Lambda)}^2$, $\|(\nablaslash^2\nu)^\circ\|_{L^2_{- \expoP}(\Lambda)}^2$, and $\|\nu\|_{H^1_{- \expoP}(\Lambda)}^2$.  This yields the desired coercivity estimate in $H^2_{- \expoP}(\Lambda)$, namely the harmonic stability condition~\eqref{equa-H- 210}.
\end{proof}


\subsection{Radial stability under a Poincaré inequality}
\label{sectionN4- 2}

\subsubsection{Main statement of radial stability}

We now use the constants $d^{\min}_{n,p}$ and~$\beta_{n,p}$ defined explicitly in \eqref{equa-H2110} and~\eqref{equa-cons-beta}.

\begin{theorem}[Radial stability of the Hamiltonian operator $\ssA^\lambda$]
\label{art2 -prop--radial} 
Fix an exponent $p\in(p_n^\flat,n- 2)$ and a localization function $\lambda$. If the weighted Poincar\'e constant satisfies
\bel{art2 -art2 -equa-4178-bis}
\CPoin < \beta_{n,p}, 
\ee
then harmonic stability holds and, for every nonzero $\nu\in\ker(\ssA^\lambda)$, one has\footnote{The lower bound is necessarily non-strict: when $\lambda\equiv 1$, every nonzero constant function belongs to $\ker(\ssA^1)$, and the quotient $\la\Deltaslash\nu\ra/\la\nu\ra$ vanishes.}
\bel{art2 -equa-kss27}
0 \leq \frac{\la\Deltaslash\nu\ra}{\la\nu\ra}
< d^{\min}_{n,p} .
\ee
Consequently, the \underline{Hamiltonian radial stability} condition \eqref{equa-H211} holds. Moreover, the upper bound in~\eqref{art2 -art2 -equa-4178-bis} tends to $+\infty$ in the harmonic limit $p\to n- 2$, so that \eqref{art2 -art2 -equa-4178-bis} holds for any localization provided $p$ is sufficiently close to $n- 2$. Alternatively, it holds uniformly in~$p$ on any family of localization pairs for which $C_\Poin^\lambda$ is sufficiently small, in particular on the small domains constructed in \autoref{sectionN3- 2 -complete}. 
\end{theorem}

\subsubsection{Proof of radial stability}

Since $(n^2 -3n+3)/(n-1)>1$, the definitions of $\alpha_{n,p}$ and $\beta_{n,p}$ give
\be
\beta_{n,p}
\leq \frac{2\sqrt{2(1+a_{n,p})}}{c_{n,p}}
< \sqrt{\frac{8(n^2 -3n+3)}{n-1}}\,
\frac{\sqrt{1+a_{n,p}}}{c_{n,p}}
=\alpha_{n,p}.
\ee
Hence \autoref{art2 -prop-harm-rad} gives harmonic stability. Let $0\neq\nu\in\ker(\ssA^\lambda)$ and set $\nut=\nu- \la\nu\ra$. By harmonic stability, the average map is injective on $\ker(\ssA^\lambda)$; hence $\la\nu\ra\neq0$. The fluctuation $\nut$ obeys
\bel{art2 -nut-438}
\aligned 
\la\nut\ra = 0 , \qquad
\ssA^{\lambda}[\nut] = c_{n,p} \la\nu\ra \lambda^{- 2\expoP} \Deltaslash(\lambda^{2\expoP}).
\endaligned
\ee 
Multiplying the latter equation by $\lambda^{2\expoP}\nut$ and integrating by parts, we obtain the equation 
\bel{art2 -equa-428}
\aligned
\ssrmA^{\lambda}[\nut,\nut]
& = c_{n,p} \la\nu\ra \la\Deltaslash\nut\ra .
\endaligned
\ee
On the other hand, the condition on $\CPoin$ allows us to pick $\gamma$ so that
\bel{art2 -equa-choicegamma}
\frac{c_{n,p}^2 (\CPoin)^2}{8(1+a_{n,p})} < \gamma < \min\Bigl(1, \frac{1+a_{n,p}}{b_{n,p}} \Bigr) ,
\ee
which is in particular bounded above by $(n^2 -3n+3)/(n-1)>1$. { Applying \autoref{art2 -lem-ssrmAcoer} with $\rho'=0$, and then using Jensen's inequality, yields}
\be
\ssrmA^{\lambda}[\nut,\nut] 
\geq \Bigl( \frac{n^2 -3n+3}{n-1} - \gamma \Bigr)  \|\Deltaslash \nut \|^2_{L^2_{- \expoP}(\Lambda)}
\geq \Bigl( \frac{n^2 -3n+3}{n-1} - \gamma \Bigr) \la\Deltaslash\nut\ra^2. 
\ee

\bse
Therefore, in combination with~\eqref{art2 -equa-428} we find  
\be
c_{n,p} \la\nu\ra \la\Deltaslash\nut\ra 
\geq \Bigl( \frac{n^2 -3n+3}{n-1} - \gamma \Bigr) \la\Deltaslash\nut\ra^2. 
\ee
{ Since $\la\nu\ra\neq0$, we can divide by $\la\nu\ra^2>0$ and obtain}
\be
\frac{\la\Deltaslash\nut\ra}{\la\nu\ra} \biggl( c_{n,p}
- \Bigl( \frac{n^2 -3n+3}{n-1} - \gamma \Bigr) \frac{\la\Deltaslash\nut\ra}{\la\nu\ra} \biggr) \geq 0 .
\ee
Since $c_{n,p}>0$ and the coefficient multiplying the squared ratio is positive, this quadratic constraint implies 
\bel{art2 -ratiodnutnun}
\aligned
0 & \leq \frac{\la\Deltaslash\nut\ra}{\la\nu\ra}
  \leq c_{n,p} \Bigl( \frac{n^2 -3n+3}{n-1} - \gamma \Bigr)^{-1} \\
& < c_{n,p} \max\Bigl( \frac{n^2 -3n+3}{n-1} - 1 ,  \frac{n^2 -3n+3}{n-1} - \frac{1+a_{n,p}}{b_{n,p}} \Bigr)^{-1} 
\\
& = c_{n,p} \max\Bigl( \frac{(n- 2)^2}{n-1} ,  \frac{c_{n,p}}{d_{n,p}} \Bigr)^{-1} ,
\endaligned
\ee
Since $\la\Deltaslash\nu\ra =\la\Deltaslash\nut\ra$, this is precisely \eqref{art2 -equa-kss27} and proves radial stability. The value zero is allowed; it occurs, for example, for the constant kernel element when $\lambda\equiv 1$.
\ese
This completes the proof of \autoref{art2 -prop--radial}.


\subsection{Estimates for the Hamiltonian silhouette}
\label{sectionN4-3}

{ It is only at this stage, thanks to \autoref{art2 -prop--radial}, that we know, for every nonzero kernel element $\nu$, that
\be
\frac{\la- \Deltaslash\nu+d_{n,p}\nu\ra}{\la\nu\ra}
=d_{n,p}- \frac{\la\Deltaslash\nu\ra}{\la\nu\ra}>0.
\ee
The normalization functional is therefore nonzero, and we may define the silhouette function $\nu^\normal$ by the normalization \eqref{equa-H-normalization}, which is motivated by the study of the ADM mass in~\cite{LL-optimal-main}.}

Moreover, certain Sobolev norms of the silhouette function $\nu^{\normal} \in \ker(\ssA^{\lambda})$ can be controlled, as follows. Observe that the $L^2$ norm of $\nut^\normal$ could be made arbitrary small by taking the Poincaré constant sufficiently small, but the $\dot{H}^1$ and $\dot{H}^2$ norms are only \emph{bounded}. Along the way we introduce a pair of auxiliary constants $0<K_{2\Poin}<K_{1\Poin}<1$ in \eqref{art2 -k1poin-def} and~\eqref{art2 -k2poin-def} that depend on $n$, $p$ and the Poincaré constant $\CPoin$, and which remain bounded away from zero in either of the two limits $\CPoin\to 0$ and $p\to n- 2$.

\begin{proposition}[Fluctuations of the Hamiltonian silhouette function]
\label{art2 -lem-nut-small}
Provided the weighted Poincaré constant $\CPoin$ satisfies the radial-stability condition \eqref{art2 -art2 -equa-4178-bis}, the fluctuation $\nut^\normal$ of the Hamiltonian silhouette function satisfies the following integral estimates\footnote{According to our notation~\eqref{equa-norm-poids}, squared norms contain the factor $1/\aire[\Lambda,\lambda]$.}:
\bel{art2 -equa-4206} 
\aligned  
\|\Deltaslash \nut^\normal \|_{L^2_{- \expoP}(\Lambda)}
& \leq \frac{(n-1)c_{n,p}}{(n^2 -3n+3)K_{1\Poin}} \, |\la\nu^\normal\ra|,
& \qquad
\bigl\|\nablaslash^2\nut^\normal,\nablaslash\nut^\normal\bigr\|_{L^2_{- \expoP}(\Lambda)}
& \lesssim \frac{c_{n,p}}{K_{1\Poin}} |\la\nu^\normal\ra|,
\endaligned
\ee 
with $\tcnp$ defined in \eqref{equa-cntilde}, and with the short-hand notation  
\bel{art2 -k1poin-def}
K_{1\Poin} \coloneqq 1 - (\tcnp)^2 (\CPoin)^2.
\ee
\end{proposition} 

\begin{proof}
{\bf 1. \it Average Laplacian.} The proof of \autoref{art2 -prop--radial}, and especially~\eqref{art2 -ratiodnutnun}, shows that
\be
{ 0 \leq} \frac{\la\Deltaslash\nut^\normal\ra}{\la\nu^\normal\ra}
\leq c_{n,p} \Bigl( \frac{n^2 -3n+3}{n-1} - \gamma \Bigr)^{-1}
\ee
for any $\gamma$ in the range~\eqref{art2 -equa-choicegamma}.  By taking the limit where $\gamma>c_{n,p}^2 (\CPoin)^2 / (8(1+a_{n,p}))$ tends to its lower bound, this implies
\be
{ 0 \leq} \frac{\la\Deltaslash\nut^\normal\ra}{\la\nu^\normal\ra}
\leq \frac{(n-1)c_{n,p}}{(n^2 -3n+3)K_{1\Poin}} ,
\ee
which provides a quantitative estimate for $\la\Deltaslash\nut^\normal\ra/\la\nu^\normal\ra$.

\medskip

\bse
\noindent{\bf 2. \it Hessian estimate.} \autoref{art2 -lem-ssrmAcoer} shows 
\be
\aligned
& \|(\nablaslash^2 \nut^\normal)^\circ\|^2_{L^2_{- \expoP}(\Lambda)} + \Bigl( \frac{n^2 -3n+3}{n-1} - \gamma \Bigr)  \|\Deltaslash \nut^\normal \|^2_{L^2_{- \expoP}(\Lambda)}
\\
& \leq \ssrmA^{\lambda}[\nut^\normal,\nut^\normal]
= c_{n,p} \la\nu^\normal\ra \la\Deltaslash\nut^\normal\ra
\leq \frac{n-1}{n^2 -3n+3} \frac{(c_{n,p})^2}{K_{1\Poin}} \la\nu^\normal\ra^2
\endaligned
\ee
for all $\gamma>c_{n,p}^2 (\CPoin)^2 / (8(1+a_{n,p}))$. This immediately bounds the traceless Hessian as stated. Taking the limit where $\gamma$ tends to its lower bound yields the desired bound on $\|\Deltaslash \nut^\normal \|^2_{L^2_{- \expoP}(\Lambda)}$. In view of Jensen's inequality $\bigl|\la\Deltaslash\nut^\normal\ra\bigr| \leq \|\Deltaslash \nut^\normal \|_{L^2_{- \expoP}(\Lambda)}$, we can check that our result is consistent with the previous estimate on the average.

\medskip

\noindent{\bf 3. \it $L^2$ estimate.} Next, let us pick $\rho'$ so that
$0 < \rho' < 2 (1+a_{n,p}) \bigl( (\CPoin)^{- 2} - (\tcnp)^2 \bigr)$, 
and let us take $\gamma$ so that
\be
\frac{c_{n,p}^2/4}{2 (1+a_{n,p}) (\CPoin)^{- 2} - \rho'} < \gamma < \frac{n^2 -3n+3}{n-1} .
\ee
\autoref{art2 -lem-ssrmAcoer} applies and gives
\be
\rho' \|\nut^\normal\|_{L^2_{- \expoP}(\Lambda)}^2
\leq \ssrmA^\lambda[\nut^\normal,\nut^\normal]
\leq 8(1+a_{n,p}) (\tcnp)^2 (K_{1\Poin})^{-1} \la\nu^\normal\ra^2 
= \frac{n-1}{n^2 -3n+3}\,\frac{(c_{n,p})^2}{K_{1\Poin}}\,\la\nu^\normal\ra^2. 
\ee 
By taking the limit where $\rho'$ tends to its upper bound, we obtain 
\be
\|\nut^\normal\|_{L^2_{- \expoP}(\Lambda)}
\leq \frac{2\tcnp \CPoin}{1 - (\tcnp)^2 (\CPoin)^2} \bigl|\la\nu^\normal\ra\bigr| .
\ee
\ese

\medskip

\noindent{\bf 4. \it Gradient estimate.} Finally, we control $\|\nablaslash\nut^\normal\|_{L^2_{- \expoP}(\Lambda)}$ by returning to \eqref{art2 -ssrmAcoer-ineq}, for $\gamma =(n^2 -3n+3)/(n-1)$ and $\rho'=0$, and bounding the $\nut^2$ term by $|\nablaslash\nut|^2$ using the Poincaré inequality. This yields
\be
2 (1+a_{n,p}) K_{1\Poin} \|\nablaslash\nut^\normal\|_{L^2_{- \expoP}(\Lambda)}^2
\leq \ssrmA^\lambda[\nut^\normal,\nut^\normal]
\leq \frac{n-1}{n^2 -3n+3} \frac{(c_{n,p})^2}{K_{1\Poin}} \la\nu^\normal\ra^2 .
\ee
Incidentally, this bound on $\nablaslash\nut^\normal$ immediately gives a bound on $\|\nut^\normal\|_{L^2_{- \expoP}(\Lambda)}$ via the Poincaré inequality.
\end{proof}

We thus arrive at estimates for averages $\la\nu^\normal\ra$ and $\la\Deltaslash\nut^\normal\ra$. 

\begin{proposition}[Averages of the Hamiltonian silhouette function]
\label{art2 -cor-lem-nut-small}
In the setup of \autoref{art2 -lem-nut-small}, condition \eqref{art2 -art2 -equa-4178-bis} implies $0<K_{2\Poin}<K_{1\Poin}<1$, where
\bel{art2 -k2poin-def}
K_{2\Poin} = \frac{(n-1)(1+a_{n,p})}{(n^2 -3n+3)b_{n,p}} - (\tcnp)^2 (\CPoin)^2
\ee
Then one has
\be 
\aligned
0 & \leq \la\Deltaslash\nut^\normal\ra \leq \Bigl( \frac{K_{1\Poin}}{K_{2\Poin}} - 1 \Bigr) \theta^\lambda , \qquad
&&& \frac{\theta^\lambda}{d_{n,p}} & \leq \la\nu^\normal\ra \leq \frac{K_{1\Poin}}{K_{2\Poin}} \frac{\theta^\lambda}{d_{n,p}} .
\endaligned
\ee
\end{proposition}

\begin{proof}
We recall that $K_{1\Poin} = 1 - (\tcnp)^2 (\CPoin)^2$.
A short calculation shows that
\bse
\be
K_{2\Poin} = \frac{(n-1)(1+a_{n,p})}{(n^2 -3n+3)b_{n,p}} - (\tcnp)^2 (\CPoin)^2
= K_{1\Poin} - \frac{(n-1)c_{n,p}}{(n^2 -3n+3)d_{n,p}}
< K_{1\Poin} < 1
\ee
since $c_{n,p},d_{n,p}>0$ in the exponent range of interest. Moreover, \eqref{art2 -art2 -equa-4178-bis} and the second term in the minimum defining $\beta_{n,p}$ imply
\be
(\CPoin)^2
<\frac{8(1+a_{n,p})^2}{b_{n,p}c_{n,p}^2},
\ee
which is equivalent to $K_{2\Poin}>0$.

Next, the normalization of the silhouette function $\theta^\lambda = d_{n,p}\la\nu^\normal\ra - \la\Deltaslash\nu^\normal\ra$ together with
the upper bound~\eqref{art2 -equa-4206} and { the non-negativity $\la\Deltaslash\nut^\normal\ra/\la\nu^\normal\ra\geq0$ imply}
\bel{art2 -zetadnukk}
\frac{\theta^\lambda}{d_{n,p}\la\nu^\normal\ra}
= 1 - \frac{\la\Deltaslash\nut^\normal\ra}{d_{n,p}\la\nu^\normal\ra}
\geq 1 - \frac{(n-1)c_{n,p}}{(n^2 -3n+3) d_{n,p}K_{1\Poin}} = \frac{K_{2\Poin}}{K_{1\Poin}} .
\ee
\ese
This lower bound is positive because $K_{2\Poin}>0$; hence $\la\nu^\normal\ra>0$, and inverting the inequality gives the stated upper bound on~$\la\nu^\normal\ra$. The bounds $0\leq\la\Deltaslash\nut^\normal\ra/\la\nu^\normal\ra\leq d_{n,p}(1-K_{2\Poin}/K_{1\Poin})$ then translate into the stated bounds on $\la\Deltaslash\nut^\normal\ra$. Finally, the lower bound on $\la\nu^\normal\ra$ follows from the first equality in~\eqref{art2 -zetadnukk}, the non-negativity of $\la\Deltaslash\nut^\normal\ra/\la\nu^\normal\ra$, and $\la\nu^\normal\ra>0$. 
\end{proof}


\section{Harmonic and radial stability for the localized momentum operator}
\label{section-4} 
 
\subsection{Harmonic stability under a Korn-Poincaré inequality}
\label{art2 -section=3.1}

\subsubsection{Structure matrices}

We now turn our attention to the harmonic momentum operator $\ssB^\lambda$ defined on the weighted spherical domain $(\Lambda,d\chi)$, and geometric inequalities for the quadratic functional~$\ssrmB^\lambda$ given in~\eqref{ssBalpha-quaform-0}, above.
Before proceeding, we introduce several objects that are necessary to deal with vectors fields.

The \emph{structure matrices}\footnote{The matrix $\matV$ is introduced here for convenience but is only used later, in \autoref{section-5}.} are defined from weighted integrals of the coordinates $\xh_k = x_k/r$,
\bel{equa-matrix-STU}
\aligned
S^2 & \coloneqq (\la\xh_k \xh_l \ra) , \quad
T \coloneqq (\delta_{kl} + \la \xh_k \xh_l \ra),
\quad 
U \coloneqq  (\delta_{kl} - \la\xh_k \xh_l \ra ) ,
\quad
\matV \coloneqq (\la\xh_k \xh_l\ra - \la\xh_k\ra\la\xh_l\ra) .
\endaligned
\ee
They are constant, positive-definite matrices.  (In particular $S^2$ has a unique positive-definite square root~$S$.) In addition, observe that 
\be
U<T
\ee
in the sense that the difference $(T-U)$ is positive-definite, and thus $U^{-1} > T^{-1}$ since the difference
\be
U^{-1}-T^{-1} = T^{-1}(T-U)T^{-1} + T^{-1}(T-U)U^{-1}(T-U)T^{-1} > 0
\ee
is the sum of two positive-definite matrices.

For a vector $v\in\RR^n$ and matrix~$M$ such as $T^{-1}$ or~$U$, we denote
\bel{vM2}
|v|_M^2 \coloneqq v^k M_{kl} v^l .
\ee


\subsubsection{Decomposition into average and fluctuation}

It will also be necessary to decompose vector fields into certain averages and fluctuations.
The standard basis vectors of $\RR^n$ seen as vector fields on~$\RR^n$ restricted to the sphere~$\Sphe^{n-1}$ are (for $k=1,\dots,n$)
\be
\xistar_k = (\xistarperp_k,\xistarpar_k) = (\xh_k, \nablaslash \xh_k) ,
\quad\text{with components}\quad
(\xistarpar_k)_l = \delta_{kl} - \xh_k \xh_l .
\ee
We introduce, for a vector $\xi=(\xiperp,\xipar)$, its \emph{fluctuations}~$\xi^{\fluc}$ and \emph{averages} $\projP^k(\xi)$,
\bel{equa-tildemoment}
\xi^{\fluc} = \xi - \projP^k(\xi) \xistar_k , \qquad
\projP^k(\xi) = (T^{-1})^{kl} \, \bigl\la 2 \xh_l \xiperp + \xipar{}_l \bigr\ra .
\ee
In particular, given the components of~$\xistar_m$,
\be
\projP^k(\xistar_m) = (T^{-1})^{kl} \la 2 \xh_l \xh_m + \delta_{lm} - \xh_l \xh_m \ra = \delta^k_m ,
\qquad
(\xistar_m)^{\fluc} = 0 .
\ee

As an illustration, of how to use these concepts, $L^2$-type squared norms split into contributions from these averages and fluctuations.

\begin{lemma}[Average-fluctuation identity]
  For any vector field $\xi=(\xiperp,\xipar)$ one has the following identity:
  \be
  \fint_{\Lambda} \bigl( 2 |\xiperp|^2 + |\xipar|^2 \bigr) d\chi
  = \bigl| \la 2\xh\xiperp + \xipar\ra \bigr|_{T^{-1}}^2 + \fint_{\Lambda} \bigl( 2 |\xi^{\fluc\perp}|^2 + |\xi^{\fluc\parallel}|^2 \bigr) d\chi .
  \ee
\end{lemma}

\begin{proof} 
Using $\xi=\xi^\fluc + \projP^k(\xi) \xistar_k$, and the explicit components $\xistarperp_k=\xh_k$ and $\xistarpar_{k\ l}=\delta_{kl}- \xh_k\xh_l$, we find
\be
\aligned
\fint_{\Lambda} \bigl( 2 |\xiperp|^2 + |\xipar|^2 \bigr) d\chi
& = \fint_{\Lambda} \bigl( 2 |\xi^{\fluc\perp}|^2 + |\xi^{\fluc\parallel}|^2 \bigr) d\chi
+ 2 \projP^k(\xi) \fint_{\Lambda} \bigl( 2 \xi^{\fluc\perp}\xistarperp_k + \xi^{\fluc\parallel}\cdot\xistarpar_k \bigr) d\chi \\
& \quad + \projP^k(\xi) \projP^l(\xi) \fint_{\Lambda} \bigl( 2 \xistarperp_k \xistarperp_l + \xistarpar_k \cdot \xistarpar_l \bigr) d\chi
\\
& = \fint_{\Lambda} \bigl( 2 |\xi^{\fluc\perp}|^2 + |\xi^{\fluc\parallel}|^2 \bigr) d\chi
+ \projP^k(\xi) \projP^l(\xi) T_{kl} ,
\endaligned
\ee
where the cross-term vanishes due to $\la 2\xh_k\xi^{\fluc\perp}+\xi^{\fluc\parallel}{}_k\ra =0$. This establishes the identity.
\end{proof}


\subsubsection{A weighted Korn inequality}

An important ingredient in proving coercivity of $\ssrmB^\lambda$ is the Korn inequality, which states that the quantity $\fint_{\Lambda}\bigl|\Sym(\nablaslash \xipar)\bigr|^2\,d\chi$ controls $\xipar$ modulo the Killing fields of the $(n-1)$-sphere. A basis of the latter is given by \emph{rotations} (on the sphere)
\be
\zeta =\xh_l\nablaslash\xh_m- \xh_m\nablaslash\xh_l ,
\qquad \text{with} \qquad
\zeta_k=\delta_{km}\xh_l- \delta_{kl}\xh_m .
\ee
General Killing fields take the following form for some constant anti-symmetric matrix~$L^{lm}$,
\bse
\bel{art2 -equa-k}
\zeta = L^{lm}(\xh_l\nablaslash\xh_m- \xh_m\nablaslash\xh_l) .
\ee
They enjoy the properties 
\bel{art2 -equa-killing}
\zeta_k = - 2L^{kl}\xh_l ,
\qquad
\Sym(\nablaslash\zeta)=0 .
\ee
\ese
These explicit expressions will be useful later on. At this stage, we simply point out the inequality \eqref{art2 -equa--weighted-Korn}, below, proving that these fields are the main obstacle  to reaching coercivity statements. Note in passing that in the limit of small domain size, we can achieve arbitrarily small values of $C_{0\Korn}^\lambda$ while keeping $C_{1\Korn}^\lambda$~bounded.
(See \autoref{sectionN3- 2}.)

\begin{proposition}[A weighted Korn inequality]
\label{art2 -equa--weighted-Korn}
In any localization domain $(\Lambda, d\chi = \lambda^{2 \expoP} d\xh)$, the following Korn inequality holds for a pair of positive constants $(C_{1\Korn}^\lambda , C_{0\Korn}^\lambda)$ and for any vector field $\xi^\parallel$ tangent to~$\Lambda$: 
\bel{art2 -equa--1133-copie}
\aligned
& \fint_{\Lambda}   \bigl|\Sym(\nablaslash \xi^{\parallel} )\bigr|^2 \, d\chi 
\geq \min_{\zeta} \fint_{\Lambda} 
\Bigl(
{1 \over (C_{1\Korn}^\lambda)^2}  | \nablaslash(\xi^{\parallel} - \zeta) |^2  
+ {1 \over (C_{0\Korn}^\lambda)^2} | \xi^{\parallel} - \zeta |^2 
\Bigr) \, d\chi, 
\endaligned
\ee
where the minimum is taken over all Killing fields~$\zeta$ in~\eqref{art2 -equa-k}. 
\end{proposition}


\subsubsection{A weighted Korn-Poincaré inequality} 

While \autoref{art2 -equa--weighted-Korn} above applies to a \emph{tangent} vector field, our next inequality \emph{couples} a scalar field $\xi^\perp$ and a vector field $\xi^\parallel$. It states the coercivity of the harmonic functional~$\ssrmB^{\lambda}$ for $a_{n,p}=0$ under the vanishing-average condition $\bigl\la 2 \, \xh_l\, \xi^{\perp} + \xi_{l}^{\parallel} \bigr\ra  = 0$.

\begin{proposition}[A weighted Korn-Poincaré inequality]
\label{art2 -lem-poinckorn}
For any localization function~$\lambda$, there exist constants $\CzeroKP,\ConeKP>0$ such that for any scalar-vector field $\xi=(\xi^\perp, \xi^\parallel)$,
\bel{art2 -equa--1133}
\aligned
\bigl.\ssrmB^{\lambda}[\xi]\bigr|_{a_{n,p}=0} 
& = \fint_{\Lambda} \Bigl(
 (n-1) \Bigl(\xi^{\perp} + \frac{1}{n-1} \nablaslash\cdot \xi^{\parallel}\Bigr)^2
+ \frac{1}{2} \Bigl|\nablaslash \xi^{\perp} -  \xi^{\parallel} \Bigr|^2
  + \bigl|\Sym(\nablaslash \xi^{\parallel} )^\circ\bigr|^2
\Bigr) \, d\chi 
\\
& \geq
\frac{1}{(\ConeKP)^2}
\bigl\|\nablaslash\xiperp,\nablaslash\xipar\bigr\|_{L^2_{- \expoP}(\Lambda)}^2
+ \frac{1}{(\CzeroKP)^2}
\bigl\|\xiperp,\xipar\bigr\|_{L^2_{- \expoP}(\Lambda)}^2
\\[-.2ex]
& \text{provided }
\bigl\la 2 \, \xh_l\, \xiperp + \xipar{}_l \bigr\ra  = 0 .
\endaligned
\ee
\end{proposition}

\begin{proof}
\bse
We proceed by contradiction. If \eqref{art2 -equa--1133} did not hold, we could find two sequences $\xiperp_\eps$ and $\xipar_\eps$ such that
$\bigl\la 2 \, \xh_l\, \xiperp_\eps + \xipar_{\eps\,l} \bigr\ra  = 0$ and
\bel{art2 -equa-1134-9}
\aligned
&\bigl\|\xiperp_\eps,\xipar_\eps\bigr\|_{H^1_{- \expoP}(\Lambda)}^2 = 1 ,
\\
&\lim_{\eps \to 0} \biggl(
\Bigl\| \xiperp_\eps + \frac{1}{n-1}
\nablaslash\cdot\xipar_\eps\Bigr\|_{L^2_{- \expoP}(\Lambda)}^2
  + \Bigl\| \nablaslash\xiperp_\eps - \xipar_\eps
\Bigr\|_{L^2_{- \expoP}(\Lambda)}^2
+ \Bigl\| \Sym(\nablaslash\xipar_\eps)^\circ
\Bigr\|_{L^2_{- \expoP}(\Lambda)}^2
\biggr)
= 0 .
\endaligned
\ee
The first condition gives a uniform bound in the weighted Hilbert space $H^1_{- \expoP}(\Lambda)$. After extracting a subsequence, $(\xiperp_\eps,\xipar_\eps)$ therefore converges weakly in this space, and strongly\footnote{This is a consequence of the compact embedding $H^1_{- \expoP}(\Lambda)\Subset L^2_{- \expoP}(\Lambda)$, cf.~\eqref{equa-section5-weighted-rellich}, below.} in $L^2_{-\expoP}(\Lambda)$, to a limit $(\xiperp_0,\xipar_0)$. The average condition also passes to the weak limit.
By virtue of the lower semi-continuity of the weighted norm we deduce that 
\be
\Bigl\| \xiperp_0 + \frac{1}{n-1} \nablaslash\cdot\xipar_0\Bigr\|_{L^2_{- \expoP}(\Lambda)}^2
+ \Bigl\| \nablaslash\xiperp_0 - \xipar_0 \Bigr\|_{L^2_{- \expoP}(\Lambda)}^2
+ \Bigl\| \Sym(\nablaslash\xipar_0)^\circ \Bigr\|_{L^2_{- \expoP}(\Lambda)}^2
\leq 0 .
\ee
As a result, we get the pointwise identities (almost everywhere on $\Lambda$)
\be
\xiperp_0 + \frac{1}{n-1} \nablaslash\cdot\xipar_0 = 0 , \quad
\nablaslash\xiperp_0 - \xipar_0 = 0 , \quad
\Sym(\nablaslash\xipar_0)^\circ = 0 .
\ee
The second equation gives $\xipar_0 = \nablaslash\xiperp_0$, so that $\xiperp_0\in H^2_{-\expoP}(\Lambda)$.  The first identity then gives $\Deltaslash\xiperp_0 = -(n-1)\xiperp_0$ and the last gives $(\nablaslash^2\xiperp_0)^\circ=0$, hence altogether
\be
\nablaslash^2\xiperp_0=-\xiperp_0\gslash.
\ee
On a connected domain of the round sphere, its solution are the ambient linear functions $\xiperp_0 = L^k \xh_k$ for some constant vector $L\in\RR^n$.  Hence, $\xipar_0 = \nablaslash\xiperp_0 = L^k\nablaslash\xh_k$.
The average condition $\bigl\la 2 \, \xh_l\, \xiperp_0 + \xipar_{0\,l} \bigr\ra  = 0$ then becomes $0 = L^k \bigl\la 2 \, \xh_l\, \xh_k + \nablaslash_l\xh_k \bigr\ra = L^k T_{kl}$, which implies $L^k=0$.

Altogether, $\xi_0=0$, so $\xi_\eps$ converges weakly to zero in $H^1_{- \expoP}(\Lambda)$, and strongly in $L^2_{-\expoP}(\Lambda)$.  Together with the first and last terms in the limit in~\eqref{art2 -equa-1134-9}, this implies that
\be
\bigl\| \Sym(\nablaslash\xipar_\eps) \bigr\|_{L^2_{-\expoP}(\Lambda)}
\lesssim \|\xiperp_\eps\|_{L^2_{-\expoP}(\Lambda)} + \Bigl\| \xiperp_\eps + \frac{1}{n-1} \nablaslash\cdot\xipar_\eps \Bigr\|_{L^2_{-\expoP}(\Lambda)} + \Bigl\| \Sym(\nablaslash\xipar_\eps)^\circ \Bigr\|_{L^2_{-\expoP}(\Lambda)}
\to 0 .
\ee
By \autoref{art2 -equa--weighted-Korn}, there exist Killing fields $\zeta_\eps$ such that $\xipar_\eps- \zeta_\eps\to0$ strongly in $H^1_{- \expoP}(\Lambda)$. The sequence $\zeta_\eps$ is bounded; since the space of conformal Killing fields is finite-dimensional, a subsequence converges strongly in $H^1_{- \expoP}(\Lambda)$. Its limit must coincide with the weak limit of $\xipar_\eps$, and hence vanishes. Therefore $\xipar_\eps\to0$ strongly in $H^1_{- \expoP}(\Lambda)$. The second term in \eqref{art2 -equa-1134-9} now gives $\nablaslash\xiperp_\eps\to0$. Moreover, $\la\xiperp_\eps\ra\to0$ by weak convergence, so the weighted Poincaré inequality yields $\xiperp_\eps\to0$ strongly in $L^2_{- \expoP}(\Lambda)$. Thus $\xi_\eps\to0$ strongly in $H^1_{- \expoP}(\Lambda)$, contradicting the normalization in \eqref{art2 -equa-1134-9}.
\ese
\end{proof} 


\subsubsection{Main statement of harmonic stability}

We now apply our weighted Korn-Poincaré inequality in \autoref{art2 -lem-poinckorn} and control the functional~$\ssrmB^{\lambda}$.

\begin{theorem}[Harmonic stability for the momentum operator $\ssB^\lambda$] 
\label{art2 -lem-39harmmom}
Fix $p\in(p_n^\flat,n- 2)$ and a localization function $\lambda$, with $d\chi=\lambda^{2\expoP}d\xh$. Assume that $a_{n,p}$ is sufficiently small, specifically
\bel{harmmom-anp-small}
0 < a_{n,p} < \sqrt{2} (\CzeroKP)^{-1}. 
\ee
Then the \underline{momentum harmonic stability}~\eqref{equa-stable-M-414} holds and, moreover, denoting $\eps = 2^{-1/2}\CzeroKP a_{n,p}$,
\be
\aligned
\ssrmB^{\lambda}[\xi]
& \geq (1 - \eps) \ssrmB^{\lambda}[\xi]\bigr|_{a_{n,p}=0}
\quad \text{provided } \, \bigl\la 2 \, \xh_l\, \xi^{\perp} + \xi_{l}^{\parallel} \bigr\ra  = 0 .
\endaligned
\ee   
\end{theorem}

\begin{proof}
We have
\be
\ssrmB^{\lambda}[\xi] = \ssrmB^{\lambda}[\xi]\Bigr|_{a_{n,p}=0} - \frac{a_{n,p}}{2} \fint_{\Lambda} (\nablaslash\xiperp- \xipar)\cdot\xipar \, d\chi ,
\ee
where the first term is coercive thanks to~\eqref{art2 -equa--1133} hence dominates the second one provided $a_{n,p}$ is small enough.
For a quantitative result, we control $\nablaslash \xiperp -  \xipar$ by the first line in~\eqref{art2 -equa--1133} and $\xipar$ by the second to get
\be
\biggl| \fint_{\Lambda} (\nablaslash\xiperp- \xipar)\cdot\xipar \, d\chi \biggr|
\leq \bigl\|\nablaslash\xiperp- \xipar\bigr\|_{L^2_{- \expoP}(\Lambda)} \|\xipar\|_{L^2_{- \expoP}(\Lambda)}
\leq \sqrt{2} \CzeroKP \ssrmB^{\lambda}[\xi]\Bigr|_{a_{n,p}=0} .
\qedhere
\ee
\end{proof}


\subsection{Radial stability under a Korn-Poincaré inequality}
\label{art2 -section=3.2}

\subsubsection{Structure matrices}

Assume henceforth that harmonic stability holds, so that $\ker(\ssB^\lambda)$ is $n$-dimensional, and let $\{\xi^{(j)}\}_{j=1}^n$ be any basis of this space. Recall our notation~\eqref{equa-the-matrix}, namely
\be
(\Xi^\notreM)^{(j)}{}_k
\coloneqq
\bigl\la - \nablaslash_k \xi^{(j) \perp} + 2 a_{n,p} \xh_k\, \xi^{(j) \perp} \bigr \ra
+(1+a_{n,p}) \la \xi^{(j) \parallel}{}_k \ra.
\ee
Radial stability is the statement that this matrix is invertible.  The matrix $\Xi^\notreM$ depends on the chosen basis of the kernel, whereas its invertibility does not.
We also use the matrix product
\be
(\Xi^\notreM T^{-1})^{(j)l} \coloneqq (\Xi^\notreM)^{(j)}{}_{k} (T^{-1})^{kl}
\ee
with the usual convention for matrix multiplication.

\subsubsection{Silhouette fluctuations and structure matrices}

Before radial stability has been established, no preferred basis is available.  Nevertheless, as the basis is eventually chosen to consist of the silhouette vector fields, we shall denote it $\{\xi^{\normal(i)}\}_{i=1}^n$; the superscript $\normal$ is at this stage only a provisional label and does not yet impose the normalization~\eqref{equa-M-normalization}.
It is useful to decompose the structure matrix as $\Xi^\notreM = - \zeroXi + a_{n,p} \matQ \, T$, namely
\bel{Xi-pieces}
\aligned
(\Xi^\notreM)^{(j)}{}_k
& \coloneqq - \zeroXi^{(j)}{}_k + a_{n,p}  \matQ^{(j)l} \, T_{lk} ,
\\
\zeroXi^{(i)}{}_k & \coloneqq \bigl\la \nablaslash_k \xi^{\normal (i)\perp} - \xi^{\normal (i)\parallel}{}_k\bigr\ra ,
\\
\matQ^{(j)k} & \coloneqq  \bigl\la 2 \xh_l \xi^{\normal (j) \perp} + \xi^{\normal (j) \parallel}{}_l\bigr\ra (T^{-1})^{lk} = \projP^k(\xi^{\normal (j)}) ,
\endaligned
\ee
in terms of $\projP^k$ defined in~\eqref{equa-tildemoment}.

For the corresponding fluctuations, a direct calculation gives the identity 
\bel{art2 -ssrmbxisil} 
\ssrmB^{\lambda}[\xi^{\normal (i) \fluc}, \xi^{\normal (j) \fluc}]
= \frac{a_{n,p}}{2} (\zeroXi \matQ^t)^{(i)(j)}
= \frac{a_{n,p}}{2} \zeroXi^{(i)}{}_k \matQ^{(j)k} .
\ee
Observe that the right-hand side of \eqref{art2 -ssrmbxisil} depends on the silhouette fields themselves, not just their fluctuation. 
Interestingly, $\zeroXi$ has the same expression in terms of fluctuations, namely 
\be
\zeroXi^{(i)}{}_k = \bigl\la \nablaslash_k \xi^{\normal (i)\fluc\perp} - \xi^{\normal (i)\fluc\parallel}{}_k\bigr\ra.
\ee
The above identity \eqref{art2 -ssrmbxisil} is analogous to the one derived in \eqref{art2 -equa-428} for the Hamiltonian  operator, but is more involved.
We seek inequalities relating the matrices $\matQ,\zeroXi,T,U,S^2$.

In the calculations below, we also use the transpose notation $\zeroXi^t = (\zeroXi_k{}^{(j)}{})$ and $\matQ^t = (\matQ^{k (j)})$, and we also write $\matQ^{-1} = \bigl( (\matQ^{-1})_{k(j)}\bigr)$. 
Hence, for instance, in \eqref{art2 -equa-matrix-formu}, below, we use the convention 
\be
\zeroXi U^{-1} \zeroXi^t
= \bigl( (\zeroXi U^{-1} \zeroXi^t)^{(i)(j)} \bigr) = \bigl(\zeroXi^{(i)}{}_k (U^{-1})^{kl}  \zeroXi^{(j)}{}_l\bigr).
\ee

\subsubsection{Main statement of radial stability}

\begin{theorem}[Radial stability for the momentum operator $\ssB^\lambda$] 
\label{art2 -coro-radial-momen}
Fix $p\in(p_n^\flat,n- 2)$ and a localization function $\lambda$, with $d\chi=\lambda^{2\expoP}d\xh$. Assume the notation and inequalities in \eqref{art2 -equa--1133-copie} and \eqref{art2 -equa--1133}, together with
\bel{art2 -coro-radial-momen-anprange}
0 < a_{n,p} < \sqrt{2} (\CzeroKP)^{-1} \, |S^{- 2}|^{-1}.
\ee
Then the matrix inequalities $0<(1- \eps)\matQ^{-1}\zeroXi\leq a_{n,p}U$ and $\matQ^{-1}\Xi^\notreM\geq a_{n,p}S^2$ hold, where $\eps = 2^{-1/2}\CzeroKP a_{n,p}$. Matrix inequalities are understood in terms of symmetric parts. Namely, for all (constant) non-zero vectors $\wvect = (\wvect^k) \in\RR^n$ one has
\bse
\begin{align}
\label{art2 -VMQ1VaVUV}
0 < \wvect^k (\matQ^{-1})_{k(i)} \zeroXi^{(i)}{}_l \wvect^l & \leq \frac{1}{1- \eps} a_{n,p} \wvect^k U_{kl} \wvect^l .
\\
\label{art2 -avxi}
0 < a_{n,p} \wvect^k (S^2)_{kl} \wvect^l & \leq \wvect^k (\matQ^{-1})_{k(i)} (\Xi^\notreM)^{(i)}{}_l \wvect^l ,
\end{align}
\ese
As a consequence, the \underline{radial stability} condition \eqref{equa-Xi-invertible} in \autoref{def-radial-Mstab} holds.
\end{theorem}

\begin{proof} {\it 1. Invertibility of the structure matrices.} Harmonic stability makes the moment map $\xi\mapsto\la2\xh\xi^\perp+\xi^\parallel\ra$ injective on $\ker(\ssB^\lambda)$.  Since this kernel has dimension $n$, the matrix $\matQ$ is invertible.  To prove that $\zeroXi$ is invertible, suppose that $\zeroXi^t v=0$.  The identity~\eqref{art2 -ssrmbxisil} and harmonic coercivity imply that the fluctuation of $v_{(i)}\xi^{\normal(i)}$ vanishes.  Thus this kernel element is of the form $c^k\xistar_k$.  Since
\be
\nablaslash(c^k\xh_k)=c^k\xistarpar_k,
\ee
direct substitution in the definition of $\ssB^\lambda$ gives
\be
\ssB^\lambda[c^k\xistar_k]
=\frac{a_{n,p}}2
\left(
\lambda^{- 2\expoP}\nablaslash\cdot\bigl(\lambda^{2\expoP}\nablaslash(c^k\xh_k)\bigr),
\nablaslash(c^k\xh_k)
\right).
\ee
Since $a_{n,p}>0$, such a field belongs to $\ker(\ssB^\lambda)$ only when $c=0$.  Hence $v=0$, and $\zeroXi$ is invertible.

\vskip.2cm
\bse
\noindent{\it 2. Estimates for the averages.} For any non-zero linear combination $\xi=\sum_iv_{(i)}\xi^{\normal (i) \fluc}$ the coercivity established earlier in \autoref{art2 -lem-39harmmom} implies
\be
\ssrmB^{\lambda} [\xi, \xi] 
\geq {1- \eps \over 2} \big| \bigl\la \nablaslash \xi^{\perp} - \xi^{\parallel}{}\bigr\ra \bigr|^2 _{U^{-1}}
\ee
with $\eps = 2^{-1/2} \CzeroKP a_{n,p}$.
Using matrix notation and the calculation of $\ssrmB^{\lambda}$ applied to pairs of silhouette vector fields, with the notation in \eqref{art2 -ssrmbxisil} we thus deduce that 
\bel{art2 -equa-matrix-formu}
a_{n,p} \wvect^t \zeroXi \matQ^t \wvect
\geq (1- \eps) \wvect^t \zeroXi U^{-1} \zeroXi^t \wvect , \qquad \wvect \neq 0.
\ee

\vskip.2cm

\noindent{\it 3. Matrix inequalities.} To prove the lower bound in~\eqref{art2 -VMQ1VaVUV} (positivity of~$\matQ^{-1} \zeroXi$), we take $v \coloneqq  (\matQ^t)^{-1} \wvect$ in~\eqref{art2 -equa-matrix-formu}; its right-hand side is positive because $\zeroXi$ is invertible and $U^{-1}$ is positive-definite. To prove the upper bound, we write
\be
\aligned
& \matQ \Bigl(\frac{a_{n,p}}{1- \eps} U - \matQ^{-1} \zeroXi\Bigr)^t \matQ^t
\\
& = \Bigl( \zeroXi \matQ^t - \frac{1- \eps}{a_{n,p}} \zeroXi U^{-1} \zeroXi^t\Bigr) + (1- \eps) a_{n,p}^{-1} \Bigl(\frac{a_{n,p}}{1- \eps} \matQ U - \zeroXi\Bigr) U^{-1} \Bigl(\frac{a_{n,p}}{1- \eps} \matQ U - \zeroXi\Bigr)^t .
\endaligned
\ee
The first term has non-negative symmetric part by \eqref{art2 -equa-matrix-formu}, and so does the second term because $U^{-1}$ is positive-definite. Their sum therefore has non-negative symmetric part, which establishes the upper bound in \eqref{art2 -VMQ1VaVUV}.
 
\medskip

\noindent{\it 4. Radial stability.}
Using $T=\Id+S^2$, $U=\Id-S^2$, and $\Xi^\notreM = -\zeroXi + a_{n,p} \matQ \, T$, the upper bound in~\eqref{art2 -VMQ1VaVUV} can be restated as
\be
\wvect^k (\matQ^{-1})_{k(i)} (\Xi^\notreM)^{(i)}{}_l \wvect^l
\geq \frac{1}{1- \eps} a_{n,p} \wvect^k \Bigl((2 - \eps) (S^2)_{kl} - \eps \delta_{kl}\Bigr) \wvect^l .
\ee
Then, for the range of values~\eqref{art2 -coro-radial-momen-anprange} of~$a_{n,p}$ we have $\eps\leq \wvect^t S^2 \wvect/|\wvect|^2<1$ and thus we have
\be
(2 - \eps) \wvect^t S^2 \wvect - \eps |\wvect|^2 \geq (1- \eps) \wvect^t S^2 \wvect , \qquad 1- \eps > 0 ,
\ee
which establishes~\eqref{art2 -avxi}.
As a result, $\matQ^{-1}\Xi^\notreM$ is greater than a positive-definite matrix.
Thus, there cannot be any non-zero vector $\wvect$ in the kernel $\ker\Xi^\notreM$: for such a vector the right-hand side of~\eqref{art2 -avxi} would vanish, while the left-hand side would be positive.
We conclude that $\Xi^\notreM$ is invertible.
\ese
\end{proof}


\subsection{Estimates for the momentum silhouettes}

We return to the identity \eqref{art2 -ssrmbxisil} for the momentum fluctuations and continue to rely on the coercivity of $\ssrmB^{\lambda}$.  The matrix $\matQ=(\matQ^{(j)k})$ is invertible by the preceding proof, and the estimate \eqref{art2 -equa8d0} below involves this matrix.  The preceding matrix inequalities are covariant under a change of basis of the kernel.  Since $\Xi^\notreM$ is now known to be invertible, we may therefore replace the provisional basis by the uniquely normalized silhouette basis satisfying~\eqref{equa-M-normalization}, motivated by the ADM momentum at the asymptotic end.

\begin{proposition}[Fluctuations of the momentum silhouette fields]
In the setup of \autoref{art2 -coro-radial-momen} and especially under the smallness condition~\eqref{art2 -coro-radial-momen-anprange} on~$a_{n,p}$,
the fluctuations of the silhouette vector fields $\xi=v_{(j)}\xi^{\normal(j)}$ for any constant vector $v=(v_{(i)})$ satisfy, for $\eps=2^{-1/2} \CzeroKP a_{n,p}$,
\bel{art2 -equa8d0}
\aligned
\frac{1}{(\ConeKP)^2} \bigl\|\nablaslash\xi^{\fluc\perp},\nablaslash\xi^{\fluc\parallel}\bigr\|_{L^2_{- \expoP}(\Lambda)}^2
+ \frac{1}{(\CzeroKP)^2} \bigl\|\xi^{\fluc\perp},\xi^{\fluc\parallel}\bigr\|_{L^2_{- \expoP}(\Lambda)}^2
\leq \frac{a_{n,p}^2}{(1- \eps)^2} \bigl|v_{(j)} \matQ^{(j)\bullet}\bigr|_U^2 .
\endaligned
\ee
\end{proposition}

\begin{proof}
\bse
For any (non-zero) constant vector $v_{(j)}$, in the (scalar) equation
\be
2 \ssrmB^{\lambda}[v_{(i)}\xi^{\normal(i)\fluc}, v_{(j)}\xi^{\normal(j)\fluc}] 
= v_{(i)}v_{(j)} a_{n,p} \matQ^{(i)k} \zeroXi^{(j)}{}_k,
\ee
the left-hand side is coercive thanks to \autoref{art2 -lem-39harmmom} and \autoref{art2 -lem-poinckorn}, which apply to the vector field $\xi^\fluc$ since $ \bigl\la 2 \, \xh_l\, \xi^{\normal (j) \fluc\perp} + \xi^{\normal (j) \fluc\parallel} \bigr\ra  = 0$.
For the right-hand side, we apply \eqref{art2 -VMQ1VaVUV} with $\wvect^k$ replaced by $v_{(j)} \matQ^{(j)k}$, namely 
\be
0 < v_{(j)} v_{(i)}  \matQ^{(j)k} \zeroXi^{(i)}{}_k  
\leq \frac{a_{n,p}}{1- \eps} \bigl|v_{(j)} \matQ^{(j)\bullet}\bigr|_U^2 .
\qedhere
\ee
\ese
\end{proof} 

 
\section{Application to the Einstein constraint equations and the small aperture limit}
\label{section-5}

\subsection{Localization problem of interest}
\label{sectionN3-localization}

This section reconnects the spherical analysis of Sections~\ref{section-2} to~\ref{section-4} with the localization problem for the Einstein constraint equations.  We recall the geometric origin of the two operators, derive their harmonic--spherical decompositions, and establish the weighted Poincar\'e estimates needed for spherical domains of arbitrarily small diameter.  When these domains are realized as conical slices, this is the small-aperture regime of the localization problem.

Let $\Mbf$ be an $n$-dimensional manifold, with $n\geq3$, and let $(g,k)$ be an initial data set.  In the constraint equations~\eqref{eq:ee11}, we replace $k$ by the contravariant momentum tensor
\be
h\coloneqq\bigl(k-(\Tr_gk)g\bigr)^{\sharp\sharp}.
\ee
In these variables, the vacuum constraints are $\notreH(g,h)=0$ and $\notreM(g,h)=0$ in terms of the maps
\bel{equa-section5-constraint-map}
\aligned
\notreH(g,h)
&\coloneqq R_g+\frac1{n-1}(\Tr_gh)^2 -|h|_g^2,
\qquad
&
\notreM(g,h)
&\coloneqq\bfDiv_gh.
\endaligned
\ee
Given seed data $(g_s,h_s)$, the localization problem consists in finding a correction $(\gdiff,\hdiff)$, supported in a chosen gluing domain $\Omega$, such that $(g_s+\gdiff,h_s+\hdiff)$ satisfies the constraints and retains the prescribed exterior data away from~$\Omega$.
The variational construction in~\cite{LL-optimal-main} chooses corrections in the range of weighted formal adjoint operators (cf.~\eqref{equa-section5-weighted-adjoint-Ansatz} below).

We now specialize to one asymptotically Euclidean end.  Write $r=|x|$ and $\xh=x/r$.  A conical localization region is determined by a connected spherical domain $\Lambda\subset\Sbb^{n-1}$ and, outside a compact set, has the form
\bel{equa-section5-cone}
\Omega_R=\{x\in\RR^n:r>R,\ \xh\in\Lambda\}
=\bigcup_{r>R}\Lambda_r,
\qquad
\Lambda_r=r\Lambda.
\ee
The angular localization function $\lambda$ is positive on $\Lambda$, vanishes linearly at $\del\Lambda$, and is extended homogeneously along the rays, so $\lambda(r\xh)=\lambda(\xh)$.  For a projection exponent $p\in(0,n- 2)$ and a localization exponent $\expoP\geq2$, we use
\bel{omegabfp}
\omega_p=\lambda^{\expoP}r^{\frac n2 -p},
\qquad
d\chi=\lambda^{2\expoP}d\xh.
\ee
The exponent $p$ determines the radial decay selected by the projection, whereas $\expoP$ determines the order of vanishing at the lateral boundary.  The value $p=n- 2$ is the harmonic threshold; throughout the stability theorems we impose $p\in(p_n^\flat,n- 2)$.

In an asymptotically Euclidean end, the constraint equations $\notreH(g,h)=0$ and $\notreM(g,h)=0$ can be expressed in terms of their linearization at the Euclidean data $(\delta,0)$, with radially-decaying sources, and the adjoint operators involved in selecting $(\gdiff,\hdiff)$ likewise reduce to Euclidean ones.
At the Euclidean data $(\delta,0)$, the linearized operators and their formal adjoints are
\bel{equa-section5-linearized-constraints}
\aligned
d\notreH_{(\delta,0)}[\gdiff,\hdiff]
&=\del_i\del_j\gdiff_{ij}- \Delta(\Tr_\delta\gdiff),
&
d\notreH_{(\delta,0)}^*[u]
&=\Hess u-(\Delta u)\delta,
\\
d\notreM_{(\delta,0)}[\gdiff,\hdiff]_i
&=\del_j\hdiff_i{}^j,
&
d\notreM_{(\delta,0)}^*[Z]_{ij}
&=- \frac12(\del_iZ_j+\del_jZ_i).
\endaligned
\ee
The corrections are then parametrized at leading order by
\bel{equa-section5-weighted-adjoint-Ansatz}
\gdiff =\omega_p^2d\notreH_{(\delta,0)}^*[u],
\qquad
\hdiff =\omega_{p+1}^2d\notreM_{(\delta,0)}^*[Z].
\ee
Composing the linearized constraints with these weighted adjoints produces a fourth-order scalar equation for $u$ and a second-order vector equation for~$Z$.  This ``squaring'' converts the underdetermined constraint system into elliptic equations with unique solutions.

The Euclidean model operators resulting from~\eqref{equa-section5-weighted-adjoint-Ansatz} are the localized Hamiltonian and momentum operators
\bel{equa:acalew0}
\aligned
\notreH^\lambda[u]
&\coloneqq \omega_p^{- 2}\Bigl(
(n-1)\Delta(\omega_p^2\Delta u)
+\del_i\del_j(\omega_p^2)\del_i\del_j u
- \Delta(\omega_p^2)\Delta u
\Bigr),
\\
\notreM^\lambda[Z]^i
&\coloneqq- \frac12(\Delta Z_i+\del_j\del_iZ_j)
-(\del_j\log\omega_{p+1})(\del_jZ_i+\del_iZ_j).
\endaligned
\ee
In the full asymptotically Euclidean problem, the corresponding operators differ from~\eqref{equa:acalew0} by variable-coefficient terms with additional radial decay.  The Euclidean operators therefore determine the indicial structure and the stability conditions that control the asymptotic problem.  The inner boundary conditions at $r=R$, which are required for the global projection construction, do not enter the analysis at infinity carried out here.

The construction in~\cite{LL-optimal-main} shows that super-harmonic localization follows once the Hamiltonian and momentum equations both satisfy harmonic, radial, and shell stability.  The role of the present paper is to verify these abstract hypotheses through functional inequalities for the localization function $\lambda$. The first two stability levels are purely spherical and were treated in Sections~\ref{section-2} to~\ref{section-4}.  Shell stability also involves the radial evolution and will be established in Sections~\ref{section-6} to~\ref{section-9}.


\subsection{Localized harmonic--spherical decompositions}
\label{sectionN2 -1}

The scale invariance of $\lambda$ allows us to separate radial derivatives from derivatives tangent to the unit sphere.  We write
\be
\vartheta\coloneqq r\del_r,
\qquad
a_{n,p}=2(n- 2 -p),
\ee
and use $\nablaslash$ and $\Deltaslash$ for the Levi--Civita connection and Laplace--Beltrami operator of the round metric $\gslash$.  The constants $b_{n,p}$ and $c_{n,p}$ are listed in \autoref{sec-appendix-A}.

\subsubsection{Hamiltonian decomposition}

The localized Hamiltonian operator admits the decomposition
\bse
\label{equa--488- 2}
\be
\label{equa-key-decompose-H-repeat}
r^4\notreH^\lambda[u]
=\Arr[u]+\Ars^\lambda[u]+\ssA^\lambda[u],
\ee
where
\bel{equa-62b}
\aligned
\Arr[u]
&\coloneqq(n-1)\vartheta(\vartheta+a_{n,p})
\bigl(\vartheta^2+a_{n,p}\vartheta-b_{n,p}\bigr)u,
\\
\Ars^\lambda[u]
&\coloneqq\lambda^{- 2\expoP}(\vartheta+a_{n,p})\Bigl(
2\vartheta\nablaslash\cdot(\lambda^{2\expoP}\nablaslash u)
+(n- 2)\vartheta\bigl(\lambda^{2\expoP}\Deltaslash u
+\Deltaslash(\lambda^{2\expoP}u)\bigr)
\\
&\hspace{5.4cm}
+\frac{c_{n,p}}{a_{n,p}}
\bigl(\Deltaslash(\lambda^{2\expoP}u)
- \lambda^{2\expoP}\Deltaslash u\bigr)
\Bigr),
\\
\ssA^\lambda[u]
&\coloneqq\lambda^{- 2\expoP}\Bigl(
(n- 2)\Deltaslash(\lambda^{2\expoP}\Deltaslash u)
+\nablaslash^a\nablaslash^b
\bigl(\lambda^{2\expoP}\nablaslash_a\nablaslash_bu\bigr)
\\
&\hspace{5.4cm}
- 2(a_{n,p}+1)\nablaslash\cdot
(\lambda^{2\expoP}\nablaslash u)
-c_{n,p}\Deltaslash(\lambda^{2\expoP}u)
\Bigr).
\endaligned
\ee
\ese
Here $\Arr$ is purely radial and independent of $\lambda$, $\Ars^\lambda$ contains both radial and spherical derivatives, and $\ssA^\lambda$ is purely spherical.  If $\nu$ is a function of~$\xh$, then $\vartheta(r^{-a_{n,p}}\nu)=-a_{n,p}r^{-a_{n,p}}\nu$.  Every term in $\Arr+\Ars^\lambda$ therefore vanishes on this profile, and we have the intrinsic characterization
\bel{equa-asymptoOper-reapeat}
\ssA^\lambda[\nu]
=r^{4+a_{n,p}}\notreH^\lambda[r^{-a_{n,p}}\nu].
\ee
Thus, $\ssA^\lambda$ is the operator seen by the borderline harmonic profile selected by the weight.  Its kernel determines the Hamiltonian silhouette.

\subsubsection{Momentum decomposition}
\label{section=9.1}

For a vector field $Z$ on the cone, we use the radial--tangential splitting
\be
Z_i=\xh_iZ^\perp+Z_i^\parallel,
\qquad
\xh^iZ_i^\parallel=0.
\ee
The momentum operator is correspondingly represented by a $2\times2$ block operator.  One has
\bel{equa-section5-momentum-decomposition}
r^2\notreM^{\lambda\bullet\bullet}[Z]
=\Brr^{\bullet\bullet}[Z]
+\Brs^{\lambda\bullet\bullet}[Z]
+\ssB^{\lambda\bullet\bullet}[Z],
\ee
where the non-zero components of radial and mixed derivative operators are
\bel{equa-Bmoment}
\aligned
\Brr^{\perp\perp}[Z^\perp]
&=- \vartheta(\vartheta+a_{n,p})Z^\perp,
&
\Brs^{\lambda\parallel\perp}[Z^\perp]_a
&=- \frac12(\vartheta+a_{n,p})\nablaslash_aZ^\perp,
\\
\Brs^{\lambda\perp\parallel}[Z^\parallel]
&=- \frac12\lambda^{- 2\expoP}(\vartheta+a_{n,p})
\nablaslash\cdot(\lambda^{2\expoP}Z^\parallel),
&
\Brr^{\parallel\parallel}[Z^\parallel]_a
&=- \frac12\vartheta(\vartheta+a_{n,p})Z_a^\parallel.
\endaligned
\ee
All unlisted components of $\Brr$ and~$\Brs$ vanish.
The components of the purely spherical derivative operators are
\bel{Brs-expr-main}
\aligned
\ssB^{\lambda\perp\perp}[\xi^\perp]
&=(n-1)\xi^\perp
- \frac12\lambda^{- 2\expoP}\nablaslash\cdot
(\lambda^{2\expoP}\nablaslash\xi^\perp),
\\
\ssB^{\lambda\parallel\perp}[\xi^\perp]_a
&=- \frac12\nablaslash_a\xi^\perp
- \lambda^{- 2\expoP}\nablaslash_a
(\lambda^{2\expoP}\xi^\perp),
\\
\ssB^{\lambda\perp\parallel}[\xi^\parallel]
&=\frac{1+a_{n,p}}2\lambda^{- 2\expoP}
\nablaslash\cdot(\lambda^{2\expoP}\xi^\parallel)
+\nablaslash\cdot\xi^\parallel,
\\
\ssB^{\lambda\parallel\parallel}[\xi^\parallel]_a
&=\frac{1+a_{n,p}}2\xi_a^\parallel
- \lambda^{- 2\expoP}\nablaslash^b
\Bigl(\lambda^{2\expoP}\Sym(\nablaslash\xi^\parallel)_{ab}\Bigr).
\endaligned
\ee
As in the Hamiltonian case, the factors $\vartheta+a_{n,p}$ show that, for every pair $\xi=(\xi^\perp,\xi^\parallel)$ that only depends on~$\xh$,
\bel{equa-section5-momentum-harmonic-profile}
\ssB^{\lambda\bullet\bullet}[\xi]
=r^{2+a_{n,p}}\notreM^{\lambda\bullet\bullet}
[r^{-a_{n,p}}\xi].
\ee
The operator $\ssB^\lambda$ is the momentum operator of main interest.  Its $n$-dimensional kernel gives the momentum silhouette vector fields. Sections~\ref{section-2} to~\ref{section-4} involve only the purely spherical operators $\ssA^\lambda$ and $\ssB^\lambda$; the radial and mixed pieces $\Arr,\Ars^\lambda,\Brr,\Brs^\lambda$ enter only in the construction of the shell functionals.


\subsection{From the constraint equations to the stability criteria}
\label{sectionN3-stability-bridge}

\subsubsection{On the origin of stability conditions}

The harmonic stability conditions arise very naturally in the study of the harmonic operators $\ssA^\lambda$ and $\ssB^\lambda$ on spherical domains.  While the radial stability conditions can be expressed in terms of the kernels of these operators, they only arise naturally once accounting for the radial and mixed operators $\Arr,\Ars^\lambda,\Brr,\Brs^\lambda$.
To motivate them, as well as the shell stability conditions, we briefly outline some aspects of the localization problem, concentrating on Hamiltonian operators for definiteness.  See \cite[Section 1.3]{LL-optimal-main} for additional details.

A key step is to understand the decay of solutions of $\notreH^\lambda[u] = 0$.
By a suitable integration on spherical shells $\Lambda_r = r\Lambda \subset \Omega_R$, the average $\la u\ra(r)$ on~$\Lambda_r$ satisfies an ordinary differential equation of the schematic form (cf.~\eqref{equa-defb-b-b}, below)
\bel{radial-eq-easy}
\bigl( - \bnotreH_{1} \vartheta(\vartheta+a_{n,p}) + \bnotreH_{0} \bigr) \, \la u\ra 
= \Kappa'[\ut] + C^u r^{-a_{n,p}}, 
\ee
where $\ut = u - \la u\ra$ are the fluctuations of~$u$ on~$\Lambda_r$ and $\Kappa'[\ut]$ is a linear functional in~$\ut$.
The radial decay of homogeneous solutions is controlled by characteristic exponents of the radial differential operator on the left-hand side of~\eqref{radial-eq-easy}.
The \emph{radial stability condition} is equivalent to
\be
\bnotreH_{1} \bnotreH_{0} > 0 ,
\ee
which states that these exponents are outside the interval $[-a_{n,p},0]$.
Thus, provided $\Kappa'[\ut]$ is controlled, the shell average identity~\eqref{radial-eq-easy} implies that $\la u\ra$ decays at least at the harmonic decay rate of the term $C^u r^{-a_{n,p}}$.

Fluctuations are controlled by several quadratic functionals defined as integrals over~$\Lambda_r$, and satisfying another radial differential equation.  The shell functional $\Phi^\notreH[u](r)$ involves up to $2$~derivatives of~$u$ and is non-negative, and the shell dissipation functionals $\Psi^\notreH_\beta[u](r)$ derivatives up to order~$3$.
Solving the equations yields
\bel{intro-Phi-solve}
\Phi^\notreH[u](r) + \sum_{\beta=a,2a} \text{radial integral of } \Psi^\notreH_\beta[u] = \Cstar r^{-2a} ,
\ee
with suitably positive integration weights.  If $\Psi^\notreH_\beta$ were coercive, then~\eqref{intro-Phi-solve} would control $u$ with $r^{-a}$~decay.  Instead, just as for the harmonic functional~$\ssrmA^\lambda$, the dissipations $\Psi^\notreH_\beta$ fail to be coercive due to the averages~$\la u\ra$.
It is only by combining the shell functionals together with the shell average identity that the desired decay is reached.

In short, the discussion above separates three logically distinct estimates.  Harmonic stability is a coercive estimate for the purely spherical operators modulo their distinguished finite-dimensional modes.  Radial stability asserts that the equations obtained by averaging the radial--spherical decomposition are non-degenerate on these modes.  Shell stability controls the remaining fluctuations by a quadratic functional whose radial identity has essentially a non-negative dissipation.
 These are exactly the hypotheses used in~\cite{LL-optimal-main} to obtain super-harmonic estimates for the weighted potentials $(u,Z)$ and hence for the metric and momentum corrections in~\eqref{equa-section5-weighted-adjoint-Ansatz}.

\subsubsection{Geometric constants}

Theorem~\ref{art2 -thm:informal-suff-stab} reduces the verification of these hypotheses to geometric constants associated with the weight $\lambda$.  For the Hamiltonian operator, harmonic and radial stability follow when $(n- 2 -p)C_\Poin$ is sufficiently small, while shell stability follows when
\bel{equa-section5-hamiltonian-smallness}
\diam(\Lambda)+\CPoin
+\CPoinTwo
\quad\hbox{is sufficiently small}.
\ee
For the momentum operator, the three stability conditions follow when $a_{n,p}=2(n- 2 -p)$ is sufficiently small relative to the weighted Poincar\'e, Korn, Korn--Poincar\'e, and Hardy constants of the fixed localization pair.  Consequently, two complementary parameters are available: the exponent $p$ can be moved toward the harmonic threshold, and the size of the spherical localization domain can be reduced.

The standard Poincar\'e constant controls mean-free scalar modes.  A second-order inequality is needed in the Hamiltonian shell estimate because the dissipation contains the spherical Hessian, whereas the error terms contain the spherical gradient.

We now turn to this limit of a small spherical region.
On a shrinking cap, ordinary Hessian coercivity cannot have a small constant: restrictions of ambient affine functions become almost affine in rescaled coordinates.  The correction involving the average gradient and the matrix $\matG$ in~\eqref{equa-Poin2} records these modes in an explicit term, leaving a coercive fluctuation estimate. The next two subsections first establish the three required inequalities on a fixed spherical domain and then determine their behavior under shrinking.

%


\subsection{First- and second-order Poincar\'e inequalities}
\label{sectionN3- 2}

The rest of this section concerns Poincaré-type inequalities.  Here in \autoref{sectionN3- 2} we simply prove that the inequalities hold for any localization function, for some constants.  In \autoref{sectionN3- 2 -complete} we describe how these constants scale in the small-domain limit. We prove these scaling properties in \autoref{section=5-derivation}.

\begin{proposition}[Three weighted Poincar\'e inequalities on a spherical domain]
\label{prop:three-weighted-poincare}
Let $\Lambda\subset\Sbb^{n-1}$ be a connected domain with smooth boundary. Assume that $\lambda\in C^\infty(\overline\Lambda)$ is positive in $\Lambda$, vanishes linearly on $\del\Lambda$, and is comparable to the spherical distance to $\del\Lambda$ in a boundary collar. Then the following statements hold.
\bei
\item[(i)] There is a finite constant $\CPoin$ such that, 
for every $\nu\in H^1_{- \expoP}(\Lambda)$, 
\bel{equa-section5-fixed-first-poincare}
\|\nu- \langle\nu\rangle\|_{L^2_{- \expoP}}
\leq \CPoin
\|\nablaslash\nu\|_{L^2_{- \expoP}}.
\ee

\item[(ii)] There is a finite constant $C_{\mathbf{Hess}}^\lambda$ such that, for every $\nu\in H^2_{- \expoP}(\Lambda)$, 
\bel{equa-section5-fixed-hessian-poincare}
\|\nablaslash\nu\|_{L^2_{- \expoP}}
\leq C_{\mathbf{Hess}}^\lambda
\|\nablaslash^2\nu\|_{L^2_{- \expoP}}.
\ee

\item[(iii)] 
The matrix 
\be
\matG=\Id+\langle\xh\otimes\xh\rangle - 2 \langle\xh\rangle \otimes \langle\xh\rangle
\ee
 is positive definite, and there is a finite optimal constant $\CPoinTwo$ for which
\bel{equa-section5-fixed-corrected-poincare}
\|\nablaslash\nu\|_{L^2_{- \expoP}}^2
\leq (\CPoinTwo)^2
\|\nablaslash^2\nu\|_{L^2_{- \expoP}}^2
+\langle\nablaslash\nu\rangle^t\matG
\langle\nablaslash\nu\rangle.
\ee
\eei
\end{proposition}

Clearly, we have
\be
\CPoinTwo
\leq C_{\mathbf{Hess}}^\lambda.
\ee
Proposition~\ref{prop:three-weighted-poincare} is a fixed-domain result only.  The constant $C_{\mathbf{Hess}}^\lambda$ generally deteriorates when $\Lambda$ shrinks, because the limiting Euclidean Hessian has the affine functions in its kernel.  The point of the corrected inequality is that, after the affine modes are recorded by $\langle\nablaslash\nu\rangle^t\matG\langle\nablaslash\nu\rangle$, its constant can be bounded by a multiple of the spherical-domain diameter; this is proven next. 

\begin{proof}
We first record the compactness property used twice below.  Under the assumptions on $\lambda$, one has
\bel{equa-section5-weighted-rellich}
H^1_{- \expoP}(\Lambda)\Subset L^2_{- \expoP}(\Lambda),
\qquad
H^2_{- \expoP}(\Lambda)\Subset H^1_{- \expoP}(\Lambda).
\ee
Indeed, this is the usual Rellich theorem on every subset compactly contained in $\Lambda$.  In a boundary chart $(z,t)$, where $t$ is the distance to the boundary, the density is uniformly comparable, together with its smooth tangential factors, to $t^{2\expoP}dz\,dt$.  The one-dimensional weighted compactness theorem in the $t$ variable, followed by the ordinary tangential Rellich theorem, gives~\eqref{equa-section5-weighted-rellich} on the chart.  The comparison constants are uniform on a finite boundary atlas, and a partition of unity completes the proof.

\bse
To prove item~\textup{(i)}, suppose that~\eqref{equa-section5-fixed-first-poincare} fails.  After subtracting the averages and normalizing, there are functions $\nu_j$ such that
\be
\langle\nu_j\rangle=0,
\qquad
\|\nu_j\|_{L^2_{- \expoP}}=1,
\qquad
\|\nablaslash\nu_j\|_{L^2_{- \expoP}}\longrightarrow0.
\ee
By the first compact embedding in~\eqref{equa-section5-weighted-rellich}, a subsequence converges strongly in $L^2_{- \expoP}$ to a function $\nu$.  We have $\nablaslash\nu=0$ weakly, so connectedness implies that $\nu$ is constant.  Its weighted average is zero, hence $\nu=0$, in contradiction with the normalization.

For item~\textup{(ii)}, suppose that~\eqref{equa-section5-fixed-hessian-poincare} fails.  We may subtract $\langle\nu_j\rangle$ and normalize a sequence so that
\be
\|\nablaslash\nu_j\|_{L^2_{- \expoP}}=1,
\qquad
\|\nablaslash^2\nu_j\|_{L^2_{- \expoP}}\longrightarrow0.
\ee
Item~\textup{(i)} bounds $\nu_j$ in $H^1_{- \expoP}$, while the second embedding in~\eqref{equa-section5-weighted-rellich} gives, after extraction, strong convergence in $H^1_{- \expoP}$ to a function $\nu$ satisfying
\be
\nablaslash^2\nu=0,
\qquad
\|\nablaslash\nu\|_{L^2_{- \expoP}}=1.
\ee
\ese
Set $\wvect=\nablaslash\nu$.  The identity $\nablaslash \wvect=0$ means that $\wvect$ is parallel.  On the unit sphere,
\be
\Riem(X,Y)\wvect=\langle Y,\wvect\rangle X- \langle X,\wvect\rangle Y.
\ee
At a point at which $\wvect\neq0$, choose a non-zero vector $X$ orthogonal to $\wvect$ and set $Y=\wvect$.  Parallelness gives $\Riem(X,\wvect)\wvect=0$, whereas the preceding curvature identity gives $\Riem(X,\wvect)\wvect=|\wvect|^2X\neq0$.  Therefore $\wvect=0$ throughout $\Lambda$, which contradicts the normalization.

Finally, for $v\in\RR^n$ we have (with $\matV$ given in~\eqref{equa-matrix-STU})
\be
v^t\matG v
=|v|^2 -(v\cdot\langle\xh\rangle )^2+v^t\matV v, \qquad \matV = S^2 - \langle\xh\rangle\otimes\langle\xh\rangle.
\ee
The covariance matrix $\matV$ is non-negative and $|\langle\xh\rangle |<1$ because $\Lambda$ is open.  Hence $\matG$ is positive definite.  Adding the non-negative correction term in~\eqref{equa-section5-fixed-corrected-poincare} to the right-hand side of~\eqref{equa-section5-fixed-hessian-poincare} proves item~\textup{(iii)} with $C_{\Poin2}^\lambda \leq C_{\mathbf{Hess}}^\lambda$.
\end{proof}


\subsection{The small-domain limit of Poincar\'e constants}
\label{sectionN3- 2 -complete}

We seek here to prove (iii) in \autoref{prop:Poincare-scaling}.  For the reader's convenience, we present the geometric set-up again here.

Let $d=n-1$.  Fix a bounded, connected domain $\breve\Lambda\Subset B_1(0)\subset\RR^d$ with smooth boundary and a smooth defining function $\breve\lambda$ which is positive in $\breve\Lambda$, vanishes linearly on $\del\breve\Lambda$, and is comparable there to the Euclidean distance to the boundary. We use
\bel{equa-scaled-base-measure}
d\breve\mu
\coloneqq
\frac{\breve\lambda^{2\expoP}dy}
{\int_{\breve\Lambda}\breve\lambda^{2\expoP}dy},
\qquad
\langle f\rangle_0\coloneqq\int_{\breve\Lambda}f\,d\breve\mu,
\ee
and denote by $\breve C_\Poin$ the optimal Euclidean weighted Poincar\'e constant
\bel{equa-scaled-base-poincare}
\|f- \langle f\rangle_0\|_{L^2(d\breve\mu)}
\leq\breve C_\Poin\|Df\|_{L^2(d\breve\mu)}
\ee
with $D$ the coordinate derivatives on~$\RR^d$.
For $0<\alpha\leq1$, define a family of spherical domains by
\bel{equa-scaled-cap-chart}
\Phi_\alpha(y)
\coloneqq\bigl(\alpha y,\sqrt{1- \alpha^2|y|^2}\bigr),
\qquad
\Lambda_{(\alpha)}=\Phi_\alpha(\breve\Lambda),
\qquad
\lambda_{(\alpha)}\circ\Phi_\alpha =\breve\lambda.
\ee
We also consider averages over $(\Lambda_{(\alpha)},\lambda_{(\alpha)})$, namely
\be
\la f\ra_\alpha = \frac{\int_{\Lambda_{(\alpha)}} f \lambda_{(\alpha)}^{2\expoP} d\xh}{\int_{\Lambda_{(\alpha)}} \lambda_{(\alpha)}^{2\expoP} d\xh} .
\ee
The first and second moments of coordinates~$\xh_k$ allow us to write structure matrices analogous to~\eqref{equa-matrix-STU},
\bel{equa-matrix-STU-alpha}
\aligned
S_\alpha^2 & \coloneqq (\la\xh_k \xh_l \ra_\alpha) , \qquad
T_\alpha \coloneqq \Id + S_\alpha^2 ,
\qquad 
U_\alpha \coloneqq \Id - S_\alpha^2 ,
\\
\matV_\alpha & \coloneqq S_\alpha^2 - \la\xh\ra_\alpha \otimes \la\xh\ra_\alpha
= (\la\xh_k \xh_l\ra_\alpha - \la\xh_k\ra_\alpha\la\xh_l\ra_\alpha) ,
\qquad
G_\alpha \coloneqq U_\alpha + 2 \matV_\alpha ,
\endaligned
\ee
which are all positive definite.
We are interested in the $\alpha\to 0$ limit of Poincaré-type constants.

\begin{proposition}[The small-domain limit of Poincar\'e constants]
\label{lem:Poincare-scaling-complete}
For the family~\eqref{equa-scaled-cap-chart}, one has the following three properties.
\bei
\item[(i)] The first-order weighted Poincar\'e constant and the spherical diameter satisfy
\bel{equa-scaled-first-poincare}
\aligned
C_\Poin^{\lambda_{(\alpha)}}
&=\alpha\,\breve C_\Poin\bigl(1+\Obig(\alpha^2)\bigr),
\\
\diam(\Lambda_{(\alpha)})
&=\alpha\,\diam_{\mathrm E}(\breve\Lambda)
\bigl(1+\Obig(\alpha^2)\bigr).
\endaligned
\ee

\item[(ii)] For every fixed $\alpha>0$, the Hessian constant in
\bel{equa-fixed-hessian-poincare}
\|\nablaslash\nu\|_{L^2_{- \expoP}(\Lambda_{(\alpha)})}
\leq C_{\mathbf{Hess}}^{\lambda_{(\alpha)}}
\|\nablaslash^2\nu\|_{L^2_{- \expoP}(\Lambda_{(\alpha)})}
\ee 
is finite but $C_{\mathbf{Hess}}^{\lambda_{(\alpha)}}\to +\infty$ as $\alpha\to0$.

\item[(iii)]

The optimal constant in
\bel{equa-corrected-second-poincare}
\|\nablaslash\nu\|_{L^2_{- \expoP}}^2
\leq \bigl(C_{\Poin2}^{\lambda_{(\alpha)}}\bigr)^2
\|\nablaslash^2\nu\|_{L^2_{- \expoP}}^2
+\langle\nablaslash\nu\rangle^t\matG_\alpha
\langle\nablaslash\nu\rangle
\ee
satisfies
\bel{equa-corrected-second-scaling}
C_{\Poin2}^{\lambda_{(\alpha)}}
\lesssim\alpha.
\ee
All implicit constants depend only on $n$, $\expoP$, and the reference localization $\breve\lambda$.
\eei
\end{proposition}

The proposition gives the three smallness estimates used in item~\textup{(iii)} of Proposition~\ref{prop:Poincare-scaling}:
\bel{equa-section5-combined-small-aperture}
\diam(\Lambda_{(\alpha)})
+C_\Poin^{\lambda_{(\alpha)}}
+C_{\Poin2}^{\lambda_{(\alpha)}}
\lesssim\alpha.
\ee
The use of $\matG_\alpha$ is essential here.  It is intrinsic, depends only on the first two weighted moments of the spherical domain, and converges quadratically to a projector orthogonal to~$\la\xh\ra_\alpha$. Formula~\eqref{equa-G-exact-remainder}, derived in the proof below, gives instead an exact non-negative decomposition into a fluctuation term and a controlled affine remainder.


\subsection{Derivation of the small-domain limit}
\label{section=5-derivation}

We define the first and second moments of the reference pair by
\bel{equa-base-first-second-moments}
m\coloneqq\langle y\rangle_0,
\qquad
M\coloneqq\langle y\otimes y\rangle_0 .
\ee

\vskip.2cm

\noindent{\it 1. Reference Poincar\'e inequality.}
The weighted compactness argument in Proposition~\ref{prop:three-weighted-poincare}, applied to $(\breve\Lambda,\breve\lambda)$, proves that $\breve C_\Poin<+\infty$.  More explicitly, a sequence with zero $d\breve\mu$-average, unit $L^2(d\breve\mu)$ norm, and gradient converging to zero has a strongly convergent subsequence.  Its limit is both constant and of zero average, which contradicts its unit norm.

\vskip.2cm

\noindent{\it 2. Pullback of the metric and proof of item~\textup{(i)}.}
Write $s_\alpha(y)=\sqrt{1- \alpha^2|y|^2}$. A direct calculation gives
\bel{equa-scaled-metric-formulas}
\aligned
(\Phi_\alpha^*\gslash)_{ab}
&=\alpha^2\left(\delta_{ab}
+\frac{\alpha^2y_ay_b}{s_\alpha^2}\right),
&
(\Phi_\alpha^*\gslash)^{ab}
&=\alpha^{- 2}(\delta^{ab}- \alpha^2y^ay^b),
\\
\Phi_\alpha^*(d\chi)
&=\alpha^d\breve\lambda^{2\expoP}s_\alpha^{-1}dy,
&
s_\alpha^{-1}
&=1+\Obig(\alpha^2)
\quad\hbox{uniformly on $\breve\Lambda$}.
\endaligned
\ee
After normalization of the measure, the Rayleigh quotient is therefore
\bel{equa-scaled-rayleigh-quotient}
\frac{(C_\Poin^{\lambda_{(\alpha)}})^2}{\alpha^2}
=\sup_{f\not\equiv\mathrm{const}}
\frac{\displaystyle\inf_{c\in\RR}
\int_{\breve\Lambda}|f-c|^2s_\alpha^{-1}d\breve\mu}
{\displaystyle\int_{\breve\Lambda}
(\delta^{ab}- \alpha^2y^ay^b)
\partial_af\partial_bf\,s_\alpha^{-1}d\breve\mu}.
\ee
The numerator and denominator differ from their values at $\alpha =0$ by relative factors $1+\Obig(\alpha^2)$, uniformly in $f$.  Taking the supremum proves the first expansion in~\eqref{equa-scaled-first-poincare}.  Expanding the spherical distance between $\Phi_\alpha(y)$ and $\Phi_\alpha(y')$, uniformly for $y,y'\in\overline{\breve\Lambda}$, gives the diameter expansion.

\vskip.2cm

\noindent{\it 3. Fixed-domain Hessian coercivity.}
For every $\alpha>0$, the pair $(\Lambda_{(\alpha)},\lambda_{(\alpha)})$ satisfies the assumptions of Proposition~\ref{prop:three-weighted-poincare}.  Item~\textup{(ii)} of that proposition proves~\eqref{equa-fixed-hessian-poincare}.  Notice that this argument only gives finiteness for fixed $\alpha$.  We now consider ambient linear functions
\be
\nu_c(\xh) = c^k\xh_k.
\ee
They have
\bel{equa-coordinate-function-identities}
\nablaslash\nu_c=(\Id- \xh\otimes\xh)c,
\qquad
\nablaslash^2\nu_c=- \nu_c\gslash ,
\ee
hence the squared $L^2$~norm of their gradient and Hessian are
\be
\|\nablaslash\nu_c\|_{L^2_{-\expoP}(\Lambda_{(\alpha)})}^2 = c^t U_\alpha c ,
\qquad
\|\nablaslash^2\nu_c\|_{L^2_{-\expoP}(\Lambda_{(\alpha)})}^2 = (n-1) c^t S^2_\alpha c ,
\ee
in terms of the structure matrices $S^2_\alpha=\la\xh\xh\ra_\alpha$ and $U_\alpha = \Id - S^2_\alpha$.
With respect to the splitting $\RR^n=\RR^d\oplus\RR e_n$, the metric and measure expansions yield
\bel{equa-S-small-cap-expansion}
S_\alpha^2
=\begin{pmatrix}
\alpha^2M+\Obig(\alpha^4)&\alpha m+\Obig(\alpha^3)\\
\alpha m^t+\Obig(\alpha^3)&
1- \alpha^2\operatorname{tr}M+\Obig(\alpha^4)
\end{pmatrix}.
\ee
Thus, for $c$ orthogonal to~$e_n$, the gradient squared norm is of order one and the Hessian squared norm of order~$\alpha^2$, leading to a divergent lower bound for the Hessian constant,
\be
C_{\mathbf{Hess}}^{\lambda_{(\alpha)}}\gtrsim \frac{1}{\alpha} , \qquad \alpha \to 0 .
\ee

\vskip.2cm

\noindent{\it 4. Moment expansions and intrinsic correction.}
\bse
Next, we consider $\matG_\alpha = U_\alpha + 2\matV_\alpha$, where $\matV_\alpha = S^2_\alpha - \la\xh\ra_\alpha\otimes\la\xh\ra_\alpha$.  Since both $U_\alpha$ and $\matV_\alpha$ are positive-definite, $\matG_\alpha$~is as well.
As a complement to~\eqref{equa-S-small-cap-expansion}, we expand
\be
\langle\xh\rangle _\alpha
=\left(\alpha m,
1- \frac12\alpha^2\operatorname{tr}M\right)+\Obig(\alpha^3).
\ee
We next compare $\matG_\alpha$ with the orthogonal projection onto the hyperplane perpendicular to~$\langle\xh\rangle _\alpha$, that is,
\be
P_\alpha
= \Id - {\langle\xh\rangle_{(\alpha)} \otimes \langle\xh\rangle_{(\alpha)}
\over \bigl| \langle\xh\rangle_{(\alpha)}\bigr|^2} .
\ee
We have $\operatorname{tr}\matV_\alpha =1-|\langle\xh\rangle _\alpha|^2$, 
which gives the identity
\be
\matG_\alpha-P_\alpha
=\matV_\alpha
+ \bigl( \operatorname{tr}\matV_\alpha \bigr) \, {\langle\xh\rangle_{(\alpha)} \otimes \langle\xh\rangle_{(\alpha)} \over \big| \langle\xh\rangle_{(\alpha)} \bigr|^2}
\ee
from which it is immediate to see
\bel{equa-G-projector-comparison}
\|\matG_\alpha-P_\alpha\|
\lesssim\diam(\Lambda_{(\alpha)})^2
\ee
since $|\langle\xh\rangle _\alpha|$ remains uniformly bounded away from zero for small~$\alpha$.
\ese

\vskip.2cm

\noindent{\it 5. Separation of the affine component.}
\bse
With $\wvect_\alpha[\nu]\coloneqq\langle\nablaslash\nu\rangle_\alpha$, let us set
\bel{equa-G-affine-decomposition}
c_\alpha =U_\alpha^{-1}\wvect_\alpha[\nu],
\qquad
w=\nu-c_\alpha\cdot\xh.
\ee
The coordinate identities give
\be
\langle\nablaslash w\rangle_\alpha=0,
\qquad
\|\nablaslash\nu\|_{L^2_{- \expoP}(\Lambda_{(\alpha)})}^2
=\|\nablaslash w\|_{L^2_{- \expoP}(\Lambda_{(\alpha)})}^2
+c_\alpha^tU_\alpha c_\alpha.
\ee
\ese
Consequently, we have the exact identity
\bel{equa-G-exact-remainder}
\aligned
\mathcal Q_\alpha[\nu]
&\coloneqq
\|\nablaslash\nu\|_{L^2_{- \expoP}(\Lambda_{(\alpha)})}^2
- \wvect_\alpha[\nu]^t\matG_\alpha \wvect_\alpha[\nu]
\\
&=\|\nablaslash w\|_{L^2_{- \expoP}(\Lambda_{(\alpha)})}^2
+c_\alpha^tR_\alpha c_\alpha,
\qquad
R_\alpha\coloneqq U_\alpha-U_\alpha \matG_\alpha U_\alpha
\\
&=\|\nablaslash w\|_{L^2_{- \expoP}(\Lambda_{(\alpha)})}^2
+c_\alpha^t\Bigl(
S_\alpha^4U_\alpha
+2(U_\alpha\langle\xh\rangle _\alpha)
\otimes(U_\alpha\langle\xh\rangle _\alpha)
\Bigr)c_\alpha.
\endaligned
\ee
Here we used $U_\alpha =\Id-S_\alpha^2$, which commutes with $S_\alpha^2$.  In particular, $R_\alpha$ and $\mathcal Q_\alpha$ are non-negative.

\bse
Let
\be
\mathsf C=M-m\otimes m
\ee
be the covariance matrix of the reference domain.  It is positive definite because $\breve\Lambda$ is open.  Write $c=(c^\top,c^\perp)\in\RR^d\oplus\RR e_n$ and introduce
\be
\beta =\alpha^2c^\top,
\qquad
\gamma =\alpha c^\perp.
\ee
Using~\eqref{equa-S-small-cap-expansion} in~\eqref{equa-G-exact-remainder}, we obtain
\bel{equa-G-remainder-scaling}
c^tR_\alpha c
=|\mathsf C\beta|^2
+\operatorname{tr}(\mathsf C)(m\cdot\beta+\gamma)^2
+o(1)(|\beta|^2+|\gamma|^2).
\ee
The remainder is uniform on the unit sphere in the $(\beta,\gamma)$ variables.  Hence, for all sufficiently small $\alpha$,
\bel{equa-section5-R-two-sided}
c_0\bigl(\alpha^4|c^\top|^2+\alpha^2|c^\perp|^2\bigr)
\leq c^tR_\alpha c
\leq C_0\bigl(\alpha^4|c^\top|^2+\alpha^2|c^\perp|^2\bigr),
\ee
with constants depending only on the reference pair.
\ese

\vskip.2cm

\noindent{\it 6. Contradiction argument for the corrected estimate.}
Suppose that~\eqref{equa-corrected-second-scaling} fails.  After normalization, there are $\alpha_j\to0$ and functions $\nu_j$ such that
\bel{equa-G-contradiction-normalization}
\mathcal Q_{\alpha_j}[\nu_j]=1,
\qquad
\alpha_j\|\nablaslash^2\nu_j\|_{L^2_{- \expoP}}
\longrightarrow0.
\ee
Define $c_j$ and $w_j$ by~\eqref{equa-G-affine-decomposition}, subtract the average of $w_j$, and write
$w_j\circ\Phi_{\alpha_j}=\alpha_jh_j$. 
Equations~\eqref{equa-G-exact-remainder}, \eqref{equa-scaled-metric-formulas}, and~\eqref{equa-section5-R-two-sided} give
\be
\|Dh_j\|_{L^2(d\breve\mu)}
+|\beta_j|+|\gamma_j|\lesssim1,
\qquad
\beta_j=\alpha_j^2c_j^\top,
\quad
\gamma_j=\alpha_jc_j^\perp.
\ee
The condition $\langle\nablaslash w_j\rangle=0$, tested against the tangential and normal ambient coordinate directions, gives
\bel{equa-G-orthogonality-limits}
\langle Dh_j\rangle_0\longrightarrow0,
\qquad
\langle y\cdot Dh_j\rangle_0\longrightarrow0.
\ee
After extracting a subsequence, we may assume that
\be
\beta_j\to\beta,
\qquad
\gamma_j\to\gamma,
\qquad
h_j\rightharpoonup h
\quad\hbox{in }H^1(d\breve\mu).
\ee
\bse
The pullback Hessian identity reads
\be
\Phi_{\alpha_j}^*(\nablaslash^2\nu_j)_{ab}
=\alpha_j\Bigl(
\partial_{ab}h_j
- \Gamma^c_{\alpha_j,ab}\partial_ch_j
- \alpha_j(c_j\cdot\Phi_{\alpha_j})
(\delta_{ab}+\Obig(\alpha_j^2))
\Bigr),
\ee
where
\be
\Gamma^c_{\alpha,ab}
=\alpha^2y^c\delta_{ab}+\Obig(\alpha^4).
\ee
The second relation in~\eqref{equa-G-contradiction-normalization} therefore implies
\bel{equa-G-Hessian-limit}
D^2h_j-(\gamma_j+\beta_j\cdot y)I_d
\longrightarrow0
\quad\hbox{in }L^2(d\breve\mu).
\ee
The weak limit satisfies
\be
\partial_{ab}h=(\gamma+\beta\cdot y)\delta_{ab}
\quad\hbox{in the sense of distributions}.
\ee
\ese
\bse
Since $d=n-1\geq2$, choose $a\neq c$.  Differentiating the $(a,a)$ equation with respect to $y_c$ and commuting derivatives with the $(a,c)$ equation gives
\be
\beta_c=\partial_c\partial_{aa}h
=\partial_a\partial_{ac}h=0.
\ee
Thus $\beta =0$ and $Dh=\gamma y+a$ for some constant vector~$a$.  Passing to the limit in~\eqref{equa-G-orthogonality-limits}, we obtain
\be
a =- \gamma m,
\qquad
0=\langle y\cdot Dh\rangle_0
=\gamma(\operatorname{tr}M-|m|^2)
=\gamma\operatorname{tr}\mathsf C.
\ee
Since $\mathsf C$ is positive definite, $\gamma =0$ and $a =0$.  Every convergent subsequence of $(\beta_j,\gamma_j)$ therefore has zero limit, so $\beta_j,\gamma_j\to0$.
\ese

Equation~\eqref{equa-G-Hessian-limit} now gives $D^2h_j\to0$.  Applying~\eqref{equa-scaled-base-poincare} componentwise to $Dh_j$, together with the first relation in~\eqref{equa-G-orthogonality-limits}, yields $Dh_j\to0$.  Finally,~\eqref{equa-G-exact-remainder} and~\eqref{equa-section5-R-two-sided} imply
${
\mathcal Q_{\alpha_j}[\nu_j] \to 0}$ 
contradicting~\eqref{equa-G-contradiction-normalization}.  This proves~\eqref{equa-corrected-second-scaling} for all sufficiently small $\alpha$.  On every compact interval $\alpha_0\leq\alpha\leq1$, the pulled-back metrics and measures form a uniformly equivalent smooth family; the fixed-domain argument in Proposition~\ref{prop:three-weighted-poincare} then gives a uniform corrected constant.  Enlarging the implicit constant completes the proof for the whole range $0<\alpha\leq1$.


\compactpart{Decay structure toward the sphere at infinity}

\section{Shell stability for the localized Hamiltonian operator}
\label{section-6}

\subsection{Organization of the remainder of this paper}

There remains to introduce and establish shell stability conditions.  In \autoref{section-6}, we define the Hamiltonian shell functional and dissipation, and control fluctuation operators in terms of $\diam(\Lambda)$, $C_\Poin$, and $C_{\Poin2}$.  Provided these geometric constants are sufficiently small, this reduces shell stability to a semi-coercivity property of the dissipation.  In \autoref{section-8}, we derive the functionals explicitly and prove that its parameter can be chosen to ensure the necessary semi-coercivity.  In \autoref{section-7}, we formulate momentum shell stability and estimate the momentum fluctuation operator, using the weighted Korn and Korn--Poincar\'e inequalities already established in \autoref{section-4}.  Finally, \autoref{section-9} combines these inequalities with the weighted Hardy estimate to prove semi-coercivity of the momentum dissipation.  These four sections complete the proof of our main \autoref{art2 -thm:informal-suff-stab}, which confirms that, in small domains and with $p$~close to harmonic, the analytic hypotheses required by the optimal localization theorem for the Einstein constraints hold.

We retain the notation introduced in \eqref{equa-defineG} and~\eqref{equa-matrix-STU}: in short, these are five positive-definite matrices
\be
S^2=\la\xh\xh\ra , \quad
T=\Id+S^2 , \quad
U=\Id-S^2 , \quad
\matV=S^2-\la\xh\ra\la\xh\ra , \quad
\matG=U+2\matV .
\ee


\subsection{Two identities} 

\subsubsection{Shell functional}

We are interested in the radial decay of general solutions $u\colon\Omega_R \to \RR$ to 
\bel{equa-solutionH} 
\notreH^{\lambda}[u] = E, 
\ee
where $E\colon\Omega_R \to \RR$ is a given scalar field. Recall that the decomposition~\eqref{equa-key-decompose-H-repeat} separates the fourth-order operator into the purely radial fourth-order part $\Arr$, the mixed radial--spherical part $\Ars^\lambda$, and the purely spherical fourth-order part $\ssA^\lambda$. Let us introduce the following quadratic functional,
referred to as the \emph{Hamiltonian shell functional}, given by
\bel{eq-shell-func}
\aligned
\Phi^\notreH[u] 
& = \frac{1}{2(n-1)} \fint_{\Lambda_r} \Bigl(
(n-1)\bigl(\vartheta^2 u + a_{n,p} \vartheta u - \cstun u\bigr)^2
+ (n- 2) ( \Deltaslash u)^2
+ \bigl|\nablaslash^2 u\bigr|^2 \\
& \qquad\qquad
+ 2 (1+a_{n,p}+\cstun) \,  |\nablaslash u|^2
- 2 (c_{n,p} + (n- 2)\cstun) u \Deltaslash u
+ \cstdeux \, u^2
\Bigr) \, d\chi .
\endaligned
\ee
This form is derived in \autoref{section-8} starting from a general sum of squares Ansatz with $13$ free parameters and imposing cancellations that enable the present discussion to go through.

For sufficiently large $\cstdeux\geq 0$, the integrand is manifestly non-negative hence the functional is non-negative.
Specific values of $\cstun, \cstdeux$ are selected in \autoref{section-8}, which are uniform in $p\in(p_n^\flat,n-2)$ and are suitable for the class of domains with small aperture.  Other values are selected in the proof of stability of the isotropic localization function $\lambda\equiv 1$ in \cite[Theorem~4.10]{LL-optimal-main}.
We thus keep $\cstun,\cstdeux$ as free parameters at this stage.

The shell functional obeys a crucial dissipation property\footnote{The word ``dissipation'' is used here in a loose sense.} obtained by multiplying~\eqref{equa-solutionH} by
\bel{test-function-theta2u}
{1 \over n-1} r^4 \bigl( \vartheta^2 u + a_{n,p} \vartheta u - \cstun u \bigr) .
\ee

\begin{definition}
\label{def-Hamil-shell}
Consider the Hamiltonian operator. 
The second-order radial differential identity satisfied by the shell functional~\eqref{eq-shell-func}, 
derived in \cite{LL-optimal-main} for any solution $u\colon\Omega_R \to \RR$ to~\eqref{equa-solutionH}, is called the (Hamiltonian) \underline{shell functional identity}  and takes the form
\bel{main-func-identity}
- (\vartheta + a_{n,p}) (\vartheta + 2a_{n,p}) \Phi^\notreH[u]
+ \Chi^\notreH[u] = \Mu^\notreH[u,E]. 
\ee 
\end{definition}

\bse
\label{equa-condition-monotone}
Detailed expressions of the various functionals, which play no role in the present section, are postponed to the next section.
\bei
\item The \emph{radial differential operator} $(\vartheta + a_{n,p}) (\vartheta + 2a_{n,p})$ arises when seeking to ``bridge'' the variational decay rate $r^{-a_{n,p}/2}$ and the harmonic decay rate $r^{-a_{n,p}}$. 

\item The \emph{bare dissipation functional} $\Chi^\notreH$ is an integral functional in the variables $\vartheta^j \nablaslash^k u$ for $j+k\leq 4$ with $j \leq 3$ and $k \leq 2$. It features products of second-order and fourth-order derivatives of~$u$ (and no quadratic term in these latter derivatives) hence cannot enjoy positivity properties.

\item To deal with these linear terms, we introduce the two (shifted) \emph{dissipation functionals} 
\bel{equa-310c} 
\Psi^\notreH_{\beta}
\coloneqq \Chi^\notreH - (\vartheta +\beta) \Upsilon^\notreH, 
\qquad \beta \in \{ a_{n,p}, 2 a_{n,p}\}, 
\ee
in which we have subtracted a \emph{radial integration functional} $\Upsilon^\notreH[u]$. 
\autoref{art2 -lema-gammagamma} will state a semi-coercivity property of~$\Psi^\notreH_\beta[u]$.

\item The \emph{source functional} $\Mu^\notreH$ is the negative of the weighted product of the factor in~\eqref{test-function-theta2u} with the source~$E$, namely 
\bel{Psi-gendef} 
\Mu^\notreH[u,E] \coloneqq 
\frac{1}{n-1} \fint_{\Lambda_{r}}  \bigl(- \vartheta ( \vartheta + a_{n,p}) u  + \cstun u \bigr) \, r^4 E \, d\chi.  
\ee 
\eei
\ese


\subsubsection{Shell average}

The monotonicity and coercivity properties will typically hold only \emph{modulo averages}, and it is essential to consider also the radial evolution of the average of the solution. 
By contracting the Hamiltonian equation~\eqref{equa-solutionH} with an element of the co-kernel (namely $1$) and any non-trivial element of the kernel of the harmonic Hamiltonian operator, 
we find  
a fourth-order system of two {\it coupled} differential equations satisfied by the averages $\la u\ra$ and $\la\Deltaslash u\ra$, which can be integrated twice. Specifically, using the notation 
\be
\vartheta = r \, \del_r
\ee
and performing a radial integration twice, we find a second-order equation (cf.~ \cite{LL-optimal-main}). 

\begin{definition}
\label{def-shellH}
Consider the Hamiltonian operator. 
The identity obtained in~\cite{LL-optimal-main} by contracting~\eqref{equa-solutionH} with the relevant kernel and cokernel elements and then integrating the resulting radial system twice is called the (Hamiltonian) \underline{shell average identity} and takes the form
\bel{equa-defb-b-b}
\bigl( - \bnotreH_{1} \vartheta(\vartheta+a_{n,p}) + \bnotreH_{0} \bigr) \, \la u\ra 
= \bigl( -(n- 2) \vartheta + {c_{n,p} / a_{n,p}} \bigr) \Kappa^\notreH[\ut]
+ \Nu^\notreH[E] + C^u r^{-a_{n,p}}, 
\ee
where the constants $\bnotreH_{1}$ and $\bnotreH_{0}$ can be computed explicitly.
\end{definition}

 In particular, the radial stability condition can be reformulated as 
\bel{equa-b2 -positive}
\bnotreH_{1} \, \bnotreH_{0}>0.  
\ee
We distinguish the following objects.

\bei 

\item The \emph{fluctuation operator} $\Kappa^\notreH$ is a linear functional consisting of an integral over $\Lambda_{r}$.   
Interestingly, this operator is bilinear in the \emph{fluctuations} of a solution $u$, defined as 
\bel{equa-tildenotation}
\ut \coloneqq u - \la u \ra ,
\ee
and in the fluctuations $\nut^\normal$ of~$\nu^\normal$.
Thanks to a favorable bounds on~$\nut^\normal$, $\Kappa^\notreH$ can be thought of as a {\it lower-order term.}
For instance, it vanishes identically when $\Lambda$ is the whole sphere and $\lambda$ is taken to be constant.

\item The \emph{source term} $\Nu^\notreH[E]$ is an integral operator acting on the function $E$. 
The constant $C^u$ in the harmonic term depends on the solution $u$.

\eei   
  

\subsection{Shell stability}
\label{section=6.3}

\subsubsection{Main definition}

With a slight abuse of notation we introduce the norm 
\bel{equa-normsH}
\bigl( \norm{u}^\notreH \bigr)^2
\coloneqq \| \vartheta^2 u\|^2_{L^2_{- \expoP}(\Lambda_{r})} + \| \vartheta u\|^2_{H^1_{- \expoP}(\Lambda_{r})} + \| u\|^2_{H^2_{- \expoP}(\Lambda_{r})} 
\ee
of $u$ and its radial derivatives on each spherical shell $\Lambda_{r}$ of each asymptotic end.
We often simplify the notation and do not specify the variable $r$ explicitly.  The third stability condition involves a \emph{radial Hardy constant} denoted by~$c^{\notreH}_{\textnormal{radial}}$ (and explicited in~\eqref{def-cnotreHradial}) whose necessity arises from an interplay of the Hamiltonian shell identity~\eqref{main-func-identity} and the ODE~\eqref{equa-defb-b-b} for the average $\la u\ra$.
It is a constant arising in a weighted radial Hardy-type inequality, derived earlier in~\cite{LL-optimal-main}. 

\begin{definition}
\label{def-shell-Hstab}  
Consider the Hamiltonian operator. 
A localization function $\lambda$ on $\Sbb^{n-1}$ satisfying the Hamiltonian harmonic and radial stability conditions is said to satisfy the Hamiltonian \underline{shell stability condition} if the decomposition \eqref{main-func-identity} of the  functional \eqref{eq-shell-func} enjoys the following properties.
\bse
\label{equa-last-twoH} 
\bei 

\item \emph{Continuity of the shell functionals.} For any scalar field $u\colon\Omega_R \to \RR$ and each spherical shell~$\Lambda$, and for $\beta\in\{a_{n,p},2a_{n,p}\}$, one has
\bel{equa-conditionH}
\aligned
0 \leq \Phi^\notreH[u] & \lesssim \bigl( \norm{u}^\notreH\bigr)^2 ,
\qquad
\\
\bigl|\Psi^\notreH_{\beta}[u]\bigr|
& \lesssim \bigl( \norm{\vartheta u}^\notreH \bigr)^2 + \bigl( \norm{u}^\notreH \bigr)^2 .
\endaligned
\ee

\item \emph{Semi-coercivity of the shell dissipation.}
There exists a constant $\gamma^\notreH_{\textnormal{shell}} >0$ such that for $\beta\in\{a_{n,p},2a_{n,p}\}$, any scalar field $u\colon\Omega_R \to \RR$, and any spherical shell~$\Lambda =\Lambda_r$, one has
\begin{align}
\label{equa-conditionH2}
\Psi^\notreH_{\beta}[u]
+ \gamma^\notreH_{\textnormal{shell}} \,  \Bigl( \la u\ra^2
- c^{\notreH}_{\textnormal{radial}} \bigl( \Kappa^\notreH[\ut] \bigr)^2\Bigr) 
& \gtrsim
\bigl( \norm{\vartheta u}^\notreH \bigr)^2 + \bigl( \norm{u}^\notreH \bigr)^2 .
\end{align}
\eei 
\ese
The localization function is called \emph{Hamiltonian-stable} when the conditions \eqref{equa-H- 210}, \eqref{equa-H211}, and \eqref{equa-last-twoH} hold. 
\end{definition}

An interesting observation is that shell stability implies harmonic stability. Yet, in our presentation we found it convenient to introduce harmonic stability separately first.


\subsubsection{Fundamental estimate of the dissipation functional}

The continuity conditions~\eqref{equa-conditionH} are immediate given the explicit expressions of the functionals, so the key condition is the semi-coercivity property~\eqref{equa-conditionH2}.
Our analysis of the Hamiltonian shell functional (cf.~\autoref{section-8}) will provide us with the following structure.
Importantly, all constants are \emph{uniform} in the harmonic limit $p\to n- 2$, but not necessarily when $p$ approaches its lower bound $p\to p_n^\flat$.
Since the latter limit is of no interest to us, we shall omit the $p$-dependence of the constants $\gamma^{(m)}_n$: by restricting for example $p$ to the subinterval $[(n-2)/2, n-2)$ these constants can be taken to only depend on the dimension~$n$.
We refer to~\autoref{section-8} for a proof of \autoref{art2 -lema-gammagamma}.

\begin{proposition}[Semi-coercivity of the Hamiltonian dissipation functional]
\label{art2 -lema-gammagamma}
For a projection exponent $p\in(p_n^\flat, n-2)$, there exist positive constants $\cstun,\cstdeux$ and $\onegamma,\twogamma,\threegamma$ such that in a given asymptotic end
the dissipation functionals~$\Psi^\notreH_\beta$ defined in~\eqref{equa-310c}
 obey the semi-coercivity property
\bel{art2 -equa-98f0}
\Psi^\notreH_\beta[u]
\geq \onegamma \bigl(\norm[\big]{\vartheta u}^\notreH\bigr)^2
+ \twogamma \bigl(\norm[\big]{u}^\notreH\bigr)^2
-\threegamma\fint_{\Lambda}u^2\,d\chi
\qquad \beta\in\{a_{n,p},2a_{n,p}\}.
\ee 
\end{proposition}


\subsubsection{Main result for this section} 

The remainder of this section is devoted to deducing from \autoref{art2 -lema-gammagamma} the following result, which is part of \autoref{art2 -thm:informal-suff-stab}.
The proof relies on a detailed control of the fluctuation term~$\Kappa^\notreH[\ut]$.
For the small-domain family, the required estimate on $\CPoinTwo$ is supplied by \autoref{lem:Poincare-scaling-complete}. In addition, remarkably, the upper bound in~\eqref{art2 -equa-hypopo} only depends on the dimension~$n$.

\begin{theorem}[Shell stability for the localized Hamiltonian operator]
\label{art2 -prop-dj39}
Fix $p\in(p_n^\flat,n- 2)$ and consider the Hamiltonian operator. Provided the two weighted Poincaré constants and the diameter of~$\Lambda$ are sufficiently small in the sense that
\bel{art2 -equa-hypopo} 
\CPoin 
+ \CPoinTwo 
+ \diam(\Lambda) \leq c_2^\notreH(n)
\ee
for a sufficiently small positive constant depending only on~$n$, then the semi-coercivity property~\eqref{equa-conditionH2} (in \autoref{def-shell-Hstab}) holds. 
\end{theorem}


\subsection{Derivation of the Hamiltonian shell stability}
\label{art2 -section=2.3}

\subsubsection{Basic estimate of the fluctuation operator}

The fluctuation operator is given in \cite[Definition 6.5]{LL-optimal-main}, namely,
\be
\aligned
\Kappa^\notreH[\ut]  
& \coloneqq \fint_{\Lambda} \Bigl(
(n-1)(\vartheta+a_{n,p}) \bigl(\vartheta^2+a_{n,p}\vartheta-b_{n,p}\bigr) (\nut^\normal \ut)
\\[-1ex]
& \qquad\quad + (\vartheta+a_{n,p})
\bigl(
- 2 \, \nablaslash \ut \cdot \nablaslash \nut^\normal
+ (n-2) ( \nut^\normal \Deltaslash \ut + \ut \, \Deltaslash\nut^\normal)
\bigr)
\\[-.5ex]
& \qquad\quad + \bigl( (n-2)(n-2-a_{n,p}) + 1\bigr) \bigl( \ut \, \Deltaslash\nut^\normal - \nut^\normal \, \Deltaslash\ut\bigr) \Bigr) d\chi .
\endaligned
\ee
This explicit expression reflects some fundamental structure of the Einstein operator.  Our estimate in terms of the fluctuations of the silhouette function~$\nu^\normal$ (\autoref{art2 -prop-k-fluctuations}) will be complemented next by $H^2$ estimates on~$\nut^\normal$ in \autoref{art2 -lem-nut-small}, together with a more detailed treatment of one term using the second-order Poincar\'e inequality to reach a more convenient control of fluctuations (\autoref{art2 -lem-kappanotreh-control}).  The constant~$c^{\notreH}_{\textnormal{radial}}$ is then bounded to complete the proof of shell stability.

\begin{proposition}[Hamiltonian fluctuation operator vs.~Hamiltonian silhouette function] 
\label{art2 -prop-k-fluctuations}
For a given asymptotic end, the Hamiltonian fluctuation operator is bounded as follows:
\bel{art2 -equa-bound-kernel}
\aligned
\bigl| \Kappa^\notreH[\ut]  \bigr| 
& \lesssim
\|  \nut^\normal \|_{L^2_{- \expoP}(\Lambda)} \Bigl( \norm{\vartheta\ut}^\notreH + \norm{\ut}^\notreH \Bigr)
+ \|  \nablaslash \nut^\normal \|_{L^2_{- \expoP}(\Lambda)} \hskip-.1cm \sum_{m=0,1} \bigl\| \vartheta^m\nablaslash \ut \bigr\|_{L^2_{- \expoP}(\Lambda)} 
\\
& \quad + \bigl\|{\Deltaslash\nut^\normal} \bigr\|_{L^2_{- \expoP}(\Lambda)} \hskip-.1cm \sum_{m=0,1} \| \vartheta^m \ut \bigr\|_{L^2_{- \expoP}(\Lambda)}, 
\endaligned
\ee 
in which the implied constant depends upon the dimension~$n$, only. 
\end{proposition} 


\begin{proof} 
In view of the expression of the operator $\Kappa^\notreH$ 
we find 
\be 
\aligned
\bigl| \Kappa^\notreH[\ut]  \bigr|
& \lesssim \|  \nut^\normal \|_{L^2_{- \expoP}(\Lambda)} \Bigl(
\bigl\| (\vartheta+a_{n,p}) \bigl(\vartheta^2+a_{n,p}\vartheta-b_{n,p}\bigr) \ut \|_{L^2_{- \expoP}(\Lambda)} \\
& \qquad\qquad\qquad\quad + \bigl\| (\vartheta+a_{n,p}) \Deltaslash \ut \bigr\|_{L^2_{- \expoP}(\Lambda)}
+ \bigl\| \Deltaslash \ut \bigr\|_{L^2_{- \expoP}(\Lambda)} \Bigr) \\
& \quad + \|  \nablaslash \nut^\normal \|_{L^2_{- \expoP}(\Lambda)} \, \bigl\|   (\vartheta+a_{n,p}) \nablaslash \ut   \bigr\|_{L^2_{- \expoP}(\Lambda)} 
\\
& \quad 
+ \bigl\| \Deltaslash\nut^\normal\bigr\|_{L^2_{- \expoP}(\Lambda)} \, \bigl\| \bigl( (n- 2) \vartheta + (n- 2)^2 + 1\bigr) \ut \bigr\|_{L^2_{- \expoP}(\Lambda)}, 
\endaligned
\ee
where the constants only depend on the coefficients $a_{n,p},b_{n,p},c_{n,p}/a_{n,p}$, all of which are bounded uniformly in $p\in(p_n^\flat,n-2)$. Finally, we use the notation~\eqref{equa-normsH} to obtain~\eqref{art2 -equa-bound-kernel}. 
\end{proof} 


\subsubsection{Improved control of the fluctuation operator}

We now leverage \autoref{art2 -lem-nut-small} and the second-order Poincaré inequality to simplify the bound on $\Kappa^\notreH[\ut]$.

\begin{proposition}[Control of the Hamiltonian fluctuation operator]
\label{art2 -lem-kappanotreh-control}
Provided the Poincaré constant~$\CPoin$ satisfies the radial-stability condition \eqref{art2 -art2 -equa-4178-bis}, namely $\CPoin < \beta_{n,p}$,
one has the bound
\be
\bigl| \Kappa^\notreH[\ut]  \bigr|
\lesssim c_{n,p} \bigl(\CPoin + \CPoinTwo + \diam(\Lambda)\bigr) |\la\nu^\normal\ra| \bigl( \norm{\vartheta\ut}^\notreH + \norm{\ut}^\notreH \bigr) .
\ee
\end{proposition}

\begin{proof}
The upper bound $\CPoin < \beta_{n,p}$ ensures that \autoref{art2 -lem-nut-small} applies, and provides an $L^2_{-\expoP}$ control of $\nablaslash^2\nut^\normal$ and $\nablaslash\nut^\normal$ by the average $|\la\nu^\normal\ra|$.  The Poincaré inequality then also gives an $L^2_{-\expoP}$ control of~$\nut^\normal$.
Thus, the bound in \autoref{art2 -prop-k-fluctuations} easily yields
\be
\bigl| \Kappa^\notreH[\ut]  \bigr| 
\lesssim c_{n,p} (1 + \CPoin) |\la\nu^\normal\ra| \bigl( \norm{\vartheta\ut}^\notreH + \norm{\ut}^\notreH \bigr) .
\ee
Our goal is to improve the bound in the limit of small $\CPoin + \CPoinTwo + \diam(\Lambda)$.  In particular, we will use the smallness of $\diam(\Lambda)$ to our advantage.

We return to the bound in \autoref{art2 -prop-k-fluctuations} in terms of ($L^2$ norms of) derivatives of $\nut^\normal$ and of $\ut$.
When combined with the (first-order) Poincaré inequality~\eqref{equa-Poin1} for $\nut^\normal$ and for~$\ut$, the bound implies
\be
\bigl| \Kappa^\notreH[\ut]  \bigr| 
\lesssim
\Bigl( \CPoin \bigl\| \nablaslash^2\nut^\normal, \nablaslash\nut^\normal \bigr\|_{L^2_{- \expoP}(\Lambda)}
+ \|  \nablaslash \nut^\normal \|_{L^2_{- \expoP}(\Lambda)} \Bigr) \Bigl( \norm{\vartheta\ut}^\notreH + \norm{\ut}^\notreH \Bigr) .
\ee
The first term features the favorable factor $\CPoin$ and we have a good control of up-to-second-order derivatives $\| \nablaslash^2\nut^\normal, \nablaslash\nut^\normal\|_{L^2_{- \expoP}(\Lambda)}$ by the average $|\la\nu^\normal\ra|$ thanks to \autoref{art2 -lem-nut-small}.  There remains to find an improved control of the second term $\nablaslash\nut^\normal$ that would include a favorable geometrical factor analogous to~$\CPoin$.
The key object is the \emph{second Poincaré inequality}~\eqref{equa-Poin2} proven in \autoref{section-5}, namely
\bel{Poin-deux-again}
\| \nablaslash\nut^\normal \|_{L^2_{- \expoP}(\Lambda)}^2
\leq  (\CPoinTwo)^2 \, \| \nablaslash^2\nut^\normal \|_{L^2_{- \expoP}(\Lambda)}^2 + |\la\nablaslash_\bullet\nut^\normal\ra|_{\matG}^2
\ee
for some positive constant $\CPoinTwo$. We are thus left with controlling the averages $\la\nablaslash_\bullet\nut^\normal\ra$.

Domains with $\diam(\Lambda)<\pi/2$ lie in some hemisphere hence have $\la\xh\ra\neq0$.
In such small domains, the matrix $\matG=\Id+\la\xh\xh\ra-2\la\xh\ra\la\xh\ra$ in~\eqref{Poin-deux-again} is quadratically close to the projector $\Pi^{\perp}_{\la\xh\ra}$ orthogonal to~$\la\xh_\bullet\ra$.  More precisely, in terms of $\matV=\la\xh\xh\ra-\la\xh\ra\la\xh\ra$ and its trace $\operatorname{tr}\matV = 1-|\la\xh\ra|^2$,
\be
\aligned
\bigl|(\matG - \Pi^{\perp}_{\la\xh\ra})_{kl}\bigr|
&=\Bigl|
\matV_{kl}
+\frac{\operatorname{tr}\matV}{|\la\xh\ra|^2}
\la\xh_k\ra\la\xh_l\ra\Bigr|
\\
&\leq \bigl|\la(\xh_k- \la\xh_k\ra)(\xh_l- \la\xh_l\ra)\ra\bigr|
+ \frac{|\la\xh_k\ra\la\xh_l\ra|}{|\la\xh\ra|^2}
\la|\xh- \la\xh\ra|^2\ra
\leq 2 \, \diam(\Lambda)^2 .
\endaligned
\ee 
In particular, for every $\wvect\in\RR^n$,
\be
\wvect^t\matG\wvect
\leq |\Pi^{\perp}_{\la\xh\ra}\wvect|^2
+C \, \diam(\Lambda)^2|\wvect|^2.
\ee
Applied to $\wvect=\la\nablaslash\nut^\normal\ra$, the last term is bounded by $C \, \diam(\Lambda)^2\|\nablaslash\nut^\normal\|_{L^2_{- \expoP}}^2$ and can be absorbed into the left-hand side of the second Poincaré inequality when the domain is sufficiently small.
We are thus interested in controlling $\Pi^{\perp}_{\la\xh\ra}\la\nablaslash_\bullet\nut^\normal\ra$, which are $n-1$ independent averages of~$\nablaslash\nut^\normal$.

For this we return to the equation $\ssA^\lambda[\nu^\normal]=0$ satisfied by~$\nu^\normal$.
Let us consider the integral
\be
\aligned
0 & = \fint_{\Lambda} \xh_m \ssA^\lambda[\nu^\normal] d\chi
\\
& = \fint_{\Lambda} \Bigl( (n- 2) \Deltaslash\xh_m \Deltaslash\nu^\normal + \nablaslash^2 \xh_m \cdot \nablaslash^2\nu^\normal + 2 (1+a_{n,p})\nablaslash\xh_m\cdot\nablaslash\nu^\normal - c_{n,p} (\Deltaslash\xh_m)\nu^\normal \Bigr) d\chi
\\
& = 2(1+a_{n,p}) \la\nablaslash_m\nut^\normal\ra - (n^2 -3n+3) \la\xh_m\Deltaslash\nut^\normal\ra + (n-1) c_{n,p} \la\xh_m\nu^\normal\ra .
\endaligned
\ee
Here we used that $\Deltaslash \xh_m = -(n-1)\xh_m$ and $\nablaslash^2 \xh_m = - \xh_m \, \gslash$, hence we have
$\nablaslash^2 \xh_m \cdot \nablaslash^2\nu^\normal = - \xh_m \Deltaslash\nu^\normal$, and we also replaced $\nu^\normal$ by $\nut^\normal$ in the terms involving $\nablaslash\nu^\normal$ and $\Deltaslash\nu^\normal$, since the averages vanish under these derivatives.

In a small domain, the factors $\xh_m$ are close to their average~$\la\xh_m\ra$.  Specifically, we deduce
\be
\aligned
& \Bigl| 2(1+a_{n,p}) \Pi^{\perp}_{\la\xh\ra\,lm} \la\nablaslash_m\nut^\normal\ra \Bigr|
\\
& = \Bigl| \Pi^{\perp}_{\la\xh\ra\,lm}\Bigl( (n^2 -3n+3) \bigl\la(\xh_m- \la\xh_m\ra)\Deltaslash\nut^\normal\bigr\ra - (n-1) c_{n,p} \bigl\la(\xh_m- \la\xh_m\ra)\nu^\normal\bigr\ra \Bigr) \Bigr|
\\
& \lesssim 
\diam(\Lambda) \Bigl( \|\Deltaslash\nut^\normal\|_{L^2_{- \expoP}(\Lambda)} + c_{n,p} \|\nu^\normal\|_{L^2_{- \expoP}(\Lambda)} \Bigr) ,
\endaligned
\ee
where we used Cauchy--Schwarz and the elementary bound $|\xh_m- \la\xh_m\ra|\lesssim \diam(\Lambda)$ on $\Lambda$. We were able to subtract the average $\la\xh_m\ra$ in both terms because of the projection~$\Pi^{\perp}_{\la\xh\ra}$. 


Altogether, the Poincaré inequality implies
\be
\| \nablaslash\nut^\normal \|_{L^2_{- \expoP}(\Lambda)}
\lesssim  \Bigl( \CPoinTwo + \diam(\Lambda) \Bigr) \, \| \nablaslash^2\nut^\normal \|_{L^2_{- \expoP}(\Lambda)}
+ c_{n,p} \diam(\Lambda) \|\nu^\normal\|_{L^2_{- \expoP}(\Lambda)} ,
\ee
and the bound on $\Kappa^\notreH$ becomes
\be
\aligned
\bigl| \Kappa^\notreH[\ut]  \bigr| 
\lesssim
\Bigl( \bigl(\CPoin + \CPoinTwo 
& + \diam(\Lambda)\bigr) \bigl\| \nablaslash^2\nut^\normal, \nablaslash\nut^\normal \bigr\|_{L^2_{- \expoP}(\Lambda)}
\\
& + c_{n,p} \diam(\Lambda) |\la\nu^\normal\ra| \Bigr) \Bigl( \norm{\vartheta\ut}^\notreH + \norm{\ut}^\notreH \Bigr) .
\endaligned
\ee
Using \autoref{art2 -lem-nut-small} to bound the $\dot{H}^2$ and $\dot{H}^1$ norms of $\nut^\normal$ by $c_{n,p}$ times the average $|\la\nu^\normal\ra|$, and recalling that $\|\nu^\normal\|_{L^2_{- \expoP}(\Lambda)} \leq \|\nut^\normal\|_{L^2_{- \expoP}(\Lambda)} + |\la\nu^\normal\ra|$, 
we reach the desired bound (with an implicit constant that depends on the dimension $n$, only).
\end{proof}


\subsubsection{Proof of \autoref{art2 -prop-dj39}}

\noindent{\it 1. A choice of constant.}
Our starting point is the semi-coercivity in \autoref{art2 -lema-gammagamma}, with its dimension-dependent constants~$\gamma^{(m)}_n$.  If the Poincaré constant is small enough in the sense that
\be
(\CPoin)^2 \leq (\twogamma - 1) / \threegamma ,
\ee
then by setting $\gamma^\notreH_{\textnormal{shell}} = \threegamma$ and applying the Poincaré inequality one gets
\bel{coer-629}
\Psi^\notreH_{\beta}[u]
+ \gamma^\notreH_{\textnormal{shell}} \,  \la u\ra^2
\geq \onegamma \bigl(\norm[\big]{\vartheta u}^\notreH\bigr)^2
+ \bigl(\norm[\big]{u}^\notreH\bigr)^2
\qquad \beta\in\{a_{n,p},2a_{n,p}\}.
\ee

\medskip
\bse
\noindent{\it 2. Control of the radial Hardy constant.}
We now bound the constant~$c^{\notreH}_{\textnormal{radial}}$, explicitly given in~\eqref{def-cnotreHradial}, to complement the control of the fluctuation operator (\autoref{art2 -lem-kappanotreh-control}), and the semi-coercivity of the dissipation (\autoref{art2 -lema-gammagamma}).

For a sufficiently small Poincaré constant, and specifically $\CPoin < \beta_{n,p}$, \autoref{art2 -prop--radial} states
\be
0 \leq \frac{\la\Deltaslash\nut^\normal\ra}{\la\nu^\normal\ra} < d^{\min}_{n,p}
= \min \Bigl(\frac{(n-1)c_{n,p}}{(n- 2)^2},\, d_{n,p}\Bigr) ,
\ee
which is enough to bound the exponents $\bnotreH_{1},\bnotreH_{0}$ appearing in the shell average identity~\eqref{equa-defb-b-b} and explicited in~\eqref{bH10}:
\bel{bnotre-ov-lanu}
0 < \frac{\bnotreH_{1}}{\la\nu^\normal\ra} \leq (n-1) c_{n,p} , \qquad
0 < \frac{\bnotreH_{0}}{\la\nu^\normal\ra} \leq (n-1) c_{n,p} b_{n,p} .
\ee

In fact, a stronger bound on $\la\Deltaslash\nut^\normal\ra/\la\nu^\normal\ra$ is provided by \autoref{art2 -lem-nut-small}:
\be
0 \leq \frac{\la\Deltaslash\nut^\normal\ra}{\la\nu^\normal\ra} \leq \frac{(n-1)c_{n,p}}{(n^2 -3n+3)K_{1\Poin}} ,
\qquad
K_{1\Poin} = 1 - (\tcnp)^2 (\CPoin)^2 .
\ee
This upper bound can be made arbitrarily close to $(n-1)c_{n,p}/(n^2 -3n+3)$ for sufficiently small~$\CPoin$, which leads to a positive lower bounds in~\eqref{bnotre-ov-lanu}.
For instance, for $\CPoin < \beta_{n,p} / \sqrt{n-1}$, one checks that $\la\Deltaslash\nut^\normal\ra / \la\nu^\normal\ra \leq c_{n,p}/(n-2)$, which leads to
\be
c_{n,p} \leq \frac{\bnotreH_{1}}{\la\nu^\normal\ra} \leq (n-1) c_{n,p} ,
\ee
and hence $\bnotreH_{1}/\la\nu^\normal\ra\simeq c_{n,p}$ uniformly in~$p$.  Likewise $\bnotreH_{0}/\la\nu^\normal\ra\simeq c_{n,p}$ uniformly in~$p$.

Next, the characteristic exponents of the differential operator $-\bnotreH_{1}\vartheta(\vartheta+a_{n,p})+\bnotreH_{0}$ are solutions of $\beta_- + \beta_+ = a_{n,p}$ and $-\beta_- \beta_+= \bnotreH_{0}/\bnotreH_{1}$, and this ratio obeys $\bnotreH_{0}/\bnotreH_{1}\simeq 1$ uniformly in~$p$.
Together these imply that $-\beta_- \simeq 1$ and $\beta_+ \simeq 1$.
The explicit expressions \eqref{equa-cplusminus} of $c_\pm$ and \eqref{def-cnotreHradial} of $c^{\notreH}_{\textnormal{radial}}$ then obey
\be
c_\pm \lesssim \frac{1}{a_{n,p}\la\nu^\normal\ra} , \qquad 
0 < c^{\notreH}_{\textnormal{radial}} \lesssim (c_\pm)^2 \lesssim \frac{1}{a_{n,p}^2\la\nu^\normal\ra^2} .
\ee
\ese

\medskip

\noindent{\it 3. Control of the fluctuation term.}
Once we assemble these results together with the bound in \autoref{art2 -lem-kappanotreh-control} on $\Kappa^\notreH$, the factors $c_{n,p}^2$ and $|\la\nu^\normal\ra|^2$ cancel out, and we obtain the following bound on the fluctuation term in~\eqref{equa-conditionH2} in \autoref{def-shell-Hstab}:
\be
c^{\notreH}_{\textnormal{radial}} \bigl( \Kappa^\notreH[\ut] \bigr)^2
\lesssim \bigl(\CPoin + \CPoinTwo + \diam(\Lambda)\bigr)^2 \Bigl( \norm{\vartheta\ut}^\notreH + \norm{\ut}^\notreH \Bigr)^2 ,
\ee
where the implicit constant depends on the dimension, only.
This is precisely controlled by the right-hand side of~\eqref{coer-629}, provided the geometric constants $\CPoin + \CPoinTwo + \diam(\Lambda)$ are smaller than some dimension-dependent constant, as stated in~\eqref{art2 -equa-hypopo}.
Altogether, in this regime, the Hamiltonian shell stability condition~\eqref{equa-conditionH2} holds.


\section{Shell stability for the localized momentum operator}
\label{section-7}

\subsection{Shell stability}

The following matrices are introduced (or repeated) at this point because they enter the momentum shell-stability functional and its proof in \autoref{section-9}:
\bel{equa-toutes-matrices}
\aligned
S^2 = \bigl((S^2)_{kl}\bigr) & \coloneqq (\la\xh_k\xh_l\ra) ,
& T = \Id + S^2 & \coloneqq  (\delta_{kl} + \la \xh_k\xh_l\ra),
\\
(\matV_{kl}) & \coloneqq  (\la \xh_k\xh_l\ra - \la\xh_k\ra\la\xh_l\ra),
& U = \Id - S^2 & \coloneqq  (\delta_{kl} - \la \xh_k\xh_l\ra),
\\
\matQ^{(j)k} & = \bigl\la 2 \xh_l \xi^{\normal (j) \perp} + \xi^{\normal (j) \parallel}{}_l\bigr\ra (T^{-1})^{lk} ,
& \zeroXi^{(i)}{}_k & \coloneqq \la \nablaslash_k \xi^{\normal (i)\perp} - \xi^{\normal (i)\parallel}{}_k\ra ,
\\
(\Xi^\notreM)^j_{l}
& \coloneqq 
\bigl\la - \nablaslash_l \xi^{\normal (j) \perp} + 2 a_{n,p} \xh_l\, \xi^{\normal (j) \perp} \bigr\ra
+(1+a_{n,p}) \la \xi^{\normal (j) \parallel}{}_l \ra. \mspace{-130mu}
\endaligned
\ee


\subsubsection{Main functional on spherical shells}
 
Our strategy is similar to the one for the Hamiltonian operator, but the structure of the momentum operator is different: it has a tensorial nature and only up to second-order derivatives.
We establish first the shell identity (in~\eqref{equ-shell-momen}, below) and the associated decomposition 
and then investigate the (semi-)coercivity properties of the relevant functionals. 
We introduce the \textbf{momentum shell functional} (as we call it) 
\bel{equa-quada-M}
\Phi^\notreM[Z] \coloneqq \frac{1}{2} \fint_{\Lambda} \bigl( 2 \, \Zperp{}^2 + |\Zpar|^2 \bigr) d\chi, 
\ee 
which obeys the \textbf{momentum shell identity} for the momentum
\bse
\label{equ-shell-momen}
\bel{main-func-identity-MM}
- (\vartheta + a_{n,p}) (\vartheta + 2a_{n,p}) \Phi^\notreM[Z]
+ \Chi^\notreM[Z] = \Mu^\notreM[Z].
\ee 
It involves the following terms, made explicit below. 

\bei 

\item  The {\bf bare dissipation functional} $\Chi^\notreM[Z]$ is an integral functional in the variables $Z$ and $\vartheta Z$, together with $\nablaslash Z$ and $\vartheta\nablaslash\Zperp$. Since it features $\vartheta\nablaslash\Zperp$ {\it linearly} (and not quadratically), it cannot enjoy positivity properties.
 
\item To eliminate these linear terms that would prevent the functionals from being coercive, we introduce the (shifted) {\bf dissipation functionals}  (for $\beta\in\{a_{n,p}, 2 a_{n,p}\}$)
\bel{equa-310c-MM} 
\Psi^\notreM_{\beta\iota} \coloneqq \Chi^\notreM - (\vartheta +\beta) \Upsilon^\notreM, 
\ee
in which we have subtracted a {\bf radial integration functional} $\Upsilon^\notreM[Z]$ in the variables $Z, \nablaslash \Zperp$. 

\item The {\bf source functional} $\Mu^\notreM$ is
\be
\Mu^\notreM[Z] 
\coloneqq 2 r^2 \fint_{\Lambda} Z\cdot F \, d\chi = \fint_{\Lambda} 2 \, \bigl( \Zperp \Fperp + \Zpar\cdot\Fpar \bigr) r^2 d\chi. 
\ee 
\eei 
\ese


\subsubsection{Stability condition}

From now on, we simplify the notation and no longer specify the variable $r$ explicitly.  As in the Hamiltonian case, continuity of the shell functionals in the following definition is immediate from~\eqref{equa-quada-M} and the construction of~$\Psi^\notreM$; it is included for clarity. We emphasize that the averages $\la 2\xh_l\Zperp+\Zpar_l\ra$ and the fluctuation operators $\Kappa^{\notreM(j)}[Z]$ in \eqref{equa-conditionM2} appear in the shell stability condition very differently from the Hamiltonian case.
Another major difference is that, due to infinite-dimensional space of conformal Killing vectors on the sphere~$\Sphe^2$, the lowest dimension $n=3$ is singled-out.
For this case, we introduce in~\eqref{def-Korn-loss} below a \emph{Korn loss functional} $\mathfrak{K}[Z]$ that measures the failure of an angular Korn inequality for $n=3$.

\begin{definition}
\label{def-shell-Mstab}   
Consider the momentum operator. 
A localization function satisfying the (momentum) harmonic and radial stability conditions \eqref{equa-stable-M-414} and~\eqref{equa-Xi-invertible} is said to satisfy the momentum \underline{shell stability condition} if the identity~\eqref{equ-shell-momen} enjoys the following properties.
\bse\label{equa-last-twoM}
\bei 
 
\item {\bf Continuity of the shell functionals.}
For any vector field $Z\colon\Omega_R\to\RR^n$ and any spherical shell $\Lambda$, and for $\beta\in\{a_{n,p},2a_{n,p}\}$, one has
\bel{equa-shell-M-coer}
0 \leq \Phi^\notreM[Z] \lesssim \|Z\|_{L^2_{- \expoP}(\Lambda)}^2 ,
\qquad
\bigl| \Psi^\notreM_{\beta}[Z] \bigr| \lesssim \|\vartheta Z\|_{L^2_{- \expoP}(\Lambda)}^2 + \|Z\|_{H^1_{- \expoP}(\Lambda)}^2 ,
\ee 

\item {\bf Semi-coercivity of the shell dissipation.}
There exists a Korn coefficient $\gamma_{\textnormal{Korn}}^{\notreM}>0$, only necessary in dimension $n=3$,
and there exists a positive shell coefficient\footnote{In Theorem~\ref{art2 -prop-semi-mdiss} below, we may take $\gamma_{\textnormal{shell}}^{\notreM}=a_{n,p}$.} $\gamma_{\textnormal{shell}}^{\notreM}>0$  such that for $\beta\in\{a_{n,p},2a_{n,p}\}$, for any vector field $Z\colon\Omega_R\to\RR^n$ and each spherical shell $\Lambda$, one has
\bel{equa-conditionM2}
\aligned
& \Psi^\notreM_{\beta}[Z] 
+ \gamma_{\textnormal{shell}}^{\notreM} \, \Bigl| \la 2 \,\xh \Zperp + \Zpar \ra - T(\Xi^\notreM)^{-1} \Kappa^\notreM[Z] \Bigr|^2
+ \delta_{n=3}\, \gamma_{\textnormal{Korn}}^{\notreM} \, \mathfrak{K}[Z]
\\
& \quad
\gtrsim \bigl\|\vartheta\Zperp,\, (\vartheta\Zpar+\nablaslash\Zperp),\, \Zperp\bigr\|_{L^2_{- \expoP}(\Lambda)}^2 + \bigl\|\Zpar\bigr\|_{H^1_{- \expoP}(\Lambda)}^2 .
\endaligned
\ee 
\eei 
\ese
The localization function is called \underline{momentum-stable} when the conditions~\eqref{equa-stable-M-414},~\eqref{equa-Xi-invertible}, and~\eqref{equa-last-twoM}  hold. 
\end{definition}


\subsubsection{A weighted conformal Korn inequality}

To motivate the loss functional~$\mathfrak{K}[Z]$ for $n=3$, we begin with a conformal Korn inequality that holds in all dimensions other than $n=3$.

Fix $n\geq 4$ for now.
An important ingredient in proving coercivity of $\ssrmB^\lambda$ is the conformal Korn inequality, which states that the quantity $\fint_{\Lambda}\bigl|\Sym(\nablaslash \xipar)^\circ\bigr|^2\,d\chi$ controls $\xipar$ modulo the conformal Killing fields of the $(n-1)$-sphere. A basis of the latter is given by
\bei

\item the \emph{collection of rotations} (on the sphere) $\zeta =\xh_l\nablaslash\xh_m- \xh_m\nablaslash\xh_l$ with components $\zeta_k=\delta_{km}\xh_l- \delta_{kl}\xh_m$, which are Killing fields, and 

\item the \emph{collection of translations} of~$\RR^n$ (projected to the sphere) $\zeta =\nablaslash\xh_l$ with components $\zeta_k=\delta_{kl}- \xh_k\xh_l$.
\eei
Namely, given some constant anti-symmetric matrix~$L^{lm}$ and some constant vector~$L^l$, the linear combination
\bse
\bel{art2 -equa-confk} 
\zeta = L^{lm}(\xh_l\nablaslash\xh_m- \xh_m\nablaslash\xh_l) + L^l \nablaslash\xh_l
\ee
enjoys the properties 
\bel{art2 -equa-conf-killing}
\zeta_k = L^k - 2L^{kl}\xh_l - L^l\xh_l \xh_k,
\qquad
\Sym(\nablaslash\zeta)^\circ=0,
\qquad
\nablaslash\cdot\zeta =-(n-1)L^k\xh_k.
\ee
\ese
These explicit expressions will be useful later on. At this stage, we simply point out the inequality \eqref{art2 -equa--1133-copie-2}, below, proving that these fields are the main obstacle  to reaching coercivity statements.
For brevity we do not distinguish the constant in front of $L^2$ and $\dot H^1$ norms.
(See \autoref{sectionN3- 2}.)

\begin{proposition}[A weighted conformal Korn inequality]
\label{art2 -lem-conf-korn}
In dimension $n\geq 4$,
for any localization function~$\lambda$, the following Korn inequality holds for a positive constant $C_{\Korn}^\lambda$ and for any vector field $\xi^\parallel$ tangent to~$\Lambda$: 
\bel{art2 -equa--1133-copie-2}
\aligned
& \fint_{\Lambda}   \bigl|\Sym(\nablaslash \xi^{\parallel} )^\circ\bigr|^2 \, d\chi 
\geq {1 \over (C_{\Korn}^\lambda)^2} \min_{\zeta} \fint_{\Lambda} 
\Bigl(
| \nablaslash(\xi^{\parallel} - \zeta) |^2  
+ | \xi^{\parallel} - \zeta |^2 
\Bigr) \, d\chi, 
\endaligned
\ee
where the minimum is taken over all conformal Killing fields $\zeta$ in \eqref{art2 -equa-confk}. 
\end{proposition}


\subsubsection{Korn loss functional}

We now turn to $n=3$.
The set of \emph{local} conformal Killing vector fields on $\Sphe^2$ is well-known to be described by holomorphic functions, hence to be infinite-dimensional, in contrast to higher-dimensional spheres.
The rotation and translation vector fields~\eqref{art2 -equa-confk} constitute the set of \emph{global} conformal Killing vector fields, namely those that are defined on the whole sphere~$\Sphe^2$.

The analogue of \autoref{art2 -lem-conf-korn} requires a minimization over all conformal Killing vector fields on~$\Lambda$, which is infinite-dimensional except in the non-localized case $\Lambda=\Sphe^2$.  This infinite-dimensional obstruction space cannot be controlled in our application of the inequality.

Instead, we turn to the weighted Korn inequality in a $3$-dimensional domain
\be
\Omega_{r,2r} \coloneqq \{x \in \Omega_R \mid r<|x|<2r \} .
\ee
The inequality involves standard Killing vectors of Euclidean space, which are precisely the extensions to $\RR^3$ of the \emph{global} conformal Killing vector fields~\eqref{art2 -equa-confk} on~$\Sphe^2$.
The following inequality is a standard weighted Korn inequality.  For brevity we do not distinguish the constant in front of $L^2$ and $\dot H^1$ norms.
Note that we only seek to control the norm of $\xipar$ and not $\xiperp$, which could also be obtained.

\begin{proposition}[Korn inequality in a bounded domain]
\label{art2 -lem-korn-n3}
In dimension $n=3$,
for any localization function~$\lambda$, there exists a constant $C_{\Korn 3}^\lambda>0$ such that for any vector field $\xi=(\xiperp,\xipar)\colon\Omega_{r,2r}\to\RR^n$ one has
\be
\aligned
& \int_{\Omega_{r,2r}}   \bigl|\Sym(\del \xi )\bigr|^2 \, \lambda^{-2\expoP} dx 
\geq \frac{\aire[\Lambda,\lambda]}{(C_{\Korn 3}^\lambda)^2} \min_{\zeta} \int_r^{2r} \|\xipar - \zeta\|_{\unH^1_{-\expoP}(\Lambda_s)}^2 \frac{ds}{s} ,
\endaligned
\ee
where the minimum is taken over all rotation and translation fields~\eqref{art2 -equa-confk}.
The constant $C_{\Korn 3}^\lambda$ is independent of~$r$.
\end{proposition}

\begin{definition}
The \textbf{Korn loss functional} of a scalar-vector field $Z=(\Zperp,\Zpar)$ is defined, in dimension $n=3$, as
\bel{def-Korn-loss}
\mathfrak{K}[Z] \coloneqq \bigl\|\Sym(\del Z)\bigr\|_{\unL^2_{-\expoP}(\Lambda)}^2
- \frac{1}{(C_{\Korn 3}^\lambda)^2} \min_\zeta \|\Zpar - \zeta\|_{\unH^1_{-\expoP}}^2 ,
\ee
where the minimum is taken over \emph{global} conformal vector fields of the sphere,
and where $C_{\Korn 3}^\lambda$ is defined in \autoref{art2 -lem-korn-n3}.
\end{definition}

To relate this functional to others expressed in the orthogonal and parallel components, it is useful to decompose
\be
\aligned
\bigl|\Sym(\del Z )\bigr|^2
& = \bigl|\Sym(\nablaslash\Zpar)^\circ\bigr|^2
+ (n-1) \Bigl|\Zperp + \frac{1}{n-1} \nablaslash\cdot\Zpar\Bigr|^2
\\
& \quad + \bigl|(\vartheta-1)\Zpar + \nablaslash\Zperp\bigr|^2 + |\vartheta\Zperp|^2 .
\endaligned
\ee
Thanks to the first term, the higher-dimensional analogue of this functional (for a suitable choice of constant) would be non-negative by \autoref{art2 -lem-conf-korn}.  Here, it is generally negative, but its radial integral  is non-negative by  \autoref{art2 -lem-korn-n3}:
\bel{def-Korn-loss-integral}
\int_r^{2r} \mathfrak{K}[Z] \frac{ds}{s} = \int_r^{2r} \Bigl( \bigl\|\Sym(\del Z)\bigr\|_{\unL^2_{-\expoP}(\Lambda)}^2
- \frac{1}{(C_{\Korn 3}^\lambda)^2} \min_\zeta \|\Zpar - \zeta\|_{\unH^1_{-\expoP}}^2 \Bigr) \frac{ds}{s} \geq 0 .
\ee
This integrated positivity property is sufficient for the Einstein gluing problem in~\cite{LL-optimal-main}.


\subsection{Control of the fluctuation operator}
\label{art2 -section=3.3}

\subsubsection{Momentum averages}

For clarity in the presentation, we work throughout this section and \autoref{section-9} in dimension
\be
n\geq 4 .
\ee
The proof for $n=3$ is obtained by tracking all uses of the conformal Korn inequality~\eqref{art2 -korn-ineq-11}, below, and adding the corresponding multiple of the loss functional~$\mathfrak{K}[Z]$.

This section and the next are devoted to deriving momentum shell stability as a consequence of geometric inequalities.  We bound the fluctuation operator~$\Kappa^\notreM$ here in \autoref{art2 -section=3.3}, and, in \autoref{art2 -section=3.4}, we prove semi-coercivity of the dissipation functionals~$\Psi^\notreM_\beta$.
Motivated by this later section, we introduce now a quadratic functional that controls some first-order derivatives of~$Z$, as well as its $L^2_{- \expoP}$ norm with some small coefficient~$a_{n,p}^{1/2}$,
\bel{art2 -equa-eofz}
\aligned
\Ecal[Z]
& \coloneqq   
\bigl\|  (\vartheta-1) \Zpar + \nablaslash\Zperp  \bigr\|_{L^2_{- \expoP}(\Lambda_r)}^2 
+ \bigl\| \vartheta\Zperp  \bigr\|_{L^2_{- \expoP}(\Lambda_r)}^2 
+ \bigl\| \Zperp + \nablaslash\cdot\Zpar / (n-1) \bigr\|_{L^2_{- \expoP}(\Lambda_r)}^2 
\\
& \, \quad + \bigl\|\Sym(\nablaslash\Zpar)^\circ\bigr\|_{L^2_{- \expoP}(\Lambda_r)}^2 
+  a_{n,p} \, \bigl\|\Zpar\bigr\|_{L^2_{- \expoP}(\Lambda_r)}^2
+  a_{n,p} \, \bigl\|\Zperp\bigr\|_{L^2_{- \expoP}(\Lambda_r)}^2 .
\endaligned
\ee
In this section, we provide a simplified expression of the fluctuation operator~$\Kappa^\notreM[Z]$ up to some error terms that are controlled in terms of~$\Ecal[Z]$.

Throughout this section and the next, we fix once and for all a conformal Killing vector $\zeta_Z=\zeta(\Zpar) = L_Z^l \nablaslash\xh_l + 2L_Z^{lm}\xh_l\nablaslash\xh_m$ as before (with antisymmetric $L_Z^{lm}=-L_Z^{ml}$), chosen so that the Korn inequality of \autoref{art2 -lem-conf-korn} is satisfied, namely
\bel{art2 -korn-ineq-11}
{1 \over (C_{\Korn}^\lambda)^2} \Bigl( \bigl\| \nablaslash(\Zpar - \zeta_Z) \bigr\|_{L^2_{- \expoP}}^2  
+ \bigl\| \Zpar - \zeta_Z \bigr\|_{L^2_{- \expoP}}^2 \Bigr)
\leq \bigl\|\Sym(\nablaslash\Zpar)^\circ\bigr\|_{L^2_{- \expoP}}^2 \leq \Ecal[Z] .
\ee
First, we deduce an approximation for the average that appears in the shell stability condition.

\begin{lemma}[Expression of the momentum averages]
\label{art2 -prop-average-m}
In any localization domain, the averages of $Z$ are expressed as follows in terms of $L_Z^\bullet$ and~$L_Z^{\bullet\bullet}$ defined above,
\be
\Bigl| \bigl\la 2\xh_l\Zperp+\Zpar{}_l\bigr\ra -  T_{kl} L_Z^k - 2 L_Z^{kl} \la\xh_k\ra\Bigr|
\lesssim \bigl( 1 + C_{\Korn}^\lambda \bigr) \Ecal[Z]^{1/2} .
\ee
\end{lemma}

\begin{proof}
\bse
We introduce the difference
\be
\sigma_Z \coloneqq \Zperp - L_Z^k \xh_k = \Zperp + \frac{1}{n-1} \nablaslash\cdot\zeta_Z ,
\ee  
which is controlled as
\bel{art2 -sigmaz-bnd}
\aligned
\|\sigma_Z\|_{L^2_{- \expoP}(\Lambda_r)}
& \leq \Bigl\|\Zperp + \frac{1}{n-1} \nablaslash\cdot\Zpar\Bigr\|_{L^2_{- \expoP}(\Lambda_r)}
+ \frac{1}{n-1} \Bigl\|\nablaslash\cdot(\Zpar- \zeta_Z)\Bigr\|_{L^2_{- \expoP}(\Lambda_r)}
\\
& \leq \bigl(1+(n-1)^{-1/2} C_{\Korn}^\lambda\bigr) \Ecal[Z]^{1/2} .
\endaligned
\ee
Observe that $T_{kl}L_Z^k+2L_Z^{kl}\la\xh_k\ra = - 2(n-1)^{-1}\la\xh_l\nablaslash\cdot\zeta_Z\ra+\la\zeta_{Z\,l}\ra$, so
\be
\aligned
\bigl| \la 2\xh_l\Zperp+\Zpar{}_l\ra -  T_{kl} L_Z^k - 2 L_Z^{kl} \la\xh_k\ra\bigr|
& = \bigl| \la 2\xh_l\sigma_Z\ra + \la(\Zpar- \zeta_Z)_l\ra\bigr|
\\
& \leq 2 \|\sigma_Z\|_{L^2_{- \expoP}} + \|\Zpar- \zeta_Z\|_{L^2_{- \expoP}} ,
\endaligned
\ee
which combines with the Korn inequality and the control of $\sigma_Z$ to yield the desired bound.
\ese
\end{proof}


\subsubsection{Fluctuation operator}

An approximation of $\Kappa^\notreM[Z]$ is established next, following the same proof strategy as for the smallness of the Hamiltonian fluctuation operator $\Kappa^\notreH[Z]$ stated in \autoref{art2 -prop-k-fluctuations}.

\begin{proposition}[Control of the momentum fluctuation operator]
\label{art2 -prop-k-fluctuations-m}
In any localization domain, for all sufficiently small $a_{n,p}$ (depending on the localization domain), the fluctuation operator is bounded as follows ($l=1,\dots,n$)
\bel{art2 -equa-kmomentineq} 
\aligned
\bigl| \bigl(T (\Xi^\notreM)^{-1} \Kappa^\notreM[Z]\bigr)_l - 2 L_Z^{kl} \la\xh_k\ra \bigr|
& \lesssim (1 + C_{\Korn}^\lambda) (\CzeroKP+\ConeKP) |S^{- 2}| \,  \Ecal[Z]^{1/2} .
\endaligned
\ee
\end{proposition}

We restate the definition of the fluctuation operators here (cf.~\cite[Definition 9.7]{LL-optimal-main}) in the notation of the present paper, so that the proof below is self-contained.

\begin{definition}
\label{def:KappaM}
The \textbf{fluctuation operators} $\Kappa^{\notreM(j)}$ are linear functionals of the fluctuation~$Z^\fluc$ and its derivatives given by
\be
\aligned
\Kappa^{\notreM(j)}[Z^\fluc]
&= \bigl\la \nablaslash \xi^{\normal (j) \fluc\perp}
\cdot Z^{\fluc\parallel}\bigr\ra
- \bigl\la\xi^{\normal (j) \fluc\parallel}
\cdot \nablaslash Z^{\fluc\perp} \bigr\ra
\\
&\quad
- (\vartheta+a_{n,p}) \bigl\la
2 \xi^{\normal (j) \fluc\perp} Z^{\fluc\perp}
+ \xi^{\normal (j) \fluc\parallel}\cdot Z^{\fluc\parallel}\bigr\ra.
\endaligned
\ee
\end{definition}


\subsubsection{Proof of \autoref{art2 -prop-k-fluctuations-m}}

{\bf 1. \it Setup.}
The fluctuation operator $\Kappa^{\notreM(j)}$ given in \autoref{def:KappaM} is associated to a silhouette vector~$\xi^{\normal(j)}$.
We consider a linear combination~$\xi$ of silhouette vectors, parametrized by a constant vector $\alpha\in\RR^n$, and the corresponding linear combination $\Kappa^\notreM[\xi,Z]$ of fluctuation operators, which is a scalar:
\bel{art2 -xichoice-alphatxiinvxi}
\xi = \alpha^k T_{kl} ((\Xi^\notreM)^{-1})^l{}_{(j)} \xi^{\normal(j)} \in \ker(\ssB^\lambda) , \qquad
\Kappa^\notreM[\xi,Z] = \alpha^k T_{kl} ((\Xi^\notreM)^{-1})^l{}_{(j)} \Kappa^{\notreM(j)}[Z] .
\ee
(By taking $\alpha$ to be standard basis vectors of~$\RR^n$, we retrieve each component of the vector $T (\Xi^\notreM)^{-1} \Kappa^\notreM[Z]$ used in the proposition.)
Our strategy is to isolate contributions bounded by each term in $\Ecal[Z]$ in turn.
It is essential to distinguish terms that involve the silhouette fluctuations $\xi^{\normal (j) \fluc}$ (with  the tilde notation~\eqref{equa-tildemoment}), which are suppressed compared to the averages. We also keep track here of the fluctuation~$Z^{\fluc}$, and in particular its parallel component~$Z^{\fluc\parallel}$.

\medskip

\bse
\noindent{\bf 2. \it Reducing to the parallel fluctuation.}
We first set aside radial derivatives of~$Z$, and terms with~$a_{n,p}$,
\be
\aligned
\Kappa^\notreM[\xi,Z]
& = \bigl\la \nablaslash \xi^{\fluc\perp} \cdot Z^{\fluc\parallel} - \xi^{\fluc\parallel} \cdot \nablaslash Z^{\fluc\perp} \bigr\ra
- (\vartheta+a_{n,p}) \bigl\la 2 \xi^{\fluc\perp} Z^{\perp} + \xi^{\fluc\parallel} \cdot Z^{\parallel}\bigr\ra
\\
& = - 2 \bigl\la \xi^{\fluc\perp} \vartheta Z^{\perp} \bigr\ra
- \bigl\la \xi^{\fluc\parallel} \cdot ((\vartheta-1) Z^{\parallel} + \nablaslash Z^{\perp}) \bigr\ra
\\
& \quad - 2 a_{n,p} \bigl\la \xi^{\fluc\perp} Z^{\perp} \bigr\ra - a_{n,p} \bigl\la \xi^{\fluc\parallel} \cdot Z^{\parallel}\bigr\ra 
+ \bigl\la(\nablaslash \xi^{\fluc\perp} - \xi^{\fluc\parallel}) \cdot Z^{\fluc\parallel} \bigr\ra ,
\endaligned
\ee
where we used that $\nablaslash\Zperp - \nablaslash Z^{\fluc\perp} = \Zpar - Z^{\fluc\parallel}$.
The first four terms are bounded in terms of $\xi^{\fluc}$ and $\Ecal[Z]^{1/2}$ thanks to $a_{n,p}\lesssim a_{n,p}^{1/2}$. Hence, we have
\be
\aligned
\bigl| \Kappa^\notreM[\xi,Z] - \bigl\la(\nablaslash \xi^{\fluc\perp} - \xi^{\fluc\parallel}) \cdot Z^{\fluc\parallel} \bigr\ra \bigr|
& \lesssim \Ecal[Z]^{1/2} \, \|\xi^{\fluc}\|_{L^2_{- \expoP}(\Lambda_r)}, 
\endaligned
\ee
where the implied constant is only dimension-dependent.
\ese

\medskip

\bse
\noindent{\bf 3. \it Reducing to rotation Killing vectors.}
So far we have not used all terms in~$\Ecal[Z]$.
We exhibit in $Z^{\fluc\parallel}$ the differences $\Zpar- \zeta_Z$ and $\sigma_Z$ (controlled in \eqref{art2 -korn-ineq-11} and~\eqref{art2 -sigmaz-bnd}, respectively):
\be
\aligned
Z^{\fluc\parallel}
& = \Zpar - (T^{-1})^{kl} \, \bigl\la 2 \xh_l \Zperp + \Zpar{}_l \bigr\ra \nablaslash\xh_k
\\
& = \Bigl( (\Zpar- \zeta_Z) - (T^{-1})^{kl} \, \la\Zpar{}_l- \zeta_{Z\,l}\ra \nablaslash\xh_k \Bigr)
\\
& \quad
- 2 \bigl\la \xh_l \sigma_Z \bigr\ra (T^{-1})^{kl} \nablaslash \xh_k
+ L_Z^{lm} X^k{}_{lm} \nablaslash \xh_k ,
\endaligned
\ee
in terms of the following combination of coordinates (which is antisymmetric in $l\leftrightarrow m$)
\be
\aligned
X^k{}_{lm}
& = \delta^k_m \xh_l - \delta^k_l \xh_m + (T^{-1})^{kl} \la\xh_m\ra - (T^{-1})^{km} \la\xh_l\ra
\\
& = \delta^k_m (\xh_l- \la\xh_l\ra)
+ (T^{-1})^{kq} \Bigl\la(\xh_q- \la\xh_q\ra)(\xh_m- \la\xh_m\ra)\Bigr\ra \la\xh_l\ra
- (l\leftrightarrow m) 
\\
& = \Obig(\diam(\Lambda)) .
\endaligned
\ee
Bounding all terms except $L_Z^{lm} X^k{}_{lm} \nablaslash \xh_k$ by a multiple of $\Ecal[Z]^{1/2}$, one gets
\bel{art2 -xiz-to-rotation}
\aligned
\bigl| \Kappa^\notreM[\xi,Z] - L_Z^{lm} \bigl\la(\nablaslash \xi^{\fluc\perp} - \xi^{\fluc\parallel})_k X^k{}_{lm} \bigr\ra \bigr|
& \lesssim (1 + C_{\Korn}^\lambda) \Ecal[Z]^{1/2} \, \|\xi^{\fluc}\|_{H^1_{- \expoP}(\Lambda_r)} .
\endaligned
\ee
\ese

\medskip

\bse
\noindent{\bf 4. \it First moments of the silhouette vectors.}
We analyze further the left-hand side of~\eqref{art2 -xiz-to-rotation}.
Note that $\nablaslash \xi^{\fluc\perp} - \xi^{\fluc\parallel} = \nablaslash\xiperp - \xipar$ hence we are interested in finding bounds on the averages $\la(\nablaslash \xiperp - \xipar)_k X^k{}_{lm}\ra$, which can be thought of as moments of~$\xi$.
To get an equation for these moments, we consider the (rotation) Killing vector field $\psi = 2 L_Z^{lm} \xh_l\nablaslash\xh_m$.
It obeys $\nablaslash\cdot\psi=0$ and $\Sym(\nablaslash\psi)=0$.
Expanding $\ssB^{\lambda}[\xi]=0$ using the explicit form of~$\ssB^{\lambda}$ in \eqref{Brs-expr-main} yields
\be
0 = 2 \fint_\Lambda \psi \cdot \ssB^\lambda[\xi] d\chi
= 2 a_{n,p} L_Z^{lk} \bigl\la \xipar{}_k \xh_l \bigr\ra
- 2 L_Z^{lk} \bigl\la(\nablaslash\xi^{\fluc\perp} - \xi^{\fluc\parallel})_k \xh_l\bigr\ra .
\ee
By decomposing the vector~$\xi$ into its fluctuations $\xi^{\fluc}$ and averages $P^k(\xi) = (T^{-1})^{kl} \la 2\xh_l\xiperp+\xipar{}_l\ra$ this allows us to write
\be
\aligned
& L_Z^{lm} \bigl\la(\nablaslash \xi^{\fluc\perp} - \xi^{\fluc\parallel})_k X^k{}_{lm} \bigr\ra
\\
& = 2 a_{n,p} L_Z^{lk} \bigl\la \xi^{\fluc\parallel}{}_k \xh_l \bigr\ra - 2 L_Z^{lm} \la\xh_m\ra (T^{-1})^{lk} \bigl\la- \nablaslash_k \xiperp + 2a_{n,p}\xh_k\xiperp + (1+a_{n,p}) \xipar{}_k \bigr\ra .
\endaligned
\ee
We recognize the same average as that defining~$\Xi^\notreM$, so that the expression in terms of $\alpha$ (see~\eqref{art2 -xichoice-alphatxiinvxi}) simplifies:
\bel{art2 -LZlmbiglanxi}
L_Z^{lm} \bigl\la(\nablaslash \xi^{\fluc\perp} - \xi^{\fluc\parallel})_k X^k{}_{lm} \bigr\ra
= 2 a_{n,p} L_Z^{lk} \bigl\la \xi^{\fluc\parallel}{}_k \xh_l \bigr\ra + 2 \alpha^l L_Z^{kl} \la\xh_k\ra .
\ee
\ese
The second term is the approximation of $\Kappa^\notreM$ that we aim for, so there remains to control the first term.

\medskip

\bse
\noindent{\bf 5. \it Control of the conformal Killing field.}
From the bound $a_{n,p}\|\Zperp,\Zpar\|_{L^2_{- \expoP}(\Lambda)}^2\leq\Ecal[Z]$ one gets a control of coefficients $L_Z^\bullet$ and $L_Z^{\bullet\bullet}$, using \eqref{art2 -korn-ineq-11} and~\eqref{art2 -sigmaz-bnd},
\bel{art2 -bothbounded34}
\aligned
\|\zeta_Z\|_{L^2_{- \expoP}}
& \leq \|\Zpar\|_{L^2_{- \expoP}} + \|\Zpar- \zeta_Z\|_{L^2_{- \expoP}}
\leq \bigl( a_{n,p}^{-1/2} + C_{\Korn}^\lambda \bigr) \Ecal[Z]^{1/2} ,
\\
|SL_Z^\bullet| = \|L_Z^k \xh_k\|_{L^2_{- \expoP}}
& \leq \|\Zperp\|_{L^2_{- \expoP}} + \|\sigma_Z\|_{L^2_{- \expoP}}
\leq \bigl( a_{n,p}^{-1/2} + 1 + (n-1)^{-1/2} C_{\Korn}^\lambda \bigr) \Ecal[Z]^{1/2} ,
\endaligned
\ee
which are both bounded by $2(\Ecal[Z]/a_{n,p})^{1/2}$ for sufficiently small $a_{n,p}$ compared to the geometric constants.
We now use the explicit expression of~$\zeta_Z$ to evaluate the (squared) norms, and rewrite it in terms of the covariance matrix $\matV =(\la\xh_k\xh_l\ra- \la\xh_k\ra\la\xh_l\ra)$, which is positive-definite,
\be
\aligned
\|L_Z^k\xh_k\|_{L^2_{- \expoP}}^2 + \|\zeta_Z\|_{L^2_{- \expoP}}^2
& = (L_Z^k)^2 + 4 L_Z^k L_Z^{lk}\la\xh_l\ra + 4L_Z^{kl}L_Z^{ml}\la\xh_k\xh_m\ra
\\
& = (L_Z^k + 2 L_Z^{lk}\la\xh_l\ra)^2 + 4L_Z^{kl}L_Z^{ml}\matV_{km} .
\endaligned
\ee
The second term here provides a control of all rotation coefficients $L_Z^{\bullet\bullet}$ in terms of the matrix norm of $\matV^{-1}$ and $\Ecal[Z]$ and, taking a square root, (as usual, implicit constants depend only on the dimension)
\bel{art2 -LZbullbull-contro}
|L_Z^{\bullet\bullet}| \lesssim \, | \matV^{-1/2} | \,  \bigl( a_{n,p}^{-1/2} + 1 + C_{\Korn}^\lambda \bigr) \Ecal[Z]^{1/2} .
\ee
\ese

\medskip

\bse
\noindent{\bf 6. \it Fluctuations of the silhouette vector.}
From \eqref{art2 -xiz-to-rotation}, \eqref{art2 -LZlmbiglanxi} and~\eqref{art2 -LZbullbull-contro} we obtain
\bel{art2 -fluctuations-silh-v}
\aligned
\bigl| \Kappa^\notreM[\xi,Z] - 2 \alpha^l L_Z^{kl} \la\xh_k\ra \bigr|
& \lesssim \Bigl((1 + C_{\Korn}^\lambda)
\Bigl[1 + a_{n,p}
\, | \matV^{-1/2} | \, \Bigr]
\\[- 2pt]
&\hspace{2.5cm}{}
+ a_{n,p}^{1/2}
\, | \matV^{-1/2} | \, \Bigr)
\Ecal[Z]^{1/2}\,\|\xi^{\fluc}\|_{H^1_{- \expoP}(\Lambda_r)} .
\endaligned
\ee
For small enough $a_{n,p}$ the coefficient in parentheses reduces to geometric constants.  To finish up, we must understand the norm of~$\xi^\fluc$.
We rely on~\eqref{art2 -equa8d0} with $v_{(j)}=\alpha^k T_{kl} ((\Xi^\notreM)^{-1})^l{}_{(j)}$ and $\eps=2^{-1/2}\CzeroKP a_{n,p}$
\be
\aligned
\|\xi^{\fluc}\|_{H^1_{- \expoP}(\Lambda)}
& \leq (\CzeroKP+\ConeKP) \frac{a_{n,p}}{1- \eps} \bigl|\alpha^k T_{kl} ((\Xi^\notreM)^{-1})^l{}_{(j)} \matQ^{(j)\bullet}\bigr|_U
\\
& \lesssim a_{n,p} (\CzeroKP+\ConeKP) |\alpha| \bigl|((\Xi^\notreM)^{-1})^\bullet{}_{(j)} \matQ^{(j)\bullet}\bigr|_U
\endaligned
\ee
for sufficiently small $a_{n,p}$.
We conclude by recalling the lower bound $(\matQ^{-1} \Xi^\notreM)^{\text{sym}} \geq a_{n,p} S^2 > 0$ on the symmetric part of $\matQ^{-1}\Xi^\notreM$ in \autoref{art2 -coro-radial-momen}, which implies that the smallest singular value of $\matQ^{-1}\Xi^\notreM$ (which is greater or equal to that of its symmetric part) is greater or equal to that of~$a_{n,p}S^2$.  Thus, the operator norm of $(\Xi^\notreM)^{-1}\matQ$ is bounded by $a_{n,p}^{-1} \, |S^{- 2}|$, and we deduce
\be
\|\xi^{\fluc}\|_{H^1_{- \expoP}(\Lambda)}
\lesssim (\CzeroKP+\ConeKP) |\alpha| \, |S^{- 2}|, 
\ee
which combines with~\eqref{art2 -fluctuations-silh-v} to establish the overall bound~\eqref{art2 -equa-kmomentineq} provided $a_{n,p}$ is small enough, depending on $\CzeroKP$, $C_{\Korn}^\lambda$, $\, | \matV^{-1/2} | \, $, and the dimension.
\ese
This completes the proof of \autoref{art2 -prop-k-fluctuations-m}.


\section{Proof of the Hamiltonian shell stability}
\label{section-8}

\subsection{Class of functionals of interest}

\subsubsection{Shell functional}

\bse
We now turn our attention to the derivation of a functional associated with the localized Hamiltonian operator and the study of its coercivity properties.
This section is completely insensitive to the choice of localization domain and function.  However, the appropriate choice of some parameters (denoted by $\cstun,\cstdeux$, or $c_2,c_{13}$ below) may depend on the class of localization domains of interest.
Our strategy is to begin with a sufficiently general sum-of-squares Ansatz for $\Phi^\notreH$ that is quadratic in $u$ and its first and second-order derivatives, is manifestly non-negative without restriction on the function~$\lambda^{2\expoP}$, and only involves scalar parameters. Namely, for a collection of constants $c_1,\dots,c_{13}\in\RR$, we consider
\be
\aligned
\Phi^\notreH[u] 
& = \frac{1}{2} \fint_{\Lambda_r} \Bigl(
\bigl(\vartheta^2 u + c_1 \vartheta u - c_2 u + c_3\Deltaslash u\bigr)^2
+ c_4^2 \, \bigl| \nablaslash\vartheta u + c_5 \nablaslash u \bigr|^2
+ c_6^2 \, \bigl( \vartheta u + c_7 u + c_8 \Deltaslash u\bigr)^2 
\\
& \qquad\qquad
+ c_9^2 \bigl( \Deltaslash u + c_{10} u \bigr)^2
+ c_{11}^2 \, \bigl|(\nablaslash^2 u)^\circ\bigr|^2
+ c_{12}^2 \,  |\nablaslash u|^2
+ c_{13}^2 \, u^2
\Bigr) \, d\chi, 
\endaligned
\ee
where $(\nablaslash^2 u)^\circ$ denotes the traceless part of the Hessian. If the coefficients of $|\nablaslash u|^2$ and $u^2$ were denoted by independent real parameters rather than by the literal squares $c_{12}^2$ and $c_{13}^2$, they could be allowed to be slightly negative: suitable Poincar\'e inequalities would still make the integrated combination with $c_{11}^2\bigl|(\nablaslash^2u)^\circ\bigr|^2$ non-negative. The admissible values would, however, depend on the weight~$\lambda^{2\expoP}$, which we prefer to avoid at this stage.

Once this functional is chosen, the main identity~\eqref{main-func-identity} can be seen as a \emph{definition} of the functionals $\Mu^\notreH$ and $\Chi^\notreH$, and a partial definition of $\Upsilon^\notreH$ and~$\Psi^\notreH$ defined in the course of this section. The functional $\Phi^\notreH$ is {\it non-negative} by construction. Our main task is to evaluate $\Psi^\notreH$ and put it in a form where its positivity properties (for a suitable class of~$\lambda$) can be stated in terms of a Poincaré inequality. Throughout this section, we strive to choose some of the constants in such a way as to reach the desired positivity structure and, simultaneously, to simplify our calculations.  This will lead us to the choice
\bel{art2 -equa-classPhi-b}
\aligned
c_1 & = c_5 = c_7 = a_{n,p}, \qquad
c_3 = c_4 = c_6 = c_8 = 0, \qquad
c_9^2 = \frac{n^2 -3n+3}{(n-1)^2}, \qquad
\\
c_{10}
& = - \frac{n-1}{n^2 -3n+3} \bigl(c_{n,p}+(n- 2)c_2\bigr), \qquad
c_{11}^2 = \frac{1}{n-1}, \qquad
c_{12}^2 = \frac{2}{n-1} (1+a_{n,p}+c_2) ,
\endaligned
\ee
together with a sufficiently large $c_2>0$ and a sufficiently large $c_{13}>0$ (depending on~$c_2$).
In \autoref{section-6} we prefer to recast these in terms of two parameters $\cstun,\cstdeux>0$ with
\be
\cstun = c_2 , \qquad \frac{\cstdeux}{n-1} = c_{13}^2 + c_9^2 c_{10}^2 .
\ee
\ese


\subsubsection{Source functional}

\bse
The expression of $(\vartheta+a_{n,p})(\vartheta+2a_{n,p}) \Phi^\notreH[u]$ involves terms quadratic in derivatives of order at most three of~$u$, which we will eventually incorporate in the definition of~$\Psi^\notreH[u]$, but also terms involving fourth derivatives of~$u$. Among these, we find $\vartheta^4 u\,(\vartheta^2 u+a_{n,p}\vartheta u-c_2 u+c_3\Deltaslash u)$, and these fourth-order radial derivatives (cf.~\autoref{Rem-foot}, below) must be cancelled by the Hamiltonian operator by setting
\be
\aligned
\Mu^\notreH[u,E]
& = - \frac{1}{n-1} \fint_{\Lambda_r}  \bigl(\vartheta^2 u + a_{n,p} \vartheta u - c_2 u + c_3 \Deltaslash u \bigr) \, r^4 E \, d\chi \\
& = - \frac{1}{n-1} \fint_{\Lambda_r}  \bigl(\vartheta^2 u + a_{n,p} \vartheta u - c_2 u + c_3 \Deltaslash u \bigr) \, r^4 \notreH^{\lambda}[u] \, d\chi.
\endaligned
\ee
While the harmonic-spherical decomposition~\eqref{equa--488- 2} of the Hamiltonian involves derivatives of the weight~$\lambda^{2\expoP}$, all of them can be eliminated by integrating by parts on the shell~$\Lambda_r$, leading to a quadratic functional
\bel{art2 -Chi-def}
\Chi^\notreH[u] = (\vartheta + a_{n,p}) (\vartheta + 2a_{n,p}) \Phi^\notreH[u] + \Mu^\notreH[u,E]
\ee
consisting of (the integral of) a quadratic combination of derivatives of~$u$. The result is manifestly free from $\vartheta^4 u$ but contains other fourth-order derivatives of~$u$. In particular, it contains a term of the form $c_3\fint\Deltaslash u\Deltaslash^2 u d\chi$ which has no suitable positivity properties. This motivates us to select $c_3=0$. This choice is reflected in the absence of the Laplacian term $c_3\Deltaslash u$ in the expression of $\Mu^\notreH$ given earlier in~\eqref{Psi-gendef}. Another essential consequence is the absence of terms with third-order angular derivatives of~$u$.
\ese
%


\subsubsection{Radial integration functional}

\bse
The fourth-order derivatives that remain in~\eqref{art2 -Chi-def} all include at least one radial derivative and appear in terms of the form $\fint D^ju\,\vartheta D^3u\,d\chi$ for $j\leq 2$, where $D^k$ stands for $k$-th order (radial or angular) derivatives. Such terms are accounted for by including $\fint D^ju D^3u\,d\chi$ in the expression of~$\Upsilon^\notreH$. Specifically, an explicit calculation yields  
\be
\aligned
\Upsilon^\notreH[u]
& = \fint_{\Lambda_r} \Bigl(
c_4^2 \, (\vartheta+c_5) \nablaslash u \, \vartheta^2 \nablaslash u
+ \Bigl(c_{11}^2 - \frac{1}{n-1}\Bigr) \nablaslash^2 u \, \vartheta\nablaslash^2 u
\\
& \qquad\qquad - \Bigl( \frac{n- 2}{n-1} \, \bigl( 2\vartheta^2 u + (c_1+n- 2) \vartheta u - c_2 u \bigr) + \frac{1}{n-1}\vartheta u \\
& \qquad\qquad\qquad - c_6^2 c_8 \, \bigl( \vartheta u + c_7 u + c_8 \, \Deltaslash u \bigr) - c_9^2 \, (\Deltaslash u + c_{10} u) + \frac{c_{11}^2 + n- 2}{n-1} \Deltaslash u\Bigr) \vartheta \Deltaslash u \\
& \qquad\qquad + \upsilon^\notreH\bigl(\{\vartheta^j\nablaslash^k u,\,j+k\leq 2\}\bigr)
\Bigr) \, d\chi,
\endaligned
\ee
where $\upsilon^\notreH$ is an arbitrary quadratic form (depending on $14$~constants) at this stage since it does not involve the highest derivatives of~$u$. In terms of this functional, the functionals $\Psi^\notreH_\beta$ are then defined by
\be
\Psi^\notreH_\beta[u] \coloneqq \Chi^\notreH[u] - (\vartheta+\beta) \Upsilon^\notreH[u], \qquad\beta \in \{a_{n,p}, 2a_{n,p}\},
\ee
and are integrals of quadratic forms in $\vartheta^j\nablaslash^k u$ for $j+k\leq 3$ and $k\leq 2$. We are interested in their positivity properties. In principle, we can evaluate these quadratic forms for the most general choice of $26$ constants in $\upsilon^\notreH$ and $c_1,c_2,c_4,\dots,c_{13}$, and seek suitable choices afterwards, but we find it convenient to restrict our attention at once to choices that are inspired by the harmonic-spherical decomposition.

For this reason, we impose that $\Upsilon^\notreH[u]$ vanishes when $u$ is a function with harmonic decay.  The functional must thus be expressible in terms of
$
w \coloneqq (\vartheta + a_{n,p}) u
$
and of its radial and angular derivatives. This fixes some of the constants,
\be
\aligned
c_5 & = a_{n,p}, \qquad
c_9^2 = \frac{n^2 -3n+3}{(n-1)^2} - c_6^2 c_8^2, \qquad
c_{11}^2 = \frac{1}{n-1},
\\
c_{10}
& = - \frac{n- 2}{(n-1)c_9^2} \Bigl(c_2 + a_{n,p} (c_1 - a_{n,p}) + \frac{c_{n,p}}{n- 2}\Bigr)
+ \frac{c_6^2 c_8}{c_9^2} (a_{n,p} - c_7) ,
\endaligned
\ee
together with the condition $c_9^2\neq 0$, which will be satisfied by our choices later on.
After redefining the quadratic form~$\upsilon^\notreH$ to absorb many lower-order terms, we have 
\be
\aligned
\Upsilon^\notreH[u]
& = \fint_{\Lambda_r} \biggl(
c_4^2 \, \nablaslash w \, \vartheta \nablaslash w
- \biggl( 2 \frac{n- 2}{n-1} \, \vartheta w + \frac{n- 2}{n-1} (c_1-a_{n,p}) w + \frac{c_{n,p}}{(n-1)a_{n,p}} w - c_6^2 c_8 w \biggr) \Deltaslash w \\
& \qquad\qquad + \upsilon^\notreH\bigl(w,\vartheta w,\nablaslash w\bigr)
\biggr) \, d\chi.
\endaligned
\ee
Besides fixing some of the constants $c_i$, our choice reduced the freedom in~$\upsilon^\notreH$ to only $4$ constants, that is, 
\be
\upsilon^\notreH\bigl(w,\vartheta w,\nablaslash w\bigr)
= c_{14} (\vartheta w)^2 + c_{15} w\vartheta w + c_{16} w^2 + c_{17} |\nablaslash w|^2.
\ee
There remains to choose the $12$ constants $c_1,c_2,c_4,c_6,c_7,c_8,c_{12},\dots,c_{17}$.
\ese
%


\subsubsection{\texorpdfstring{Bilinear terms in the variables $u,w$}{Bilinear terms in the variables u and w}}

\bse
For functions~$u$ with $r^{-a_{n,p}}$ harmonic decay, we have chosen $\Upsilon^\notreH$ to vanish, so that
$\Psi^\notreH_\beta[u] = \Chi^\notreH[u]$ for $u = r^{-a_{n,p}} \nu$ and 
for any angular function~$\nu$ on~$\Lambda$,
without dependence on~$\beta$. The main identity~\eqref{main-func-identity} for such a function~$u$ reduces to
\be
\Psi^\notreH_\beta[u] = \frac{ c_2 + a_{n,p}(c_1-a_{n,p})}{n-1} \, \ssrmA^\lambda[u] \qquad \text{ for } u = r^{-a_{n,p}} \nu ,
\ee
for any angular function~$\nu$ on~$\Lambda$, where we recall $\ssrmA^\lambda[u] = \fint_{\Lambda_r} u \, \ssA^\lambda[u] \, d\chi$. Thus, in this case, the functionals of interest $\Psi^\notreH_\beta$ reduce to the asymptotic functional, whose coercivity was studied earlier. For more general functions~$u$, this means that the functionals can be written as
\bel{art2 -e011}
\Psi^\notreH_\beta[u]
= \PsiHquabeta[w]
+ \PsiHbil[u;w]
+ \frac{ c_2 + a_{n,p}(c_1-a_{n,p})}{n-1} \ssrmA^\lambda[u], 
\ee
where $\PsiHquabeta$ are quadratic functionals in~$w$ and $\PsiHbil$ is a bilinear form in $u,w$ involving up to two derivatives of~$w$ and up to two \emph{angular} derivatives of~$u$. The $\beta$ dependence only appears in the first term since $\Upsilon^\notreH$ only involves~$w$.

To better separate the functional into independent contributions, we will strive to make most terms in~$\PsiHbil[u;w]$ vanish. We consider first the terms involving second derivatives of~$w$, which read
\be
\aligned
\PsiHbil[u;w] & = \fint_{\Lambda_r} \biggl(
c_6^2 c_8 \, \Deltaslash u \vartheta^2 w
+ \frac{a_{n,p}-c_1}{n-1} \Bigl( \nablaslash^2 u\nablaslash^2 w + (n- 2) \Deltaslash u \Deltaslash w \Bigr)
- \frac{c_{n,p}c_2}{(n-1)a_{n,p}} u \Deltaslash w \\
& \qquad\quad
+ c_6^2 (c_7-a_{n,p}) u \vartheta^2 w
+ \Bigl( c_{12}^2 - \frac{2}{n-1} (1+a_{n,p}+c_2) \Bigr) \nablaslash u\,\vartheta\nablaslash w
\biggr) d\chi + \ldots,
\endaligned
\ee
where dots denote terms with at most first derivatives of~$w$. All of these terms except $u\Deltaslash w$ can be eliminated by a suitable choice of four constants,
\be
c_1 = c_7 = a_{n,p}, \qquad
c_8=0, \qquad
c_{12}^2 = \frac{2}{n-1} (1+a_{n,p}+c_2).
\ee
The $u\Deltaslash w$ term cannot be eliminated.  Indeed, the coefficient of the quadratic term in~$u$ in~\eqref{art2 -e011} must be taken to be positive, so that (with our choice $c_1=a_{n,p}$) one must have $c_2>0$.
\ese

\begin{remark}
\label{Rem-foot}
The identities in this section are derived here in the smooth setting.  We applied them in \cite{LL-optimal-main}
with a source-term $E\in H^{2*}_{n+2 -p,- \expoP}(\Omega_R)$ and a solution $u,\vartheta u\in H^2_{n- 2 -p,- \expoP}(\Omega_R)$,  
for which the identities hold in the sense of distributions on $(R,+\infty)$.  For instance, a quadratic term such as $\fint_{\Lambda_r} \vartheta^3 u\,\Deltaslash\vartheta u\,d\chi$ is a product of two functions in $L^2_{n- 2 -p,- \expoP}(\Omega_R)$, integrated on a spherical shell~$\Lambda_r$, hence is integrable on $[R,+\infty)$ with the measure $r^{a_{n,p}-1}dr$.  In contrast, a term such as $\fint_{\Lambda_r}\vartheta^2 u\,\Deltaslash\vartheta^2 u\,d\chi$ involving fourth-order derivatives is defined solely as a distribution, by using that $\varphi\vartheta^2 u\in H^1_{n- 2 -p,- \expoP}(\Omega_R)$ for a suitable test function $\varphi=\varphi(r)$ on $[R,+\infty)$.  Altogether, $\Mu^\notreH$~is constructed in~\cite{LL-optimal-main} for $c_3=0$. 
\end{remark}


\subsection{Semi-coercivity property}

\subsubsection{Quadratic terms in the variable \texorpdfstring{$w$}{w}}

We now show coercivity of the terms that only depend on $w=(\vartheta+a_{n,p})u$, for a suitable choice of some of the constants motivated in the proof that follows.

\begin{lemma}[Semi-coercivity of the Hamiltonian dissipation functional. I]
For each fixed $\beta\in\{a_{n,p},2a_{n,p}\}$, for $c_4=c_{17}=0$, $c_6=c_{16}=0$, $c_{14}=a_{n,p}$, $c_{15}=c_6^2 - 2c_2$, and for sufficiently large $c_2\geq c_2^{\min}(\beta)+1$ and $c_{13}^2\geq c_{13}^{2\,\min}(c_2,\beta)+1$ as specified in the proof, one has
\be
\PsiHquabeta[w] \gtrsim (\norm{w}^\notreH)^2
= \| \vartheta^2 w\|^2_{L^2_{- \expoP}(\Lambda)} + \| \vartheta w\|^2_{H^1_{- \expoP}(\Lambda)} + \| w\|^2_{H^2_{- \expoP}(\Lambda)}
\ee
with a positive implicit constant depending on $n,p,c_2,c_{13}$. For the simultaneous choices of $c_2$ and $c_{13}$ made after \eqref{art2 -c2c13-new}, the constant can be chosen uniformly in $p$ and to depend only on the dimension, as shown below.
\end{lemma}

The uniformity in~$p$ is asserted only after the simultaneous quantitative choice of $c_2$ and $c_{13}$ made in~\eqref{art2 -c2c13-new}: taking these constants arbitrary large alone would not suffice. The restriction $c_6=c_{16}=0$ simplifies Step~3 of the proof below.

\begin{proof}
{\it 1. Evaluating $\PsiHquabeta$.}
\bse
The integrand in~$\PsiHquabeta[w]$ is a quadratic form in $\vartheta^j\nablaslash^k w$ for $j+k\leq 2$, which splits into contributions from derivatives that are tensors on the sphere, and those that are scalars,
\be
\aligned
\PsiHquabeta[w] & = \fint_{\Lambda_r} \Bigl( \psiHtenbeta\bigl((\nablaslash^2 w)^\circ, \vartheta\nablaslash w, \nablaslash w\bigr) + \psiHscabeta\bigl(\vartheta^2 w, \Deltaslash w, \vartheta w,w\bigr) \Bigr) d\chi .
\endaligned
\ee
The tensorial terms $\psiHtenbeta[w] = \psiHtenbeta\bigl((\nablaslash^2 w)^\circ, \vartheta\nablaslash w, \nablaslash w\bigr)$ read
\be
\aligned
\psiHtenbeta[w] 
& = \frac{1}{n-1} \bigl|(\nablaslash^2 w)^\circ\bigr|^2
+ \frac{2}{n-1} |\nablaslash\vartheta w|^2
+ ((3 a_{n,p} - \beta) c_4^2 - 2 c_{17}) \nablaslash w\nablaslash\vartheta w \\
& \quad
+ \Bigl(\frac{2}{n-1} (c_2 + 1 + a_{n,p}) + a_{n,p}^2  c_4^2 - c_{17}\beta\Bigr) |\nablaslash w|^2 ,
\endaligned
\ee
and can be made manifestly non-negative by taking, for example, $c_4=c_{17}=0$. This choice is not important as these constants are absent from the scalar part of the dissipation functional.

The scalar terms are more involved, and we focus first on the contribution of $\vartheta^2 w$,
\be
\aligned
\psiHscabeta[w] & = \psiHscabeta\bigl(\vartheta^2 w, \Deltaslash w, \vartheta w,w\bigr) \\
& = (\vartheta^2 w)^2 + 2 \vartheta^2 w \Bigl( (a_{n,p} - c_{14}) \vartheta w + \frac{1}{2} (c_6^2 - c_{15} - 2 c_2) w + \frac{n- 2}{n-1} \Deltaslash w \Bigr) + \lot, 
\endaligned
\ee
where $\lot$ stands for a quadratic form in $\Deltaslash w$, $\vartheta w$ and~$w$. We choose the constants $c_{14}$ and $c_{15}$ to eliminate the cross-terms $\vartheta w\vartheta^2 w$ and $w\vartheta^2 w$. Altogether, these considerations leave us to finally choose $c_2,c_6,c_{13},c_{16}$, only, after we set 
\be
c_4 = c_{17} = 0 , \qquad c_{14} = a_{n,p} , \qquad c_{15} = c_6^2 - 2 c_2 .
\ee

\ese


\medskip

\bse
\noindent{\it 2. Full expression of the dissipation functional.}
With these choices of constants, we arrive at the first complete expression
\be
\Psi^\notreH_\beta[u]
= \PsiHquabeta[w]
+ \PsiHbil[u;w]
+ \frac{c_2}{n-1} \ssrmA^\lambda[u] .
\ee
The quadratic terms in~$u$ are $\ssrmA^\lambda[u]$,  and the bilinear terms in $u$ and $w$ are compactly expressed in terms of specific combinations of derivatives of~$w$, denoted $w_1$ and $w_2$,
\bel{art2 -PsiHbil-expr}
\aligned
\PsiHbil[u;w] & = \fint_{\Lambda_r} \bigl( w_1 \Deltaslash u - w_2 u \bigr) d\chi,
\\
w_1  & \coloneqq \frac{c_2 c_{n,p}}{(n-1)a_{n,p}} w - \frac{c_{n,p}}{n-1} \vartheta w , \quad
w_2 \coloneqq \frac{c_2 c_{n,p}}{(n-1)a_{n,p}} \Deltaslash w - c_{130} \vartheta w,
\\
c_{130} & \coloneqq c_{13}^2 + \frac{1}{n^2 -3n+3} \bigl(c_{n,p} + (n- 2)c_2\bigr)^2 - c_2(b_{n,p}-c_2).
\endaligned
\ee
\ese
The quadratic terms in~$w$ are grouped into tensor terms and scalar terms, specifically 
\bse
\label{art2 -PsiHquabeta-def}
\be
\aligned
\PsiHquabeta[w] & \coloneqq  \fint_{\Lambda_r} \Bigl( \psiHtenbeta[w] + \psiHscabeta[w] \Bigr) d\chi ,
\\
\psiHtenbeta[w]
&: = \frac{1}{n-1} \biggl(
\bigl|(\nablaslash^2 w)^\circ\bigr|^2
+2 |\nablaslash\vartheta w|^2
+2 (c_2 + 1 + a_{n,p}) |\nablaslash w|^2
\biggr) ,
\endaligned
\ee
together with 
\be
\aligned
\psiHscabeta[w]
& \coloneqq  \Bigl(\vartheta^2 w + \frac{n- 2}{n-1} \Deltaslash w\Bigr)^2 
+ \frac{1}{n-1} (\Deltaslash w)^2
+ \frac{2 (c_{n,p}/a_{n,p} + (n - 2) \beta)}{n-1} \vartheta w \Deltaslash w
\\
& \quad
+ \bigl(c_2 + b_{n,p} + a_{n,p}(a_{n,p} - \beta) \bigr) (\vartheta w)^2 \\
& \quad
- \frac{1}{n-1} \bigl(2 (n- 2) c_2 + (2a_{n,p}- \beta)c_{n,p}/a_{n,p}\bigr) w \Deltaslash w \\
& \quad
+ \bigl(2 (\beta - a_{n,p}) c_2 + (3 a_{n,p} - \beta) c_6^2 - 2 c_{16}\bigr) w \vartheta w
\\
& \quad
+ \bigl( c_{130} + c_2 b_{n,p} + a_{n,p}^2 c_6^2  - c_{16}\beta\bigr) w^2 .
\endaligned
\ee
\ese


\medskip

\bse
\noindent{\it 3. Hierarchy of constants.}
The constants $c_6$ and $c_{16}$ are of limited interest in achieving positivity properties, so we arbitrarily set them to zero to shorten our expressions, that is, 
$
c_6 = c_{16} = 0 .
$
The scalar terms can be rewritten (under the condition~\eqref{art2 -c2c13-cond} below) as
\bel{art2 -psiHscabeta-re}
\aligned
\psiHscabeta[w]
& = \Bigl(\vartheta^2 w + \frac{n- 2}{n-1} \Deltaslash w\Bigr)^2
+ \frac{1}{n-1} \bigl(\Deltaslash w + d_1(\beta) \vartheta w - d_2(c_2,\beta) w\bigr)^2 \\
& \quad + \bigl(c_2 - c_2^{\min}(\beta) \bigr) \Bigl(\vartheta w + \frac{d_3(c_2,\beta)}{c_2 - c_2^{\min}(\beta)} w\Bigr)^2
\\
& \quad
+ \bigl(c_{13}^2 - c_{13}^{2\min}(c_2,\beta) \bigr) w^2 
\endaligned
\ee
in terms of a set of constants that depend only on~$\beta$ and~$c_2$, that is, 
\be
\aligned
d_1 = d_1(\beta) \, & \! \coloneqq
\frac{c_{n,p}}{a_{n,p}} + (n - 2) \beta ,
\\
c_2^{\min} = c_2^{\min}(\beta) \, & \! \coloneqq
\frac{d_1(\beta)^2}{n-1}
+ a_{n,p}(\beta - a_{n,p}) - b_{n,p} ,
\\
d_2 = d_2(c_2,\beta) \, & \! \coloneqq (n- 2) c_2 + \frac{c_{n,p}}{2a_{n,p}} (2a_{n,p}- \beta) ,
\\
d_3 = d_3(c_2,\beta) \, & \! \coloneqq (\beta - a_{n,p}) c_2 + \frac{1}{n-1} d_1(\beta)d_2(c_2,\beta) ,
\\
c_{13}^{2\min} = c_{13}^{2\min}(c_2,\beta) \, & \! \coloneqq
\frac{d_2(c_2,\beta)^2}{n-1}
+ \frac{d_3(c_2,\beta)^2}{c_2 - c_2^{\min}(\beta)}
- \frac{\bigl(c_{n,p} + (n- 2)c_2\bigr)^2}{n^2 -3n+3}
- c_2^2 .
\endaligned
\ee
The quadratic form $\psiHscabeta$ is thus positive-definite for $c_2$ and $c_{13}$ sufficiently large, and more precisely if
\bel{art2 -c2c13-cond}
c_2 > c_2^{\min}(\beta) , \qquad
c_{13}^2 > c_{13}^{2\min}(c_2,\beta) .
\ee
Since $\psiHtenbeta$ is also positive-definite,\footnote{We have $d_1(2a_{n,p}) \geq d_1(a_{n,p})=n^2 -4n+5$, which leads to $c_2^{\min}(2a_{n,p}) \geq c_2^{\min}(a_{n,p})=(n-3)((n- 2)^3+2n-3)/(n-1)+(n-3)a_{n,p}\geq 0$.  Thus,~\eqref{art2 -c2c13-cond} implies $c_2>0$, so that the coefficient of $|\nablaslash w|^2$ in $\psiHtenbeta$ is positive.} the functional $\PsiHquabeta[w]$ is also positive-definite, thus concluding the proof.
\ese
\end{proof}


\subsubsection{Semi-coercivity of the shell dissipation functional}

\bse
We are interested in the coercivity of $\Psi^\notreH_\beta[u]$ modulo the average $\la u\ra$. To this aim, we work out a lower bound on $\Psi^\notreH_\beta$, starting from
\bel{art2 -eq662}
\Psi^\notreH_\beta[u] \geq \PsiHquabeta[w] - \bigl| \PsiHbil[u;w] \bigr| + \frac{c_2}{n-1} \ssrmA^\lambda[u].
\ee
The term $\ssrmA^\lambda[u]$ is already coercive (modulo the average), via (essentially) the harmonic stability condition. 

Equipped with a compact form~\eqref{art2 -psiHscabeta-re} of $\psiHscabeta$, we bound the combinations $w_1,w_2$ of derivatives of~$w$ that appear in the bilinear terms~\eqref{art2 -PsiHbil-expr}:
\be
|w_1| \leq d_4 \psiHscabeta[w]^{1/2} , \qquad
|w_2| \leq d_5 \psiHscabeta[w]^{1/2} ,
\ee
with constants
\be
\aligned
d_4 = d_4(c_{13},c_2,\beta) \, & \! \coloneqq \frac{|c_{n,p}|}{(n-1) \sqrt{c_2 -c_2^{\min}}}
+ \frac{1}{(n-1) \sqrt{c_{13}^2 -c_{13}^{2\min}}} \Bigl|\frac{c_2 c_{n,p}}{a_{n,p}} + \frac{c_{n,p} d_3}{c_2 -c_2^{\min}} \Bigr| ,
\endaligned
\ee
and 
\be
\aligned
d_5 = d_5(c_{13},c_2,\beta) \, & \! \coloneqq \frac{|c_{n,p}|c_2}{a_{n,p}\sqrt{n-1}}
+ \frac{1}{\sqrt{c_2 -c_2^{\min}}} \biggl| c_{130} + \frac{c_2 c_{n,p}}{(n-1)a_{n,p}} d_1 \biggr| \\
& \quad + \frac{1}{\sqrt{c_{13}^2 -c_{13}^{2\min}}} \biggl|\frac{c_2 c_{n,p}}{(n-1)a_{n,p}} d_2 + \Bigl( c_{130} + \frac{c_2 c_{n,p}}{(n-1)a_{n,p}} d_1 \Bigr) \frac{d_3}{c_2 -c_2^{\min}} \biggr| .
\endaligned
\ee
We deduce that
\be
\aligned
\bigl| \PsiHbil[u;w] \bigr|
& \leq \fint_{\Lambda_r} \psiHscabeta[w]^{1/2} \Bigl( d_4 |\Deltaslash u| + d_5 |u| \Bigr) d\chi
\\
& \leq \frac{1}{2} \PsiHquabeta[w]
+ d_4^2 \fint_{\Lambda_r} (\Deltaslash u)^2 d\chi
+ d_5^2 \fint_{\Lambda_r} u^2 d\chi .
\endaligned
\ee

Inserting into~\eqref{art2 -eq662} this bound on $\PsiHbil[u;w]$ yields
\bel{art2 -eq667}
\Psi^\notreH_\beta[u]
\geq \frac{1}{2} \PsiHquabeta[w]
+ \frac{c_2}{n-1} \ssrmA^\lambda[u]
- d_4^2 \fint_{\Lambda_r} (\Deltaslash u)^2 d\chi
- d_5^2 \fint_{\Lambda_r} u^2 d\chi .
\ee
The explicit expression~\eqref{ssAalpha-quaform-0} yields the following lower bound, which is precisely the amount of coercivity needed to absorb the unfavorable Laplacian term in~\eqref{art2 -eq667}:
\be
\aligned
\ssrmA^\lambda[u]
& \geq \fint_{\Lambda_r} \Bigl( \bigl|(\nablaslash^2 u)^\circ\bigr|^2 + |\nablaslash u|^2 + \Bigl(n- 2 + \frac{1}{n-1}\Bigr) (\Deltaslash u)^2
- c_{n,p} u \Deltaslash u
\Bigr) \, d\chi
\\
& \geq \fint_{\Lambda_r} \Bigl(
\bigl|(\nablaslash^2 u)^\circ\bigr|^2 + |\nablaslash u|^2
+ (n- 2) (\Deltaslash u)^2
- \frac{n-1}{4} c_{n,p}^2 u^2 \Bigr) \, d\chi .
\endaligned
\ee
Substitution into~\eqref{art2 -eq667} eliminates the unfavorable $d_4^2$ term and gives the next lemma.
\ese

\begin{lemma}[Semi-coercivity of the Hamiltonian dissipation functional. II]
One has 
\bel{art2 -eq667bis}
\aligned
\Psi^\notreH_\beta[u]
& \geq \frac{1}{2} \PsiHquabeta[w]
+ \Bigl( \frac{c_2}{n-1} - \frac{d_4^2}{n- 2} \Bigr)
\fint_{\Lambda_r} \Bigl(
\bigl|(\nablaslash^2 u)^\circ\bigr|^2 + |\nablaslash u|^2
+ (n- 2) (\Deltaslash u)^2 \Bigr) \, d\chi
\\
& \quad
- \Bigl( \frac{1}{4} c_2 c_{n,p}^2 + d_5^2 \Bigr) \fint_{\Lambda_r} u^2 d\chi .
\endaligned
\ee
\end{lemma}

This lemma is almost \autoref{art2 -lema-gammagamma}, but there remains to show that the constants here can be taken to be uniform in the harmonic limit $p\to n-2$.


\subsection{Closing the construction}

\subsubsection{\texorpdfstring{The last constants $c_2,c_{13}$}{The last constants c2 and c13}}

We must now choose the constants $c_2$, $c_{13}$ so as to make the coefficient of $\ssrmA^\lambda[u]$ positive, namely $d_4 < \sqrt{(n- 2) c_2 / (n-1)}$.
This (for both values $\beta\in\{a_{n,p},2a_{n,p}\}$) is equivalent to
\bel{art2 -c2c13-new}
\aligned
c_2 & > c_2^! \coloneqq \max_{\beta\in\{a_{n,p},2a_{n,p}\}} \Biggl( \frac{c_2^{\min}(\beta)}{2} + \sqrt{\frac{(c_2^{\min}(\beta))^2}{4} + \frac{c_{n,p}^2}{(n-1)(n- 2)}} \Biggr) ,
\\
c_{13}^2 & > c_{13}^{2!}(c_2) \coloneqq \max_{\beta\in\{a_{n,p},2a_{n,p}\}} \Biggl( c_{13}^{2\min}(c_2,\beta)
+ \biggl(\frac{c_2 c_{n,p}/a_{n,p} + c_{n,p} d_3(c_2,\beta)/(c_2 -c_2^{\min}(\beta))}{\sqrt{(n-1)(n- 2) c_2} - c_{n,p}/\sqrt{c_2 -c_2^{\min}(\beta)}}\biggr)^2 \Biggr) .
\endaligned
\ee
The first inequality is satisfied for sufficiently large $c_2$ (depending only on $n,p$), and then the second one is satisfied for sufficiently large $c_{13}$ (depending on $n,p,c_2$).
Since~\eqref{art2 -c2c13-new} implies the inequalities $c_2>c_2^{\min}$ and $c_{13}^2>c_{13}^{2\min}$ previously stated in~\eqref{art2 -c2c13-cond}, any such choice of constants $c_2,c_{13}$ leads to the inequality~\eqref{art2 -eq667bis} with a positive coefficient for $\ssrmA^\lambda[u]$.
The (negative) coefficient of $\fint u^2d\chi$ in this inequality can then be evaluated, as well as the constant $\gamma^{(1)}$ coming from coercivity of the term $\PsiHquabeta[w]$ defined in~\eqref{art2 -PsiHquabeta-def}.


\subsubsection{Uniform control in the harmonic limit}

\bse
This establishes \autoref{art2 -lema-gammagamma}, apart from whether the coefficients~$\gamma^{(1)},\gamma^{(2)},\gamma^{(3)}$ can be taken to remain bounded as $p\to n- 2$.
A concrete choice, initially analyzed in the limit $p\to n- 2$, is given by
$c_2 = c_2^! + 1$ and $c_{13}^2 \coloneqq \max\{0,c_{13}^{2!}\}+1$. 
Let us consider the $p\to n- 2$ limit of various constants ($\beta\in\{a_{n,p},2a_{n,p}\}$)
\be
\aligned
d_1(\beta) & \to n^2 -4n+5 ,
\\
c_2^{\min}(\beta) , c_2^! , (c_2 - 1) & \to \frac{n-3}{n-1}\bigl((n- 2)^3+2n-3\bigr) \eqqcolon C_2(n) ,
\\
d_2(c_2,\beta) & \to (n- 2) (C_2(n)+1) ,
\\
d_3(c_2,\beta) & \to \frac{n- 2}{n-1} (n^2 -4n+5) (C_2(n)+1) .
\endaligned
\ee
We then check that $\PsiHquabeta[w]$ is uniformly coercive as $p\to n- 2$.  The coefficients in the tensorial part $\psiHtenbeta$ given in~\eqref{art2 -PsiHquabeta-def} are bounded below by $1/(n-1)$.
For the scalar part~$\psiHscabeta[w]$, given in~\eqref{art2 -psiHscabeta-re}, the factors $c_2 -c_2^{\min}(\beta) \geq c_2 -c_2^! = 1$ and $c_{13}^2 -c_{13}^{2\min}(c_2,\beta) \geq c_{13}^2 -c_{13}^{2!}=1$ are suitably bounded below while $d_1,d_2,d_3$ remain bounded in this limit.

We consider the limits of further constants (again for $\beta\in\{a_{n,p},2a_{n,p}\}$)
\be
\aligned
c_{13}^{2\min} & \to
\Biggl( \frac{(n- 2)^2}{n-1}
+ \frac{(n- 2)^2(n^2 -4n+5)^2}{(n-1)^2}
- \frac{2n^2 -7n+7}{n^2 -3n+3} \Biggr)
(C_2(n)+1)^2 \eqqcolon C_{132}(n) ,
\endaligned
\ee
while 
\be
\aligned
c_{13}^{2!} & \to C_{132}(n) + \frac{(n^2 -4n+5)^2}{(n-1)(n- 2)} (1+C_2(n)) ,
\\
\frac{c_2}{n-1} - \frac{d_4^2}{n- 2}
& \to \frac{(n- 2)(1+C_2(n))}{(n-1)(n- 2)+(n^2 -4n+5)^2 (1+C_2(n))} ,
\endaligned
\ee
which is positive (for $n\geq 3$).
Finally, we end by observing that $c_{130}$ (given in~\eqref{art2 -PsiHbil-expr}) remains finite in the $p\to n- 2$ limit, and that $d_5$ also does, thanks to the lower bounds on $c_2 -c_2^{\min}$ and $c_{13}^2 -c_{13}^{2\min}$.
The same choice is uniform on the full exponent interval. Indeed, the quotients $c_{n,p}/a_{n,p}$ and all the quantities in~\eqref{art2 -c2c13-new} extend continuously to the closed interval $p\in[p_n^\flat,n- 2]$. The added unit margins keep $c_2 -c_2^{\min}$, $c_{13}^2 -c_{13}^{2\min}$, and the denominator in the definition of $c_{13}^{2!}$ strictly positive. The resulting coercivity coefficients are therefore continuous and positive on this compact interval. Their minima are positive and depend only on~$n$, which proves the dimension-only uniformity asserted in \autoref{art2 -lema-gammagamma}.
\ese
%


\subsubsection{\texorpdfstring{Constants in dimension $n=3$}{Constants in dimension n=3}}

Concretely, for $n=3$, we find $C_2(3)=0$, $C_{132}(3) =  1/6$, and 
\be
\aligned
\gamma^{(2)} = \frac{c_2}{n-1} - \frac{d_4^2}{n- 2}
& \xrightarrow{n=3,p\to 1} \frac{1}{6} ,
\\
\gamma^{(3)} = d_5^2 + \frac{n-1}{4(n- 2)} d_4^2 c_{n,p}^2 & \xrightarrow{n=3,p\to 1} \Bigl( \sqrt{2} + 9/2 + \frac{11}{2\sqrt{3}} \Bigr)^2 \simeq 82.6 ,
\endaligned
\ee
whose ratio has a finite limit $\gamma^{(3)}/\gamma^{(2)}\simeq 496$.
This ratio is obtained for the particular choice of $c_2,c_{13}^2$ defined above.
A more refined approach is to minimize the ratio over $c_2,c_{13}$.  We have found numerically that $\gamma^{(3)}/\gamma^{(2)}$ can be taken to be arbitrarily close to the minimum ($\simeq 222$) of a rational function $(x+1)^2(x^2+1)^2/(x^2/2 - 1)$ for $x\in(\sqrt{2},+\infty)$, provided $p$ is close enough to $n- 2=1$ (in dimension $3$).

This completes the proof of \autoref{art2 -lema-gammagamma} and therefore the proof of the shell stability for the localized Hamiltonian operator stated in \autoref{art2 -prop-dj39}. 


\section{Proof of the momentum shell stability}
\label{section-9}

\subsection{A class of functionals}

\subsubsection{The shell identity}

Given a positive constant $\cperp$ (later taken to be $\cperp=1$), the momentum shell functional is defined as 
\be
\Phi^\notreM[Z] \coloneqq \frac{1}{2} \fint_{\Lambda_{r}} \bigl( 2 \, \cperp \, \Zperp{}^2 + |\Zpar|^2 \bigr) d\chi.
\ee
Evaluating its second derivative $- (\vartheta + a_{n,p}) (\vartheta + 2a_{n,p}) \Phi^\notreM[Z]$ yields (in particular) the terms $- 2 \, \cperp\Zperp\vartheta^2\Zperp- \Zpar\cdot\vartheta^2\Zpar$, which are evaluated in terms of lower-order derivatives thanks to the momentum equations. 
\bse
Specifically, the source functional reads 
\be
\Mu^\notreM[Z]
\coloneqq \fint_{\Lambda_r}  2 \, \bigl( \cperp\Zperp \xh_i + \Zpar_i\bigr) \, r^2 \notreM^\lambda[Z]_i \, d\chi, 
\ee
while the functional $\Chi^\notreM$ defined by the identity~\eqref{main-func-identity-MM} above reads 
\be
\aligned
\Chi^\notreM[Z]
& = (\vartheta + a_{n,p}) (\vartheta + 2a_{n,p}) \Phi^\notreM[Z] + \Mu^\notreM[Z]
\\
& = \fint_{\Lambda_{r}} \Bigl(
2 \cperp \bigl((\vartheta+a_{n,p}) \Zperp\bigr)^2
+ 2 \cperp \Zperp \Bigl( \Brs^{\lambda\perp\parallel}[\Zpar] + \ssB^{\lambda\perp\perp}[\Zperp] + \ssB^{\lambda\perp\parallel}[\Zpar] \Bigr)
\\
& \qquad\quad + \bigl|(\vartheta+a_{n,p}) \Zpar\bigr|^2
+ 2 \Zpar\cdot\bigl( \Brs^{\lambda\parallel\perp}[\Zperp] + \ssB^{\lambda\parallel\perp}[\Zperp] + \ssB^{\lambda\parallel\parallel}[\Zpar] \bigr)
\Bigr) d\chi, 
\endaligned
\ee
Hence, we have
\bel{art2 -equa-dk392} 
\aligned
\Chi^\notreM[Z]
& = \fint_{\Lambda_{r}} \Bigl(
2 \cperp \bigl((\vartheta+a_{n,p}) \Zperp\bigr)^2
+ \cperp |\nablaslash\Zperp|^2
+ 2(n-1) \cperp \Zperp{}^2 \\
& \qquad\quad
+ \bigl|(\vartheta+a_{n,p}) \Zpar\bigr|^2
+ 2 \, \bigl|\Sym(\nablaslash\Zpar) \bigr|^2
+ (a_{n,p}+1)|\Zpar|^2
\\
& \qquad\quad
+ 2 \, (\cperp+1) \Zperp\nablaslash\cdot\Zpar
+ \cperp(\vartheta-1) \Zpar\cdot\nablaslash\Zperp
- \Zpar\cdot\bigl(\vartheta+a_{n,p}+1\bigr) \nablaslash\Zperp
\Bigr) d\chi .
\endaligned
\ee
The first two lines produce the expected $H^1$-type norm of~$Z$ {\it except} for the ``missing'' anti-symmetric part of $\nablaslash\Zpar$. The next two terms are cross-terms with at most one derivative of~$Z$ in each factor.
\ese

\bse
The last term in \eqref{art2 -equa-dk392} spoils positivity of (the integrand of) the functional~$\Chi^\notreM$, and requires the introduction of the radial integration functional~$\Upsilon^\notreM$, which is the integral of a quadratic form in $Z$ and its derivatives. Specifically, a general functional of the form needed here, ensuring that
\be
\Psi^\notreM_\beta[Z] = \Chi^\notreM[Z] - (\vartheta +\beta) \Upsilon^\notreM[Z], 
\ee
involves at most first-order derivatives of~$Z$, is given by 
\be
\Upsilon^\notreM = \fint_{\Lambda_r} \bigl( - \Zpar\cdot\nablaslash\Zperp + \onec \Zperp{}^2 + \twoc |\Zpar|^2
\bigr) d\chi. 
\ee
Here, the constants $\onec,\twoc\in\RR$ remain to be determined. With this notation, we evaluate the two momentum dissipation functionals and find 
\bel{art2 -equa--115} 
\aligned
\Psi^\notreM_\beta[Z]
& = \fint_{\Lambda_{r}} \Bigl(
2 \, \cperp \bigl((\vartheta+a_{n,p}) \Zperp\bigr)^2
+ \cperp |\nablaslash\Zperp|^2
+ 2(n-1) \cperp \Zperp{}^2
- \onec \Zperp (2\vartheta +\beta) \Zperp
\\
& \qquad\quad
+ \bigl|(\vartheta+a_{n,p}) \Zpar\bigr|^2
+ 2 \, \bigl|\Sym(\nablaslash\Zpar) \bigr|^2
+ (a_{n,p}+1)|\Zpar|^2
- \twoc \Zpar \cdot (2\vartheta +\beta) \Zpar
\\
& \qquad\quad
+ 2 \, (\cperp+1) \Zperp\nablaslash\cdot\Zpar
+ \Bigl( (\cperp+1) \vartheta + (\beta-a_{n,p}-1- \cperp) \Bigr) \Zpar\cdot\nablaslash\Zperp
\Bigr) d\chi .
\endaligned
\ee
\ese


\subsubsection{Obstruction to having a non-negative integrand}

In~\eqref{art2 -equa--115}, let us consider first the terms that are quadratic in first derivatives of~$Z$, denoting the remaining terms as $\lot$, namely 
\be
\aligned
\Psi^\notreM_\beta[Z]
& = \fint_{\Lambda_{r}} \Bigl(
 2 \, \cperp (\vartheta\Zperp)^2
 + 2 \, \bigl|\Sym(\nablaslash\Zpar) \bigr|^2
\\[-.25ex]
&\qquad\qquad
+ \Bigl|\vartheta\Zpar + \frac{\cperp+1}{2} \nablaslash\Zperp\Bigr|^2
 - \frac{(\cperp-1)^2}{4} |\nablaslash\Zperp|^2
 + \lot\Bigr) d\chi .
\endaligned
\ee
This can be non-negative only if $\cperp = 1$, which we thus assume. Moreover, with this choice of $\cperp$, the sole derivatives of~$Z$ that are controlled by the available quadratic terms are $\vartheta\Zperp$, $\Sym(\nablaslash\Zpar)$, and $\vartheta\Zpar+\nablaslash\Zperp$.

The ($n$-component) vector fields $Z$ and $\vartheta Z$ are independent on a given sphere, so it is useful to ``complete the squares'' that involve $\vartheta Z$ and, consequently, focus attention on the remaining functional in~$Z$ and its angular derivatives. Specifically, we write 
\bel{art2 -equa--119}
\aligned
\Psi^\notreM_\beta[Z]
& = \Psi^\notreM_{\beta,1}[Z] + \Psi^\notreM_{\beta,0}[Z], 
\\
\Psi^\notreM_{\beta,1}[Z] & = \fint_{\Lambda_{r}} \Bigl(
\bigl|(\vartheta+a_{n,p}- \twoc) \Zpar + \nablaslash\Zperp\bigr|^2
+ \frac{1}{2} \bigl((2\vartheta+2a_{n,p}- \onec) \Zperp\bigr)^2 \Bigr) d\chi ,
\\
\Psi^\notreM_{\beta,0}[Z] & = \fint_{\Lambda_{r}} \biggl(
2 \, \bigl|\Sym(\nablaslash\Zpar) + \Zperp \gslash\bigr|^2
+ \bigl(1- \twoc^2 +a_{n,p} + \twoc (2a_{n,p} - \beta) \bigr)|\Zpar|^2
\\
& \qquad \qquad
+ \onec \Bigl( 2a_{n,p}- \beta - \frac{\onec}{2} \Bigr) \Zperp{}^2
- \bigl(2 - 2\twoc+3a_{n,p}- \beta\bigr) \underline{\Zpar \cdot \nablaslash\Zperp}
\biggr) d\chi .
\endaligned
\ee
Since the functional $\Psi^\notreM_{\beta,1}[Z]$ may vanish for certain choices of $\vartheta Z$, non-negativity of the integrand would require the remaining terms to be non-negative. However, the derivative $\nablaslash\Zperp$ only appears linearly in $\Psi^\notreM_{\beta,0}[Z]$.
For a given value of~$\beta$ this problem could be cured by choosing $\twoc$ in order to eliminate the $\Zpar \cdot \nablaslash\Zperp$ term, but we are interested in positivity simultaneously at the two values $\beta =a_{n,p}$ and $\beta =2a_{n,p}$.
 

\subsubsection{Completing the square}

We propose to resolve the above problem by {\it integrating by parts} the underlined term\footnote{A similar integration by parts could be done in~\eqref{art2 -equa-dk392} to change $- \Zpar\cdot\vartheta\nablaslash\Zperp$ into $\vartheta\Zperp\,\lambda^{- 2\expoP}\nablaslash\cdot(\lambda^{2\expoP}\Zpar)$, but the major difference is that this term has a prefactor of order~$1$, while \eqref{art2 -eq834} features a factor $d_\beta$ that will be taken to be small.} $\Zpar \cdot \nablaslash\Zperp$ and {\it completing the square in} $\Zperp$, as this scalar unknown then appears {\it without any derivative.} To proceed, we make a convenient choice of constants $\onec,\twoc$ and introduce a notation
\bel{art2 -dbetadef}
\onec = 0 , \qquad
\twoc = 1 , \qquad
d_\beta \coloneqq \frac{3a_{n,p}- \beta}{4(n-1)} .
\ee
We obtain 
\bel{art2 -eq834}
\aligned
\Psi^\notreM_{\beta,0}[Z] & = \Psi^\notreM_{\beta,\perp}[Z] + \Psi^\notreM_{\beta,\parallel}[\Zpar] ,
\\
\Psi^\notreM_{\beta,\perp}[Z] & = 2(n-1) \fint_{\Lambda_r} \Bigl(\Zperp + \frac{\nablaslash\cdot\Zpar}{n-1} + d_\beta \lambda^{- 2\expoP} \nablaslash\cdot(\lambda^{2\expoP}\Zpar) \Bigr)^2 d\chi ,
\\
\Psi^\notreM_{\beta,\parallel}[\Zpar]
& = \fint_{\Lambda_{r}} \biggl( 2 \, \bigl|\Sym(\nablaslash\Zpar)^\circ\bigr|^2 + 4(n-1) d_\beta \, |\Zpar|^2
\\
& \qquad \qquad - 4  d_\beta \, (\nablaslash\cdot\Zpar) \lambda^{- 2\expoP} \nablaslash\cdot(\lambda^{2\expoP}\Zpar)
- 2 (n-1)  d_\beta^2 \, \Bigl(\lambda^{- 2\expoP} \nablaslash\cdot(\lambda^{2\expoP}\Zpar) \Bigr)^2
\biggr) d\chi .
\endaligned
\ee
Altogether, the functional $\Psi^\notreM_\beta[Z]$ is the sum of positive functionals $\Psi^\notreM_{\beta,1}[Z]$ and~$\Psi^\notreM_{\beta,\perp}[Z]$ given in \eqref{art2 -equa--119} and \eqref{art2 -eq834}, and of~$\Psi^\notreM_{\beta,\parallel}[\Zpar]$, which has no sign in general.
We analyze this last contribution later on in \autoref{section=9-main}, using weighted Korn, Korn--Poincaré, and Hardy inequalities.


\subsection{A weighted Hardy inequality} 
\label{art2 -section=3.4}

Our final aim in this section is to establish the momentum shell stability condition, namely the semi-coercivity of the dissipation functionals~$\Psi^\notreM_\beta$.
We begin here by establishing a functional inequality that is useful to analyze an unfavorable term $\bigl(\lambda^{- 2\expoP}\nablaslash\cdot(\lambda^{2\expoP}\Zpar) \bigr)^2$ in these functionals.
We observe that $|\nablaslash\lambda|$ is positive on the boundary of~$\Lambda$, and that $\Hessslash\lambda$ remains finite thanks to our regularity assumption, 
hence $\lvert\Hessslash\lambda\rvert\lesssim|\nablaslash\lambda|$ in a neighborhood of the boundary.
For $m\geq 1$, the (rescaled) Hessian
\be
\lambda^{-m}\Hessslash_{ab}(\lambda^m) = m(m-1) \nablaslash_a(\log\lambda) \nablaslash_b(\log\lambda) + m \lambda^{-1} \Hessslash_{ab}\lambda
\ee
is thus the sum of a positive semi-definite matrix and of a term bounded by $|\nablaslash\log\lambda^m|$ near the boundary.
Away from the boundary, $\lambda^{-m}\lvert\Hessslash\lambda^m\rvert$ is bounded because $\lambda$ is smooth and positive.
We can thus introduce the Hardy-type constant $C_\Hardy^\lambda$ as the minimal constant satisfying the following pointwise condition for all tangent vectors~$\wvect$,
\bel{art2 -equa-chardy}
- \lambda^{- \expoP/2}\wvect^a \wvect^b \Hessslash_{ab}(\lambda^{\expoP/2})
\leq \Bigl( (C_\Hardy^\lambda)^2 + C_\Hardy^\lambda |\nablaslash\log\lambda^{\expoP/2}| \Bigr) \wvect^a \wvect^b \gslash_{ab} .
\ee

\begin{lemma}[A weighted Hardy inequality for vector fields]
\label{art2 -equainequ49}
In any localization domain $(\Lambda, d\chi = \lambda^{2 \expoP} d\xh)$,
for any vector field $\Zpar \in H^1_{- \expoP}(\Lambda)$ one has 
\be
\bigl\|\lambda^{- 2\expoP}\nablaslash\cdot(\lambda^{2\expoP}\Zpar) \bigr\|_{L^2_{- \expoP}(\Lambda_r)}^2
\leq \alpha_1 \, \bigl\| \Sym(\nablaslash \Zpar) \bigr\|_{L^2_{- \expoP}(\Lambda_r)}^2
+ \alpha_0 (C_\Hardy^\lambda)^2 \,  \bigl\| \Zpar \bigr\|_{L^2_{- \expoP}(\Lambda_r)}^2 
\ee
with constants $\alpha_1,\alpha_0$ that only depend on the dimension~$n$, and in particular remain uniformly bounded for large~$\expoP$.
\end{lemma}

\begin{proof} 
\bse
It suffices by density to establish the estimate for smooth vector fields for which the integrations by parts below are legitimate. We will rely on positivity of (the integral of) $|\Sym(\nablaslash\Zpar+m\nablaslash(\log\lambda)\Zpar)|^2$ for the (somewhat arbitrary) value $m=\expoP/2$ (which obeys $m\geq 1$ thanks to $\expoP\geq 2$).
As a preliminary step, we upper bound the cross terms that appear when expanding the square.  First, by integrating by parts, we find
\be
\aligned
& \frac{\expoP}{2} \fint_{\Lambda} \nablaslash_a(\log\lambda) \Zpar_b \nablaslash^a \Zpar{}^b d\chi
\\
& = \frac{1}{2} \fint_{\Lambda} |\Zpar|^2
(- \lambda^{- \expoP/2}\Deltaslash\lambda^{\expoP/2}) d\chi
- \frac{3\expoP^2}{8} \fint_{\Lambda}
|\Zpar|^2 |\nablaslash\log\lambda|^2 d\chi
\\
& \leq \frac{n-1}{2} \Bigl(
(C_\Hardy^\lambda)^2 \|\Zpar\|_{L^2_{- \expoP}(\Lambda)}^2
+ \frac{\expoP}{2} C_\Hardy^\lambda
\fint_{\Lambda} |\nablaslash\log\lambda|\,|\Zpar|^2 d\chi
\Bigr)
- \frac{3\expoP^2}{8} \fint_{\Lambda}
|\Zpar|^2 |\nablaslash\log\lambda|^2 d\chi
\\
& \leq \frac{(n-1)(n+3)}{8} (C_\Hardy^\lambda)^2 \|\Zpar\|_{L^2_{- \expoP}(\Lambda)}^2
- \frac{\expoP^2}{4} \fint_{\Lambda} |\Zpar|^2 |\nablaslash\log\lambda|^2 d\chi .
\endaligned
\ee
A similar calculation yields
\be
\aligned
& \frac{\expoP}{2} \fint_{\Lambda} \Bigl(\nablaslash_a(\log\lambda) \Zpar_b \nablaslash^b \Zpar{}^a\Bigr) d\chi \\
& = \fint_{\Lambda}
\bigl(- \lambda^{- \expoP/2}\Hessslash_{ab}(\lambda^{\expoP/2})\bigr)
\Zpar{}^a \Zpar{}^b d\chi
- \frac{3\expoP^2}{4} \fint_{\Lambda}
(\Zpar\cdot\nablaslash\log\lambda)^2 d\chi
\\
&\quad
- \frac{\expoP}{2} \fint_{\Lambda}
(\nablaslash\cdot\Zpar)\,
\Zpar\cdot\nablaslash(\log\lambda) d\chi
\\
& \leq (C_\Hardy^\lambda)^2 \fint_{\Lambda}  |\Zpar|^2 d\chi
+ \frac{\expoP}{2} C_\Hardy^\lambda \fint_{\Lambda}
|\Zpar|^2 |\nablaslash(\log\lambda)| d\chi
\\
&\quad
+ \frac{1}{8} \fint_{\Lambda} |\nablaslash\cdot\Zpar|^2 d\chi
- \frac{\expoP^2}{4} \fint_{\Lambda} \bigl(\Zpar\cdot\nablaslash(\log\lambda)\bigr)^2 d\chi .
\\
& \leq 2 (C_\Hardy^\lambda)^2 \|\Zpar\|_{L^2_{- \expoP}(\Lambda)}^2
 +\frac{\expoP^2}{16} \fint_{\Lambda} |\Zpar|^2 |\nablaslash(\log\lambda)|^2 d\chi
\\
&\quad
- \frac{\expoP^2}{4} \fint_{\Lambda}
\bigl(\Zpar\cdot\nablaslash(\log\lambda)\bigr)^2 d\chi
 + \frac{n-1}{8}
 \bigl\|\Sym(\nablaslash\Zpar)\bigr\|_{L^2_{- \expoP}}^2 .
\endaligned
\ee
Then we obtain
\be
\aligned
0 & \leq \fint_{\Lambda} \Bigl| \Sym\Bigl(\nablaslash\Zpar + \frac{\expoP}{2} \nablaslash(\log\lambda) \Zpar\Bigr) \Bigr|^2 d\chi
\\
& \leq \fint_{\Lambda} \bigl| \Sym(\nablaslash\Zpar) \bigr|^2 d\chi
+ \frac{\expoP^2}{8} \fint_{\Lambda} \Bigl( |\nablaslash(\log\lambda)|^2 |\Zpar|^2 + \bigl(\Zpar\cdot\nablaslash(\log\lambda)\bigr)^2 \Bigr) d\chi \\
& \quad + \frac{\expoP}{2} \fint_{\Lambda} \Bigl( \nablaslash_a(\log\lambda) \Zpar_b \nablaslash^b \Zpar{}^a + \nablaslash_a(\log\lambda) \Zpar_b \nablaslash^a\Zpar{}^b \Bigr) d\chi
\\
& \leq \frac{n+7}{8} \bigl\| \Sym(\nablaslash\Zpar) \bigr\|_{L^2_{- \expoP}(\Lambda)}^2
+ \frac{n^2+2n+13}{8} (C_\Hardy^\lambda)^2 \|\Zpar\|_{L^2_{- \expoP}(\Lambda)}^2 \\
& \quad - \frac{\expoP^2}{16} \fint_{\Lambda} |\Zpar|^2 |\nablaslash\log\lambda|^2 d\chi
- \frac{\expoP^2}{8} \fint_{\Lambda} \bigl(\Zpar\cdot\nablaslash(\log\lambda)\bigr)^2 d\chi .
\endaligned
\ee
The resulting bound on $\Zpar\cdot\nablaslash(\log\lambda)$ yields
\be
\aligned
& \bigl\|\lambda^{- 2\expoP}\nablaslash\cdot(\lambda^{2\expoP}\Zpar) \bigr\|_{L^2_{- \expoP}(\Lambda)}^2
\leq 8 \expoP^2 \fint_{\Lambda} \bigl(\Zpar\cdot\nablaslash(\log\lambda)\bigr)^2 d\chi + 2 \fint_{\Lambda} (\nablaslash\cdot\Zpar)^2 d\chi
\\
& \leq 8 (n^2+2n+13) (C_\Hardy^\lambda)^2 \|\Zpar\|_{L^2_{- \expoP}(\Lambda)}^2
+ (10n+54) \fint_{\Lambda} \bigl| \Sym(\nablaslash\Zpar) \bigr|^2 d\chi .
\endaligned
\qedhere
\ee
\ese
\end{proof}  


\subsection{Main result of this section} 
\label{section=9-main}

By analyzing the explicit form \eqref{art2 -equa--119}, \eqref{art2 -eq834} of the dissipation functional $\Psi^\notreM_\beta[Z]$ (with $\onec=0$ and $\twoc=1$) using weighted Korn and Hardy inequalities, we reach the following result.
It is instructive to compare the proof below to that of coercivity of $\ssrmB^{\lambda}$ in \autoref{art2 -lem-poinckorn} (and \autoref{art2 -lem-39harmmom}).
The functional $\ssrmB^{\lambda}$ included a $|\nablaslash\xiperp- \xipar|^2$ term allowing us to control rotation Killing vectors.
The functional $\Psi^\notreM_\beta[Z]$ does not include such a term, but includes the $L^2$ norm of~$\Zpar$.  In addition, the averages $T_{kl} \projP^k(Z)=\la 2\xh_k\Zperp+\Zpar_k\ra$ are not assumed to vanish and must thus be tracked throughout.
We emphasize that our statement below proves slightly more than is required for shell stability; namely, we cover some $\beta<a_{n,p}$ and some $\beta>2a_{n,p}$, and therefore (by linearity of $\Psi^\notreM_\beta$ with respect to~$\beta$) the whole interval between them. 

\begin{theorem}[Shell stability for the localized momentum operator] 
\label{art2 -prop-semi-mdiss}
Fix $p\in(p_n^\flat,n- 2)$ and consider the momentum operator associated with a localization function $\lambda$.
For every fixed $\lambda$, there is a threshold $a_*^\lambda>0$ such that the following estimates hold whenever $0<a_{n,p}\leq a_*^\lambda$. No smallness assumption is imposed on the localization domain itself.
The momentum dissipation functionals satisfy, for $\beta=\frac12a_{n,p}$ and $\beta=\frac52a_{n,p}$ and with an implicit constant depending only on the dimension,
\be
\aligned
& \Psi^\notreM_\beta[Z]
+ a_{n,p} \, \Bigl| \la 2 \,\xh \Zperp + \Zpar \ra - T(\Xi^\notreM)^{-1} \Kappa^\notreM[Z] \Bigr|^2
\gtrsim
\Ecal[Z], \qquad n \geq 4 .
\endaligned
\ee
Here, $\Ecal[Z]$ defined in~\eqref{art2 -equa-eofz}.
In dimension $n=3$, an additional term $\gamma_{\textnormal{Korn}}^{\notreM} \, \mathfrak{K}[Z]$ must be included on the left-hand side, with a constant $\gamma_{\textnormal{Korn}}^{\notreM}$ that depends on the geometry and exponents.
Moreover, there exists a constant $c>0$ (depending on the geometry and exponents) such that
\bel{Ecoer-H30}
\Ecal[Z]
\geq
c\Bigl( \bigl\|\vartheta\Zperp,\, (\vartheta\Zpar+\nablaslash\Zperp),\, \Zperp\bigr\|_{L^2_{- \expoP}(\Lambda)}^2 + \bigl\|\Zpar\bigr\|_{H^1_{- \expoP}(\Lambda)}^2 \Bigr) .
\ee
In particular, the \underline{momentum shell stability} condition in \autoref{def-shell-Mstab} holds.
\end{theorem}

\begin{proof}
\bse
{\it 1. Preliminaries.}
For definiteness, we immediately fix the two values of $\beta$ for which we will establish the semi-coercivity property 
\be
\beta =\frac{1}{2}a_{n,p} , \quad \text{and} \quad \beta =\frac{5}{2}a_{n,p}.
\ee
Note in passing that this covers a larger interval of values than the range appearing in the definition of shell stability. The coefficient defined in~\eqref{art2 -dbetadef} then evaluates to $d_\beta =5a_{n,p}/(8(n-1))$ and $d_\beta =a_{n,p}/(8(n-1))$, respectively.
We will consider each term of the decomposition $\Psi^\notreM_\beta[Z] = \Psi^\notreM_{\beta,1}[Z] + \Psi^\notreM_{\beta,\perp}[Z] + \Psi^\notreM_{\beta,\parallel}[\Zpar]$.  Recall the Korn inequality~\eqref{art2 -korn-ineq-11} which bounds the difference $\Zpar- \zeta_Z$ for a conformal Killing vector~$\zeta_Z$ that was fixed at the beginning of \autoref{art2 -section=3.3}.  Recall as well the bound \eqref{art2 -bothbounded34} on the translation components~$L_Z^\bullet$ of~$\zeta_Z$ in terms of~$\Ecal[Z]$.
Observe finally that it suffices to prove
\bel{art2 -sufficespsiecal}
\Psi^\notreM_\beta[Z]
+ \frac{5}{6} a_{n,p} |T_{k\bullet} L_Z^k|^2
\gtrsim \Ecal[Z] .
\ee
Indeed, we have 
\be
\aligned
& \Psi^\notreM_\beta[Z]
+ a_{n,p} \, \Bigl| \la 2 \,\xh \Zperp + \Zpar \ra - T(\Xi^\notreM)^{-1} \Kappa^\notreM[Z] \Bigr|^2
\\
& \geq \Bigl( \Psi^\notreM_\beta[Z] + \frac{5}{6} a_{n,p} |T_{k\bullet} L_Z^k|^2 \Bigr)
- 5 a_{n,p} \, \Bigl| \la 2 \,\xh \Zperp + \Zpar \ra - T(\Xi^\notreM)^{-1} \Kappa^\notreM[Z] - T_{k\bullet} L_Z^k \Bigr|^2 ,
\endaligned
\ee
in which the first term controls $\Ecal[Z]$ by our assumption~\eqref{art2 -sufficespsiecal}, and the second is bounded by $a_{n,p}\Ecal[Z]$ times geometric constants thanks to \autoref{art2 -prop-average-m} and \autoref{art2 -prop-k-fluctuations-m}.  For sufficiently small $a_{n,p}$ the second term is thus dominated by the first and their sum controls $\Ecal[Z]$ as desired.
\ese

\medskip

\bse
\noindent{\it 2. Parallel dissipation.}
We analyze first the most interesting contribution to the dissipation, which only involves~$\Zpar$,
\be
\aligned
\Psi^\notreM_{\beta,\parallel}[\Zpar]
& = \fint_{\Lambda_{r}} \biggl( 2 \, \bigl|\Sym(\nablaslash\Zpar)^\circ\bigr|^2 + 4(n-1) d_\beta \, |\Zpar|^2
\\
& \qquad \qquad - 4  d_\beta \, (\nablaslash\cdot\Zpar) \lambda^{- 2\expoP} \nablaslash\cdot(\lambda^{2\expoP}\Zpar)
- 2 (n-1)  d_\beta^2 \, \Bigl(\lambda^{- 2\expoP} \nablaslash\cdot(\lambda^{2\expoP}\Zpar) \Bigr)^2
\biggr) d\chi .
\endaligned
\ee
Splitting $\nablaslash\cdot\Zpar=\nablaslash\cdot(\Zpar- \zeta_Z)-(n-1)L^l\xh_l$, integrating by parts, and using the Korn inequality yields
\be
\aligned
& \fint_{\Lambda_{r}} \bigl( (\nablaslash\cdot\Zpar) \lambda^{- 2\expoP} \nablaslash\cdot(\lambda^{2\expoP}\Zpar) \bigr) d\chi
\\
& \leq \bigl\| \nablaslash\cdot(\Zpar- \zeta_Z) \bigr\|_{L^2_{- \expoP}(\Lambda_r)}
 \bigl\| \lambda^{- 2\expoP}
\nablaslash\cdot(\lambda^{2\expoP}\Zpar)
\bigr\|_{L^2_{- \expoP}(\Lambda_r)}
- (n-1) \fint_{\Lambda_{r}}
L^l \xh_l \lambda^{- 2\expoP}
\nablaslash\cdot(\lambda^{2\expoP}\Zpar)\,d\chi .
\\
& \leq \sqrt{n-1} \, C_{\Korn}^\lambda
\bigl\|\Sym(\nablaslash\Zpar)^\circ\bigr\|_{L^2_{- \expoP}(\Lambda_r)}
 \bigl\|\lambda^{- 2\expoP}
\nablaslash\cdot(\lambda^{2\expoP}\Zpar)
\bigr\|_{L^2_{- \expoP}(\Lambda_r)}
+ (n-1) L^l \la\Zpar\cdot\nablaslash\xh_l\ra
\\
& \leq
\frac{1}{4d_\beta} \,
\bigl\|\Sym(\nablaslash\Zpar)^\circ\bigr\|_{L^2_{- \expoP}}^2
+ d_\beta (n-1) (C_{\Korn}^\lambda)^2
\bigl\| \lambda^{- 2\expoP}
\nablaslash\cdot(\lambda^{2\expoP}\Zpar) \bigr\|_{L^2_{-\expoP}}^2
+ (n-1) L^l \la\Zpar{}_l\ra .
\endaligned
\ee
We arrive at the lower bound
\bel{art2 -PsiotrM56}
\aligned
\Psi^\notreM_{\beta,\parallel}[\Zpar] + \frac{5}{6} a_{n,p} |T_{k\bullet} L_Z^k|^2
& \geq \bigl\|\Sym(\nablaslash\Zpar)^\circ\bigr\|_{L^2_{- \expoP}(\Lambda_r)}^2
\\
& \quad
- d_\beta^2 (n-1) \bigl(2 + 4(C_{\Korn}^\lambda)^2\bigr)
\bigl\| \lambda^{- 2\expoP}
\nablaslash\cdot(\lambda^{2\expoP}\Zpar)
\bigr\|_{L^2_{- \expoP}(\Lambda_r)}^2
\\
& \quad + 4(n-1) d_\beta \, \Bigl( \|\Zpar\|_{L^2_{- \expoP}(\Lambda_r)}^2 - L^l \la\Zpar{}_l\ra \Bigr) + \frac{5}{6} a_{n,p} |T_{k\bullet} L_Z^k|^2 .
\endaligned
\ee
\ese

\medskip

\bse
\noindent{\it 3. Control of the non-derivative terms.}
We focus now on the last line in~\eqref{art2 -PsiotrM56}, which involves $\Zpar$ and~$L_Z^\bullet$.
By \eqref{art2 -dbetadef} we have $4(n-1)d_\beta = 3a_{n,p}- \beta\in\{5a_{n,p}/2,a_{n,p}/2\}$ for the values $\beta$ of interest, and we know $T>\Id$, thus we have
\be
\aligned
\quad & \unquad 4(n-1) d_\beta \, \Bigl( \|\Zpar\|_{L^2_{- \expoP}(\Lambda_r)}^2 - L^l \la\Zpar{}_l\ra \Bigr) + \frac{5}{6} a_{n,p} |T_{k\bullet} L_Z^k|^2
\\
& \geq 4 (n-1) d_\beta \, \Bigl(
\|\Zpar\|_{L^2_{- \expoP}(\Lambda_r)}^2
- L^l \la\Zpar{}_l\ra + \frac{1}{3} |L_Z^k|^2 \Bigr)
\\
& \geq 4 (n-1) d_\beta \, \Bigl( \frac{1}{6} \|\Zpar\|_{L^2_{- \expoP}(\Lambda_r)}^2 + \frac{1}{30} |L_Z^k|^2 \Bigr)
\geq \frac{1}{12} a_{n,p} \|\Zpar\|_{L^2_{- \expoP}(\Lambda_r)}^2 + \frac{1}{60} a_{n,p} |L_Z^\bullet|^2 .
\endaligned
\ee
The last term $|L_Z^\bullet|^2$ can be leveraged to get a control of $\Zperp$ modulo some small unfavorable multiple of~$\Ecal[Z]$: from \eqref{art2 -sigmaz-bnd},
\be
\aligned
\frac{1}{60} a_{n,p} |L_Z^\bullet|^2
&\geq \frac{1}{60} a_{n,p}
\|L_Z^k\xh_k\|_{L^2_{- \expoP}(\Lambda_r)}^2
\\
&\geq \frac{1}{120} a_{n,p}\|\Zperp\|_{L^2_{- \expoP}}^2
- \frac{1}{30} a_{n,p}
\Bigl(1+\frac{{(C_{\Korn}^\lambda)^2}}{n-1}\Bigr)\Ecal[Z] .
\endaligned
\ee
The coefficient above is $C_{\Korn}^\lambda$, because the estimate invoked is~\eqref{art2 -sigmaz-bnd}, obtained from Korn's inequality applied to $\Zpar- \zeta_Z$.
\ese

\medskip

\bse
\noindent{\it 4. Orthogonal and derivative dissipations.}
There  remain two other contributions in~$\Psi^\notreM_\beta$.
First, we have 
\be
\aligned
\Psi^\notreM_{\beta,\perp}[Z]
& = 2(n-1) \Bigl\|\Zperp + \frac{\nablaslash\cdot\Zpar}{n-1} + d_\beta \lambda^{- 2\expoP} \nablaslash\cdot(\lambda^{2\expoP}\Zpar) \Bigr\|_{L^2_{- \expoP}(\Lambda_r)}^2
\\
& \geq (n-1) \Bigl\|\Zperp + \frac{\nablaslash\cdot\Zpar}{n-1} \Bigr\|_{L^2_{- \expoP}(\Lambda_r)}^2
- 2(n-1) d_\beta^2 \Bigl\|\lambda^{- 2\expoP} \nablaslash\cdot(\lambda^{2\expoP}\Zpar) \Bigr\|_{L^2_{- \expoP}(\Lambda_r)}^2 .
\endaligned
\ee
The positive term gives the third term in $\Ecal[Z]$ (up to a scaling).
We return to the unfavorable term momentarily.
Next,
\be
\aligned
\Psi^\notreM_{\beta,1}[Z]
& = \bigl\|(\vartheta+a_{n,p}-1) \Zpar + \nablaslash\Zperp\bigr\|_{L^2_{- \expoP}(\Lambda_r)}^2
+ 2 \bigl\|(\vartheta+a_{n,p}) \Zperp\bigr\|_{L^2_{- \expoP}(\Lambda_r)}^2
\\
& \geq \frac{1}{n} \bigl\|(\vartheta-1) \Zpar + \nablaslash\Zperp\bigr\|_{L^2_{- \expoP}(\Lambda_r)}^2
+ \frac{1}{n-1/2} \bigl\|\vartheta\Zperp\bigr\|_{L^2_{- \expoP}(\Lambda_r)}^2
- \frac{a_{n,p}}{n-1} \Ecal[Z] .
\endaligned
\ee
The positive terms here coincide with the first two terms in $\Ecal[Z]$.
\ese

\medskip

\bse
\noindent{\it 5. Finalizing the semi-coercivity.}
Using $a_{n,p}/4\leq 2(n-1)d_\beta\leq 5a_{n,p}/4$ after collecting these inequalities together, we deduce that the dissipation functional obeys
\be
\aligned
& \Psi^\notreM_\beta[Z] + \frac{5}{6} a_{n,p} |T_{k\bullet}L_Z^k|^2 \\
& \geq
\frac{1}{n} \bigl\|  (\vartheta-1) \Zpar + \nablaslash\Zperp  \bigr\|_{L^2_{- \expoP}(\Lambda_r)}^2 
+ \frac{1}{n-1/2} \bigl\| \vartheta\Zperp  \bigr\|_{L^2_{- \expoP}(\Lambda_r)}^2 
+ (n-1) \Bigl\| \Zperp + \frac{\nablaslash\cdot\Zpar}{n-1} \Bigr\|_{L^2_{- \expoP}(\Lambda_r)}^2 \\
& \quad
+ \bigl\|\Sym(\nablaslash\Zpar)^\circ\bigr\|_{L^2_{- \expoP}(\Lambda_r)}^2
+ \frac{a_{n,p}}{120} \|\Zperp\|_{L^2_{- \expoP}(\Lambda_r)}^2 
+ \frac{a_{n,p}}{12} \|\Zpar\|_{L^2_{- \expoP}(\Lambda_r)}^2
\\
& \quad
- a_{n,p}\frac{n+29+ (C_{\Korn}^\lambda)^2}{30(n-1)} \Ecal[Z]
- \frac{25}{16} a_{n,p}^2
\frac{1 + (C_{\Korn}^\lambda)^2}{n-1}
\bigl\| \lambda^{- 2\expoP}
\nablaslash\cdot(\lambda^{2\expoP}\Zpar)
\bigr\|_{L^2_{- \expoP}(\Lambda_r)}^2 .
\endaligned
\ee
To apply the Hardy inequality in \autoref{art2 -equainequ49}, one must also estimate its full symmetric-gradient term. Since $0<a_{n,p}\leq1$ in the regime under consideration, the trace decomposition and the definition of $\Ecal$ give
\be
\|\Sym(\nablaslash\Zpar)\|_{L^2_{- \expoP}}^2
\lesssim_n a_{n,p}^{-1}\Ecal[Z],
\qquad
\|\Zpar\|_{L^2_{- \expoP}}^2\leq a_{n,p}^{-1}\Ecal[Z].
\ee
Consequently, we obtain
\be
\bigl\|\lambda^{- 2\expoP}\nablaslash\cdot(\lambda^{2\expoP}\Zpar)\bigr\|_{L^2_{- \expoP}}^2
\lesssim_n a_{n,p}^{-1}\bigl(1 + (C_\Hardy^\lambda)^2\bigr)\Ecal[Z]
\ee
and, altogether,  
\be
\aligned
&
\Psi^\notreM_\beta[Z] + \frac{5}{6} a_{n,p} |T_{k\bullet}L_Z^k|^2
\\
& \geq \biggl( \frac{1}{n+120} - a_{n,p} \frac{n+29+ (C_{\Korn}^\lambda)^2}{30(n-1)}
- C(n)a_{n,p}\bigl(1+(C_{\Korn}^\lambda)^2\bigr)\bigl(1+(C_\Hardy^\lambda)^2\bigr) \biggr) \Ecal[Z] ,
\endaligned
\ee
The Hardy estimate is used without an additional factor~$a_{n,p}$: controlling the trace costs $a_{n,p}^{-1}$ through the term $a_{n,p}\|\Zperp\|^2$, and the resulting bound remains of the smallness order required for the conclusion.
This estimate controls $\Ecal[Z]$ for $a_{n,p}$ small enough compared to expressions involving the dimension, Korn, Korn--Poincaré, and Hardy constants.
In particular, no smallness assumption is imposed on the localization domain itself.
To finish the proof, we observe that all norms in~\eqref{Ecoer-H30} are manifestly controlled by~$\Ecal[Z]$ except for $\|\nablaslash\Zpar\|_{L^2_{- \expoP}(\Lambda_r)}^2$.  By the Korn inequality~\eqref{art2 -korn-ineq-11} these derivatives reduce to those of the conformal Killing vector~$\zeta_Z$, hence to its coefficients $L_Z^k$ and~$L_Z^{kl}$, which we controlled along the way in~\eqref{art2 -bothbounded34} and~\eqref{art2 -LZbullbull-contro}, respectively.
\ese

This completes the proof of \autoref{art2 -prop-semi-mdiss} concerning the shell stability for the localized momentum operator.
\end{proof}

\begin{remark}
  The constant $c^{\notreM \lambda}$ appearing in \autoref{art2 -thm:informal-suff-stab} is read off from our results on momentum stability: \autoref{art2 -lem-39harmmom} (harmonic stability), \autoref{art2 -coro-radial-momen} (radial stability), and \autoref{art2 -prop-semi-mdiss} (shell stability).
Here is one explicit sufficient choice. Let $C_n\geq1$ dominate the constants depending only on the dimension that are implicit in the estimates used in the proofs of \autoref{art2 -prop-average-m}, \autoref{art2 -prop-k-fluctuations-m}, \autoref{art2 -equainequ49}, and \autoref{art2 -prop-semi-mdiss}, and set
\be
\aligned
\sigma_\lambda
& \coloneqq \min_{|\wvect|=1} \wvect^k\la\xh_k\xh_l\ra \wvect^l
= |S^{- 2}|^{-1}
\\
\mathfrak G_\lambda
& \coloneqq
\bigl(1+C_{\Korn}^\lambda\bigr)
\Bigl(1+(\CzeroKP+\ConeKP) |S^{- 2}| \Bigr),
\\
\mathfrak H_\lambda
& \coloneqq
\frac{n+29+(C_{\Korn}^\lambda)^2}{30(n-1)}
+ C_n \bigl(1+(C_{\Korn}^\lambda)^2\bigr)
\bigl(1+(C_\Hardy^\lambda)^2\bigr).
\endaligned
\ee
With the notation $\matV =\bigl(\la\xh_k\xh_l\ra- \la\xh_k\ra\la\xh_l\ra\bigr)$ introduced above, define
\be
\aligned
a_*^\lambda
\coloneqq \min\Biggl\{
&1,
\frac{\sqrt2}{\CzeroKP},
\frac{\sqrt2\,\sigma_\lambda}{\CzeroKP},
\frac{1}{\bigl(1+|\matV^{-1/2}|\bigr)^2}
\\
&\frac{1}{2(n+120)\mathfrak H_\lambda},
\frac{1}{20C_n(n+120)\mathfrak G_\lambda^2}
\Biggr\}.
\endaligned
\ee
The first four entries ensure the smallness assumptions used for harmonic and radial stability and in the fluctuation estimate~\eqref{art2 -fluctuations-silh-v}. The fifth leaves at least $1/(2(n+120))$ in the final lower bound for the momentum dissipation, while the last one absorbs the average--fluctuation error in the passage from~\eqref{art2 -sufficespsiecal} to shell stability. Since $a_{n,p}=2(n- 2 -p)$, one may therefore take the (non-optimal) choice
\be
c^{\notreM \lambda}
\coloneqq \frac12 a_*^\lambda.
\ee
\end{remark}


\

\paragraph*{Acknowledgments.} 

The authors were partially supported by the research project ANR- 23-CE40-0010-02: Einstein-PPF, entitled {\it ``Einstein constraints: past, present, and future''} funded by the Agence Nationale de la Recherche (ANR), as well as by the MSCA Staff Exchange Project 101131233: Einstein-Waves, entitled {\it ``Einstein gravity and nonlinear waves: physical models, numerical simulations, and data analysis'',} funded by the European Research Council (ERC). 
 


\appendix 

\section{Structure and normalization constants}
\label{sec-appendix-A}

\subsubsection{Structure coefficients}

We collect here some coefficients that appear in various operators:
\bel{equa-our-parame-00} 
\aligned
a_{n,p} & \coloneqq 2( n- 2 -p)
&& \text{(harmonic exponent)},
\\
b_{n,p} & \coloneqq 
2 + (n-3)(n- 2 -a_{n,p}), 
&&  
\\ 
c_{n,p} & \coloneqq a_{n,p} \bigl( 1 + (n- 2) ( n- 2 - a_{n,p}) \bigr), 
&&  
\\
d_{n,p} & \coloneqq {(n-1) a_{n,p} b_{n,p} \over (n- 2)^2+1}
&& \text{(ADM energy coefficient),} 
\endaligned
\ee  
as well as
\be
p^\flat_n \coloneqq {(n-1)(n-3) \over 2(n- 2)} < \frac{n- 2}{2} .
\ee
For $p\in(p^\flat_n, n- 2)$ the coefficients $a_{n,p},b_{n,p},c_{n,p},d_{n,p}$ are all positive.


\subsubsection{Functional inequality constants}

We list here some constants that play an important role for the Hamiltonian harmonic and radial stability conditions (cf.~\autoref{section-3}):
\begin{align}
\label{equa-H2110}
d^{\min/\max}_{n,p} & \coloneqq
\min/\max \Bigl(\frac{(n-1)c_{n,p}}{(n- 2)^2},\, d_{n,p}\Bigr) ,
\\
\label{equa-const-alpha}
\alpha_{n,p} & \coloneqq \frac{\sqrt{8(n^2 -3n+3)}}{\sqrt{n-1}}
\,\frac{\sqrt{1+a_{n,p}}}{c_{n,p}}, 
\\
\label{equa-cons-beta}
\beta_{n,p} & \coloneqq \frac{2 \sqrt{2}}{c_{n,p}} \min\Bigl((1+a_{n,p})^{1/2} , (1+a_{n,p}) (b_{n,p})^{-1/2} \Bigr) .
\end{align}
We also have set 
\bel{equa-cntilde}
\tcnp = c_{n,p} \Bigl( \frac{n-1}{8(n^2 -3n+3)(1+a_{n,p})} \Bigr)^{1/2} .
\ee


\subsubsection{Normalization constants}

We collect here the expressions of the structure constants that arise in our analysis. Constants that depend upon the area $\aire[\Lambda,\lambda] = \int_{\Lambda} \, d\chi$  of the localization function are needed, specifically 
\bel{equa-thetalambda}
\aligned
\theta^\lambda 
& \coloneqq {2 (n-1) \over (n- 2)^2 + 1} \, {|\Sphe^{n-1}| \over \aire[\Lambda,\lambda]}
\eqqcolon  {\theta_n \over \aire[\Lambda,\lambda]}, 
\\
\eta^\lambda 
& \coloneqq {2(n-1) |\Sphe^{n-1}| \over \aire[\Lambda,\lambda]} =: {\eta_n \over \aire[\Lambda,\lambda]}.
\endaligned
\ee
We also defined in~\eqref{equa-matrix-STU} several positive-definite structure matrices in terms of averages of coordinates, where $T,U$ are used throughout the paper, the covariance matrix~$\matV$ is used mostly in \autoref{section-7}, and~$\matG$ is used exclusively in the second-order Poincaré inequality~\eqref{equa-defineG}:
\be
\aligned
S^2 & = \bigl((S^2)_{kl}\bigr) \coloneqq (\la\xh_k\xh_l\ra) ,
& T & = \Id + S^2 \coloneqq  (\delta_{kl} + \la \xh_k\xh_l\ra),
\\
& & U & = \Id - S^2 \coloneqq  (\delta_{kl} - \la \xh_k\xh_l\ra),
\\
\matV & = S^2 - \langle\xh\rangle \otimes\langle\xh\rangle ,
& \matG& = \Id + S^2 - 2\langle\xh\rangle \otimes\langle\xh\rangle =U+2\matV.
\endaligned
\ee
We also use structure matrices~\eqref{equa-the-matrix} and~\eqref{Xi-pieces} constructed from the silhouette vectors:
\be
\aligned
\matQ^{(j)k} & \coloneqq \bigl\la 2 \xh_l \xi^{\normal (j) \perp} + \xi^{\normal (j) \parallel}{}_l\bigr\ra (T^{-1})^{lk} ,
\\
\zeroXi^{(i)}{}_k & \coloneqq \bigl\la \nablaslash_k \xi^{\normal (i)\perp} - \xi^{\normal (i)\parallel}{}_k\bigr\ra  ,
\\
(\Xi^\notreM)^{(i)}{}_{k}
& \coloneqq - \zeroXi^{(j)}{}_k + a_{n,p}  \matQ^{(j)l} \, T_{lk}
\\
& \, =
\bigl\la - \nablaslash_k \xi^{\normal (i) \perp} + 2 a_{n,p} \xh_k\, \xi^{\normal (i) \perp} \bigr\ra
+(1+a_{n,p}) \la \xi^{\normal (j) \parallel}{}_k \ra
.
\endaligned
\ee
We also have, for any $\xi=(\xipar,\xiperp)$,
\be
\aligned
\xi^{\fluc} & = \xi - \projP^k(\xi) \xistar_k , \qquad
\projP^k(\xi) = (T^{-1})^{kl} \, \bigl\la 2 \xh_l \xiperp + \xipar{}_l \bigr\ra ,
\\
\xistar_k & = (\xistarperp_k,\xistarpar_k) = (\xh_k, \nablaslash \xh_k) ,
\qquad
(\xistarpar_k)_l = \delta_{kl} - \xh_k \xh_l .
\endaligned
\ee


\subsubsection{Radial Hardy constant} 

The constant $c^{\notreH}_{\textnormal{radial}}$ enters the Hamiltonian shell condition~\eqref{equa-conditionH2} and depends upon averages of the Hamiltonian silhouette function. We introduce first 
\bel{bH10}
\aligned
\bnotreH_{1}& \coloneqq (n-1) c_{n,p}\la\nu^\normal\ra - (n- 2)^2\la\Deltaslash\nut^\normal\ra, 
\\
\bnotreH_{0} & \coloneqq (n^2 -4n+5){c_{n,p} \over a_{n,p}} \Bigl( d_{n,p} \la\nu^\normal\ra - \la\Deltaslash\nut^\normal\ra \Bigr)
\endaligned
\ee
and we consider the critical exponents of the second-order operator 
\bel{betapm-def-app}
\aligned
& - \bnotreH_{1}\vartheta(\vartheta+a_{n,p}) + \bnotreH_{0} \eqqcolon - \bnotreH_{1}  (\vartheta+\beta_-)(\vartheta+\beta_+), 
\qquad
 \beta_- < 0 < a_{n,p} < \astar < \beta_+, 
\endaligned
\ee
namely we find $\beta_- + \beta_+ = a_{n,p}$ and $\beta_- \beta_+ = -  \bnotreH_{0}/\bnotreH_{1}$. We also set 
\bel{equa-cplusminus} 
c_\pm \coloneqq {1 \over \bnotreH_{1} } {(n- 2) \beta_{\pm} + {c_{n,p} / a_{n,p}} \over \beta_+ - \beta_-},
\ee 
and the radial Hardy constant is then defined by
\bel{def-cnotreHradial}
c^{\notreH}_{\textnormal{radial}} \coloneqq  \max(  g^\notreH_I,   g^\notreH_J )
\ee
with  
\bel{equa--517-repeat}
\aligned 
g^\notreH_I & \coloneqq \frac{4 c_-^2}{(a_{n,p}-\beta_-)^2} + \frac{2 c_+^2}{(\beta_+-a_{n,p})^2} + \frac{2 c_+^2 }{(2 \beta_+ - a_{n,p})(\beta_+ - a_{n,p})} ,
\\
g^\notreH_J & \coloneqq \frac{16 c_+^2}{(2 \beta_+ - a_{n,p})^2} + \frac{8 c_-^2}{(a_{n,p} - 2\beta_-)^2} + \frac{2 c_-^2}{(a_{n,p}- 2\beta_-) (a_{n,p}-\beta_-)} .
\endaligned
\ee 


\section{Harmonic--spherical decompositions of the localized Einstein constraints}
\label{sec-appendix-B}

\subsection{Decomposition of basic operators}
\label{art2 -section=5.1}
 
\subsubsection{Radius/angle split and first-order operators}

In this section we derive the proposed decompositions of the main differential operators of interest in this paper. Indeed, the expressions below are instrumental in understanding the sharp asymptotics of solutions to the Einstein constraints. Throughout, we work in an asymptotic end $\Omega_R$ which is included in the exterior of a ball in the Euclidean space $(\RR^n,\delta)$ where $\delta =(\delta_{ij})$ is the standard metric, with which we implicitly sum repeated indices.  

To separate between radial and angular directions we introduce $r=|x|$ and the unit vector $\xh_i = x_i / r \in \Lambda$. We decompose all tensors into components parallel to~$\Lambda$ (angular components) and those orthogonal to it, which are along the vector $\vartheta = r\del_r= x_i\del_i$.
For instance a vector field $Z\colon\Omega\to\RR^n$ is decomposed as
\bel{art2 -Zi-split}
Z_i = \xh_i\Zperp+ \Zpar_i , \qquad
\Zperp \coloneqq \delta(Z,\del_r) = \xh_j Z_j , \qquad
\Zpar_i \coloneqq Z_i - \xh_i\xh_j Z_j .
\ee
Given any sufficiently regular function $f\colon\RR^n\to\RR$, we decompose its gradient (in Euclidean space) as
\bel{art2 -delif}
\del_i f = \frac{1}{r} \bigl( \xh_i \, \vartheta f + \dslash_i f \bigr), \qquad
\vartheta f \coloneqq r\del_r f = x_i \del_i f, \qquad
\dslash_i f \coloneqq r \del_i f - \xh_i \, \vartheta f.
\ee
Observe that the angular derivatives~$\dslash_i$ define $n$~vector fields on~$\Lambda$ that span the tangent space $T\Lambda$ and are subject to the linear relation $\xh_i\dslash_i= 0$. The radial and angular derivatives $\vartheta$ and $\dslash_i$ obey the following identities,
\bel{art2 -basicslash}
\aligned
\vartheta r & = r, \qquad
\vartheta \xh_i = 0, \qquad
\dslash_i r = 0, \qquad
\dslash_i \xh_j = \delta_{ij} - \xh_i \xh_j,
\\
\xh_i \dslash_i & = 0, \qquad
[\vartheta, \dslash_i] = 0, \qquad
[\dslash_i, \dslash_j] = \xh_i \dslash_j - \xh_j \dslash_i.
\endaligned
\ee
Powers of~$r$ and derivatives can be commuted through~$\vartheta$ as follows for any $\alpha\in\RR$: 
\be
\vartheta(r^\alpha f) = r^\alpha (\vartheta+ \alpha) f, \qquad
\del_i\vartheta f = (\vartheta+1) \del_i f.
\ee
For instance, when studying the Laplace operator in a variational framework, we will consider the dot product of gradients of a pair of functions $f,g$, which decomposes as
$
\del_i f \del_i g = \frac{1}{r^2} \bigl( \vartheta f \vartheta g + \dslash_i f \dslash_i g \bigr), 
$
thanks to the orthogonality of $\xh_i\vartheta f$ and $\dslash_i g$, and likewise with $f$ and~$g$ interchanged.
%


\subsubsection{Elementary second-order operators}

The standard Laplacian of a function $f$ is expressed as
\be
\Delta f = \del_i \del_i f
= r^{- 2} \Bigl( (\vartheta + n - 2) \vartheta f + \Deltaslash f \Bigr), \qquad
\Deltaslash = \dslash_i \dslash_i, 
\ee
while the spherical Laplacian~$\Deltaslash$ obeys
\be
\aligned
\Deltaslash \xh_i & = - (n-1) \xh_i, \qquad
[\dslash_i, \Deltaslash] = 2 \xh_i \Deltaslash - (n-3) \dslash_i.
\endaligned
\ee
The Hessian of $f$ is decomposed as
\bel{art2 -hessf-1}
\aligned
\del_i \del_j f
& = \xh_i r^{-1} \vartheta(\xh_j r^{-1}\vartheta f) + \xh_i r^{-1}\vartheta (r^{-1}\dslash_j f)
+ r^{-1} \dslash_i(\xh_j r^{-1}\vartheta f)
+ r^{- 2} \dslash_i \dslash_j f
\\
& = r^{- 2} \Bigl(
\xh_i \xh_j (\vartheta^2 - \vartheta) f
+ \xh_i \dslash_j \vartheta f
+ \xh_j \dslash_i \vartheta f
+ \dslash_i \dslash_j f
- \xh_i \dslash_j f
+ (\delta_{ij} - \xh_i \xh_j) \vartheta f \Bigr).
\endaligned
\ee
This expression is symmetric as expected under the exchange of $i$ and~$j$, since the combination $\dslash_i \dslash_j f - \xh_i \dslash_j f$ is symmetric in view of the commutator~\eqref{art2 -basicslash}.

The link with the covariant Laplacian and Hessian arises as follows. While we will perform most calculations using the differential operators~$\dslash_i$, it is instructive to also consider the action of the Levi-Civita connection~$\nablaslash$ of the induced metric on $\Lambda\subset\RR^n$. For a pair of vector fields $Y,Z$ on $\Lambda$ one has $(\nablaslash_Y Z)_j = Y_i \bigl(\dslash_i Z_j + Z_i \xh_j\bigr)$, where $i,j=1,\dots,n$, with an implicit summation on~$i$.  The divergence of a vector $Z\in T\Lambda$ is thus $\nablaslash\cdot Z = \dslash_i Z_i$.  Note that we are describing here vector fields on~$\Lambda$ in terms of their components in the standard basis of~$\RR^n$.

The covariant Hessian and Laplacian read
\be
(\Hessslash f)_{ij} = \dslash_i \dslash_j f + \xh_j \dslash_i f, \qquad
\nablaslash\cdot\nablaslash f = \dslash_i \dslash_i f = \Deltaslash f.
\ee
The Hessian differs slightly from the combination $\dslash_i \dslash_j f - \xh_i \dslash_j f$ appearing in~\eqref{art2 -hessf-1}. It is nevertheless symmetric. The Laplacian, on the other hand, coincides with $\Deltaslash f$ defined earlier as a sum of $n$ terms $\dslash_i\dslash_i f$, $i=1,\dots,n$.

As we will manipulate higher derivatives as well, it is interesting to consider the covariant derivative of a tensor~$Z$ in $T^2\Lambda$, its divergence, and its double divergence, that is, 
\bel{art2 -double-div}
\aligned
(\nablaslash_Y Z)_{jk} & = Y_i \bigl(\dslash_i Z_{jk} + \xh_j Z_{ik} + \xh_k Z_{ji} \bigr), 
\qquad \quad 
(\nablaslash\cdot Z)_j = \dslash_i Z_{ij} + \xh_j Z_{ii}, \\
\nablaslash\cdot(\nablaslash\cdot Z) & = \dslash_j \dslash_i Z_{ij} + (n-1) Z_{ii} = \dslash_i \dslash_j Z_{ij} + (n-1) Z_{ii}.
\endaligned
\ee
It is also interesting to consider the curvature of~$\nablaslash$, expressed in terms of the unit sphere metric~$\gslash_{ab}$, where $a,b$ are abstract Penrose indices for the tangent bundle on~$\Lambda$:
\be
\aligned
\Rslash_{abcd} & = \gslash_{ac} \gslash_{bd} - \gslash_{ad} \gslash_{bc}, &
\Rslash_{ac} & = (n- 2) \gslash_{ac}, \qquad
\hskip2.cm 
\Rslash = (n- 2) (n-1),
\\
[\nablaslash_a,\nablaslash_b] v^b & = - (n- 2) v_a, &
[\nablaslash_a, \Deltaslash] f
& = [\nablaslash_a, \nablaslash_b] \nablaslash^b f = - (n- 2) \nablaslash_a f.
\endaligned
\ee


\subsection{Decomposition of the localized Hamiltonian operator}
\label{art2 -section=5.2}

We now decompose the operator 
\be
r^4 \lambda^{2\expoP} \notreH^\lambda[u]  
= r^{4-n+2p} \, \del_i \del_j \bigl( \lambda^{2\expoP} r^{n- 2p} \del_i \del_j u \bigr)
+ (n- 2) r^{4-n+2p} \, \Delta \bigl( \lambda^{2\expoP} r^{n- 2p} \Delta u \bigr)
\ee
into radial and angular contributions. The Laplacian part $\PLap\coloneqq r^{4-n+2p} \, \Delta \bigl( \lambda^{2\expoP} r^{n- 2p} \Delta u \bigr)$ is simpler so we decompose it first, that is, 
\be
\aligned 
\PLap 
& = \bigl((\vartheta+n- 2 - 2p)(\vartheta+2n-4- 2p)+ \Deltaslash\bigr) \Bigl( \lambda^{2\expoP} \vartheta(\vartheta+n- 2)u + \lambda^{2\expoP} \Deltaslash u \Bigr)
\\
& = \lambda^{2\expoP} (\vartheta+n- 2 - 2p)(\vartheta+2n-4- 2p) \vartheta(\vartheta+n- 2)u \\
& \quad
+ \lambda^{2\expoP} (\vartheta+n- 2 - 2p)(\vartheta+2n-4- 2p) \Deltaslash u
+ \vartheta(\vartheta+n- 2) \Deltaslash(\lambda^{2\expoP} u) 
+ \Deltaslash( \lambda^{2\expoP} \Deltaslash u).
\endaligned
\ee
Next, for the Hessian part $\PHess \coloneqq r^{4-n+2p} \, \del_i \del_j \bigl( \lambda^{2\expoP} r^{n- 2p} \del_i \del_j u \bigr)$ we start with~\eqref{art2 -hessf-1}, applied to $f=u$ and to $f=\lambda^{2\expoP} r^{n- 2p}\del_i\del_j u$.  We use the $i\leftrightarrow j$ symmetry to reduce the number of terms, and organize terms according to the number of angular derivatives, namely
\be
\aligned
\PHess
& = r^{2 -n+2p} \Bigl(
\xh_i \xh_j \vartheta(\vartheta- 2)
+ \delta_{ij} \vartheta
+ \xh_i \dslash_j (\vartheta-1)
+ \xh_j \dslash_i \vartheta
+ \dslash_i \dslash_j
\Bigr) \Bigl( \\ 
& \quad \qquad \lambda^{2\expoP} r^{n- 2 - 2p} \bigl(
\xh_i \xh_j \vartheta(\vartheta- 2) u
+ \delta_{ij} \vartheta u
+ \xh_i \dslash_j (2\vartheta-1) u
+ \dslash_i \dslash_j u
\bigr) \Bigr)
\\
& = (\vartheta+n- 2 - 2p) (\vartheta+n-4- 2p) \bigl( \lambda^{2\expoP} \vartheta(\vartheta-1) u \bigr) \\
& \quad + (\vartheta+n- 2 - 2p) \bigl( \lambda^{2\expoP} \vartheta(\vartheta+n- 2) u + \lambda^{2\expoP} \Deltaslash u \bigr) \\
& \quad + \xh_i \dslash_j (\vartheta+n-3- 2p)
\bigl( \lambda^{2\expoP} \xh_i \xh_j \vartheta(\vartheta- 2) u
+ \lambda^{2\expoP} \xh_i \dslash_j (2\vartheta-1) u
+ \lambda^{2\expoP} \dslash_i \dslash_j u
\bigr) \\
& \quad + \xh_j \dslash_i (\vartheta+n- 2 - 2p)
\bigl( \lambda^{2\expoP} \xh_i \xh_j \vartheta(\vartheta- 2) u
+ \lambda^{2\expoP} \xh_i \dslash_j (2\vartheta-1) u
+ \lambda^{2\expoP} \dslash_i \dslash_j u
\bigr) \\
& \quad + \dslash_i \dslash_j \bigl(
\lambda^{2\expoP} \xh_i \xh_j \vartheta(\vartheta- 2) u
+ \lambda^{2\expoP} \delta_{ij} \vartheta u
+ \lambda^{2\expoP} \xh_i \dslash_j (2\vartheta-1) u
+ \lambda^{2\expoP} \dslash_i \dslash_j u
\bigr). 
\endaligned
\ee
Thus, we have 
\be 
\aligned
\PHess 
& = \lambda^{2\expoP} (\vartheta+n- 2 - 2p) (\vartheta+n-4- 2p) \vartheta(\vartheta-1) u \\
& \quad + \lambda^{2\expoP} (\vartheta+n- 2 - 2p) \vartheta(\vartheta+n- 2) u + \lambda^{2\expoP} (\vartheta+n- 2 - 2p) \Deltaslash u \\
& \quad + (\vartheta+n-3- 2p)
\bigl( (n-1) \lambda^{2\expoP} \vartheta(\vartheta- 2) u
+ \dslash_i(\lambda^{2\expoP}) (2\vartheta-1) \dslash_i u
+ \lambda^{2\expoP} (2\vartheta- 2) \Deltaslash u
\bigr) \\
& \quad + (\vartheta+n- 2 - 2p)
\bigl( (n-1) \lambda^{2\expoP} \vartheta(\vartheta- 2) u
- \dslash_i(\lambda^{2\expoP}) \dslash_i u
- 2 \lambda^{2\expoP} \Deltaslash u
\bigr) \\
& \quad + (n-1)^2 \lambda^{2\expoP} \vartheta(\vartheta- 2) u
+ \vartheta \Deltaslash(\lambda^{2\expoP} u)
+ n (2\vartheta-1) \dslash_i(\lambda^{2\expoP} \dslash_i u)
+ \dslash_i\dslash_j\bigl(\lambda^{2\expoP} \dslash_i \dslash_j u\bigr), 
\endaligned
\ee
Recalling that $a_{n,p}=2(n- 2 -p)$, we have 
\be
\aligned
\PHess 
& = \lambda^{2\expoP} \vartheta (\vartheta+a_{n,p})
(\vartheta^2 + a_{n,p} \vartheta - 3n + 4 + 2p) u \\
& \quad + (\vartheta+a_{n,p}) \Bigl(
2 \vartheta \dslash_i(\lambda^{2\expoP} \dslash_i u)
+ \Deltaslash(\lambda^{2\expoP} u) - \lambda^{2\expoP} \Deltaslash u \Bigr)
\\
& \quad
- a_{n,p} \Bigl(
\Deltaslash(\lambda^{2\expoP} u)
+ 2 \dslash_i(\lambda^{2\expoP} \dslash_i u) \Bigr)
+ (n-3) \dslash_i(\lambda^{2\expoP}) \dslash_i u
+ 2(n- 2) \lambda^{2\expoP} \Deltaslash u
\\
& \quad
+ \dslash_i\dslash_j\bigl(\lambda^{2\expoP} \dslash_i \dslash_j u\bigr). 
\endaligned
\ee

Combining with the Laplacian and Hessian terms yields 
\bel{art2 -c26}
\aligned
r^4 \lambda^{2\expoP} \notreH^\lambda[u]
& = (n-1) \lambda^{2\expoP} \vartheta (\vartheta+a_{n,p}) \bigl(\vartheta^2 + a_{n,p} \vartheta - b_{n,p}\bigr) u \\
& \quad + (\vartheta+a_{n,p}) \Bigl(
2 \vartheta \dslash_i(\lambda^{2\expoP} \dslash_i u)
+ (n- 2) \vartheta \bigl(\lambda^{2\expoP} \Deltaslash u + \Deltaslash(\lambda^{2\expoP} u) \bigr) \\
& \qquad\qquad
+ \bigl( (n- 2) (n- 2 -a_{n,p}) + 1 \bigr)
\Bigl( \Deltaslash(\lambda^{2\expoP} u)
- \lambda^{2\expoP} \Deltaslash u \Bigr) \Bigr)
\\
& \quad
+ (n- 2) \Deltaslash( \lambda^{2\expoP} \Deltaslash u)
+ \dslash_i\dslash_j\bigl(\lambda^{2\expoP} \dslash_i \dslash_j u\bigr)
+ (n-1) \bigl( \dslash_i(\lambda^{2\expoP}) \dslash_i u + 2 \lambda^{2\expoP} \Deltaslash u \bigr) \\
& \quad
- 2 (a_{n,p} + 1) \dslash_i(\lambda^{2\expoP} \dslash_i u)
- a_{n,p} \bigl( (n- 2) (n- 2 -a_{n,p}) + 1\bigr)
\Deltaslash(\lambda^{2\expoP} u),
\endaligned
\ee
where $b_{n,p} = 2 + (n-3) (2p+2 -n)$. Finally, taking $Z=\lambda^{2\expoP}\Hessslash u$ in~\eqref{art2 -double-div} yields a useful identity
\be
\aligned
 \nablaslash_a\nablaslash_b(\lambda^{2\expoP}\nablaslash^b\nablaslash^a u)
& = \nablaslash\cdot(\nablaslash\cdot(\lambda^{2\expoP}\Hessslash u))
   = \dslash_i \dslash_j\bigl(\lambda^{2\expoP} \dslash_i\dslash_j u + \lambda^{2\expoP} \xh_j\dslash_i u\bigr) + (n-1) \lambda^{2\expoP} \Deltaslash u
\\
& = \dslash_i\dslash_j\bigl(\lambda^{2\expoP} \dslash_i \dslash_j u\bigr)
+ (n-1) \bigl( \dslash_i(\lambda^{2\expoP}) \dslash_i u + 2 \lambda^{2\expoP} \Deltaslash u \bigr),
\endaligned
\ee
which allows us to recast two terms of~\eqref{art2 -c26} in the form $\nablaslash_a\nablaslash_b(\lambda^{2\expoP}\nablaslash^b\nablaslash^a u)$. This establishes the stated decomposition of the localized Hamiltonian operator.


\subsection{Decomposition of the localized momentum}
\label{art2 -section=5.3}

There remains to decompose the operator 
\be
r^2 \notreM^\lambda[Z]^i
= - \frac{1}{2} r^{4-n+2p} \lambda^{- 2\expoP} \del_j \bigl( \lambda^{2\expoP} r^{n- 2 - 2p} (\del_i Z_j + \del_j Z_i) \bigr)
\ee
into radial and angular contributions. Contrarily to the scalar operator $\notreH^\lambda$ where the decomposition was found solely according to the derivatives involved, here there are two additional considerations: we must separate the radial and angular components of the vector~$Z$, and the ones of the vector~$\notreM^\lambda[Z]$.

First of all, we decompose as in~\eqref{art2 -Zi-split} the vector field $Z$ into radial and angular components $\Zperp$ and~$\Zpar$, defined as components along the vector $\del_r=r^{-1}\vartheta$ and orthogonal to it (namely parallel to~$\Lambda$).  The vector $\Zpar$ tangent to $\Sphe^{n-1}$ is described in an overcomplete basis $\dslash_i$, $i=1,\dots,n$.
Next we decompose its first-order derivative according to~\eqref{art2 -delif}, that is,  
\be
\aligned
\del_i Z_j
& = \frac{1}{r} (\xh_i \vartheta + \dslash_i) \bigl( \xh_j\Zperp+ \Zpar_j \bigr)
  = \frac{1}{r} \Bigl( \xh_i\xh_j (\vartheta-1) \Zperp + \xh_j \dslash_i\Zperp + \delta_{ij} \Zperp
+ \xh_i \vartheta \Zpar_j + \dslash_i\Zpar_j \Bigr).
\endaligned
\ee 
Consequently, the momentum operator $r^2 \notreM^\lambda[Z]_i$ equals 
\be
\aligned  
& - \frac{1}{2} \lambda^{- 2\expoP} \bigl(\xh_j (\vartheta+n-3- 2p) + \dslash_j\bigr) \biggl(
\lambda^{2\expoP} \Bigl( 2 \xh_i\xh_j (\vartheta-1) \Zperp + \xh_j \dslash_i\Zperp + \xh_i \dslash_j\Zperp + 2 \delta_{ij} \Zperp 
\\
& \hskip4.cm
+ \vartheta (\xh_i \Zpar_j + \xh_j \Zpar_i) + \dslash_i\Zpar_j + \dslash_j\Zpar_i \Bigr) \biggr)
\\
& = - \frac{1}{2}
(\vartheta+n-3- 2p) \Bigl( 2 \xh_i \vartheta \Zperp + \dslash_i\Zperp + (\vartheta-1) \Zpar_i \Bigr) \\
& \quad - \frac{1}{2} \lambda^{- 2\expoP}
\dslash_j(\lambda^{2\expoP}) \Bigl(
\xh_i \dslash_j\Zperp + 2 \delta_{ij} \Zperp
+ \xh_i \vartheta\Zpar_j + \dslash_i\Zpar_j + \dslash_j\Zpar_i \Bigr)
\\
& \quad - \frac{1}{2}
\Bigl( 2(n-1) \xh_i (\vartheta-1) \Zperp
+ (n+2) \dslash_i\Zperp + \xh_i \dslash_j\dslash_j\Zperp
\\
& \qquad
+ \vartheta (n \Zpar_i + \xh_i \dslash_j\Zpar_j)
+ \dslash_j \dslash_i\Zpar_j + \dslash_j\dslash_j\Zpar_i \Bigr).
\endaligned
\ee 
\bse
We can decompose this operator into a scalar-valued operator and an operator valued in~$T\Lambda$,
\be
\notreM^\lambda[Z]_i
= \xh_i \notreM^{\lambda\perp}[Z] + \notreM^{\lambda\parallel}[Z]_i.
\ee
Each one is then separated into operators acting on $\Zperp$ and $\Zpar$ components of~$Z$. The scalar-valued operator is
\be
\aligned
r^2 \notreM^{\lambda\perp}[Z]
& = r^2 \notreM^{\lambda\perp\perp}[\Zperp] + r^2 \notreM^{\lambda\perp\parallel}[\Zpar], 
\\
r^2 \notreM^{\lambda\perp\perp}[\Zperp]
& = - \vartheta(\vartheta+a_{n,p}) \Zperp
+ (n-1) \Zperp - \frac{1}{2} \lambda^{- 2\expoP} \nablaslash_a \bigl( \lambda^{2\expoP} \nablaslash^a\Zperp \bigr),
\\
r^2 \notreM^{\lambda\perp\parallel}[\Zpar]
& = - \frac{1}{2} \lambda^{- 2\expoP} (\vartheta-1) \nablaslash^a\bigl(\lambda^{2\expoP} \Zpar_a\bigr) + \nablaslash^a\Zpar_a,
\endaligned
\ee
which we have recast in terms of covariant derivatives and with indices $a,b=1,\dots,n-1$ on the sphere. The vector-valued operator is likewise
\be
\aligned
r^2 \notreM^{\lambda\parallel}[Z]_a
& = r^2 \notreM^{\lambda\parallel\perp}[\Zperp]_a + 
r^2 \notreM^{\lambda\parallel\parallel}[\Zpar]_a,
\\
r^2 \notreM^{\lambda\parallel\perp}[\Zperp]_a & = - \frac{1}{2} (\vartheta+a_{n,p}) \nablaslash_a\Zperp - \frac{1}{2} \nablaslash_a\Zperp - \lambda^{- 2\expoP} \nablaslash_a\bigl(\lambda^{2\expoP} \Zperp\bigr),
\\
r^2 \notreM^{\lambda\parallel\parallel}[\Zpar]_a & =
- \frac{1}{2} \vartheta(\vartheta+a_{n,p}) \Zpar_a
+ \frac{1}{2} (a_{n,p}+1) \Zpar_a
- \frac{1}{2} \lambda^{- 2\expoP} \nablaslash^b\Bigl(\lambda^{2\expoP} \bigl( \nablaslash_a\Zpar_b + \nablaslash_b\Zpar_a \bigr) \Bigr).
\endaligned
\ee
\ese
To conclude the derivation, we simply split each component of this decomposition into contributions involving $\vartheta+a_{n,p}$, and the remaining terms. 

\end{document}

%% file: macros-localization.tex
\usepackage{needspace,multirow,accents,comment,xparse}
\renewcommand\dot{\accentset{\bullet}}

\usepackage{amsmath}
\def\over{\csname @@over\endcsname} 

\usepackage{tikz}
\usetikzlibrary{shadings}

\newcommand{\refwithname}[2]{\hyperref[#2]{#1~\ref*{#2}}}

\theoremstyle{plain}
\newtheorem{theorem}{Theorem}[section]
\newcommand{\mynewtheorem}[2]{
  \newaliascnt{#1}{theorem}
  \newtheorem{#1}[#1]{#2}
  \aliascntresetthe{#1}
  \expandafter\providecommand\csname #1autorefname\endcsname{#2}
}
\mynewtheorem{definition}{Definition}
\mynewtheorem{proposition}{Proposition}
\mynewtheorem{lemma}{Lemma}
\mynewtheorem{corollary}{Corollary}
\mynewtheorem{remark}{Remark}
\mynewtheorem{claim}{Claim}

\providecommand{\coloneqq}{\mathrel{\mathop:}\mathrel{\mkern-1.2mu}=}
\providecommand{\eqqcolon}{=\mathrel{\mkern-1.2mu}\mathrel{\mathop:}}

\newcommand{\Hess}{\operatorname{Hess}}
\newcommand \Riem {{\operatorfont Riem}}

\newcommand \la 	{\langle}
\newcommand \ra 	{\rangle}
\newcommand \dbf 	{{\mathbf d}}

\newcommand \bei 	{\begin{itemize}}
\newcommand \eei 	{\end{itemize}}
\newcommand \del	{\partial}

\newcommand \Ecal 	{\mathcal E}
\newcommand \Fcal 	{\mathcal F}

\newcommand \RR 	{\mathbb R}
\newcommand \Sbb 	{\mathbb S}
\newcommand \eps 	{\epsilon}
\newcommand \be 	{\begin{equation}\mathopen{}} 
\newcommand \ee 	{\end{equation}}
\newcommand \bel [1]	{\be \label{#1}\mathopen{}} 
\newcommand \Bcal 	{\mathcal B}

\newcommand \Acal {\mathcal A}

\newcommand \Mbf 	{\mathbf M}

\newcommand \bse {\begin{subequations}}
\newcommand \ese {\ifcase\value{equation}\relax\GenericError{}{Subequation (\theequation) is empty}{}{}\or\GenericError{}{Subequation (\theequation) with single equation}{}{}\fi\end{subequations}}

\DeclareMathOperator \Tr {\bf Tr}
\DeclareFontFamily{U}{matha}{\hyphenchar\font45}
\DeclareFontShape{U}{matha}{m}{n}{ <5> <6> <7> <8> <9> <10> gen * matha <10.95> matha10 <12> <14.4> <17.28> <20.74> <24.88> matha12 }{}
\DeclareSymbolFont{matha}{U}{matha}{m}{n}
\DeclareFontSubstitution{U}{matha}{m}{n}
\DeclareFontFamily{U}{mathx}{\hyphenchar\font45}
\DeclareFontShape{U}{mathx}{m}{n}{ <5> <6> <7> <8> <9> <10> <10.95> <12> <14.4> <17.28> <20.74> <24.88> mathx10 }{}
\DeclareSymbolFont{mathx}{U}{mathx}{m}{n}
\DeclareFontSubstitution{U}{mathx}{m}{n}
\DeclareMathAccent{\widecheck}{0}{mathx}{"71}
\DeclareMathAccent{\wideparen}{0}{mathx}{"75}

\DeclareMathDelimiter{\VV}{0}{matha}{"7E}{mathx}{"17} 
\ExplSyntaxOn
\NewDocumentCommand\norm{som}{%
  \IfBooleanTF{#1}{\left\lVert#3\right\rVert}%
    {\IfValueTF{#2}{\use:c{\cs_to_str:N#2l}\lVert#3\use:c{\cs_to_str:N#2r}\rVert}{\lVert#3\rVert}}}
\NewDocumentCommand\Norm{som}{%
  \IfBooleanTF{#1}{\left\VV #3\right\VV}%
    {\IfValueTF{#2}{\use:c{\cs_to_str:N#2l}\VV #3\use:c{\cs_to_str:N#2r}\VV}{\mathopen{\VV} #3\mathclose{\VV}}}}
\ExplSyntaxOff

\makeatletter
\providecommand\clap[1]{\hbox to 0pt{\hss#1\hss}}
\providecommand\mathclap[1]{\mathpalette\mathclap@aux{#1}}
\newcommand\mathclap@aux[2]{\clap{$#1#2$}}
\providecommand\mathllap[1]{\mathpalette\mathllap@aux{#1}}
\newcommand\mathllap@aux[2]{\llap{$#1#2$}}
\providecommand\mathrlap[1]{\mathpalette\mathrlap@aux{#1}}
\newcommand\mathrlap@aux[2]{\rlap{$#1#2$}}
\makeatother

\newcommand \Obig {\mathcal O}
\usepackage{calligra}
\DeclareMathAlphabet{\mathcalligra}{T1}{calligra}{m}{n}
\newsavebox{\calligrabox}

\makeatletter
\newcommand \lot {{\bf l.o.t.}\@ifnextchar.\@gobble{}} 
\makeatother

\newcommand \Vol {\operatorname{\bf Vol}}
\renewcommand \ker {\operatorname{\bf ker}}


\newcommand{\putabove}[2]{%
  \ifx\Tr#2\!\fi
  \mathrlap{\overset{\hbox{\lower1.5pt\hbox{\smash{$#1$}}}}{\phantom{#2}}}#2}

\newcommand{\gdiff}{\gamma} 
\newcommand{\hdiff}{\eta} 

\newcommand \Poin {\textnormal{\textbf{P}}}  
\newcommand \Korn {\textnormal{\textbf{K}}}       
\newcommand \Hardy {\textnormal{\textbf{H}}}

\newcommand \dslash {\slashed\del{}}
\newcommand \Deltaslash {\slashed\Delta{}}
\newcommand \Hessslash {\mathop{\slashed{\Hess}{}}\nolimits}
\newcommand \nablaslash {\slashed\nabla{}}

\newcommand \xh {\widehat{x}}

\newcommand \ut {\widetilde{u}}

\newcommand \unquad {\hspace{-1em}\ifmmode\mathopen{}\fi}

\newcommand \astar {a_\star}

\newcommand \Sphe {\mathbf{S}}

\newcommand \Arr {\Acal}
\newcommand \Brr {\Bcal}

\newcommand \Ars {\slashed \Acal}
\newcommand \Brs {\slashed \Bcal}

\newcommand \nut{\widetilde\nu}

\newcommand \gslash {\slashed{g}{}}
\newcommand \Rslash {\slashed{R}{}}

\newcommand \Zpar {Z^{\parallel}}
\newcommand \Zperp {Z^{\perp}}

\newcommand \xipar {\xi^{\parallel}}
\newcommand \xiperp {\xi^{\perp}}

\newcommand{\ssAAUX}{\slashed{\Acal}}
\newcommand{\ssA}{\slashed{\ssAAUX}{}}
\declareslashed{}{/}{.1}{0}{\Acal}
\declareslashed{}{/}{.3}{0}{\ssAAUX}
\newcommand{\ssBAUX}{\slashed{\Bcal}}
\newcommand{\ssB}{\slashed{\ssBAUX}{}}
\declareslashed{}{/}{.1}{0}{\Bcal}
\declareslashed{}{/}{.3}{0}{\ssBAUX}
\newcommand{\ssFAUX}{\slashed{\Fcal}}

\declareslashed{}{/}{.1}{0}{F}
\declareslashed{}{/}{.3}{0}{\ssFAUX}
\newcommand{\ssrmAAUX}{\slashed{\mathrm{A}}}
\newcommand{\ssrmA}{\slashed{\ssrmAAUX}{}}
\declareslashed{}{/}{.1}{0}{\mathrm{A}}
\declareslashed{}{/}{.3}{0}{\ssrmAAUX}
\newcommand{\ssrmBAUX}{\slashed{\mathrm{B}}}
\newcommand{\ssrmB}{\slashed{\ssrmBAUX}{}}
\declareslashed{}{/}{.1}{0}{\mathrm{B}}
\declareslashed{}{/}{.3}{0}{\ssrmBAUX}

\newcommand \notreH {\mathscr{H}}

\newcommand \notreM {\mathscr{M}}

\AtBeginDocument{%
  \mathchardef\ordinarycolon=\mathcode`\:
  \mathcode`\:="8000
  \mathchardef\ordinaryequal=\mathcode`\=
  \mathcode`\=="8000
}
\begingroup
\makeatletter
\lccode`~=`: \lowercase{\gdef~{\@ifnextchar={\coloneqq\@gobble}%
    {\begingroup\def~{\ordinarycolon}\colon\endgroup}}}
\lccode`~=`= \lowercase{\gdef~{\@ifnextchar:{\eqqcolon\@gobble}{\ordinaryequal}}}
\endgroup
\renewcommand{\coloneqq}{\mathrel{\mathop\ordinarycolon}\mathrel{\mkern-1.2mu}\ordinaryequal}
\renewcommand{\eqqcolon}{\ordinaryequal\mathrel{\mkern-1.2mu}\mathrel{\mathop\ordinarycolon}}

\newcommand \expoP       {P}

\renewcommand \Vol {\textnormal{\textbf{V}}}

\newcommand \bfDiv {\textnormal{\textbf{Div}}}

\newcommand \aire {\textbf{Area}}

\newcommand \normal {\textnormal{\textbf{si}}} 

\newcommand \Fpar {F^{\parallel}}
\newcommand \Fperp {F^{\perp}}

\newcommand \cperp {c^{\perp}}

\newcommand \Mu {\mathrm{M}}
\newcommand \Nu {\mathrm{N}}

\newcommand \Chi {\mathrm{X}}
\newcommand \Kappa {\mathrm{K}}

\newcommand \Sym 	{{\bf Sym}}

\newcommand{\PsiHquabeta}{\Psi^{\notreH}_{\textnormal{qua},\beta}}
\newcommand{\PsiHbil}{\Psi^{\notreH}_{\textnormal{bil}}}

\newcommand{\psiHscabeta}{\psi^{\notreH}_{\textnormal{sca},\beta}}

\newcommand{\psiHtenbeta}{\psi^{\notreH}_{\textnormal{ten},\beta}}

\newcommand \onec {c_{\textnormal{i}}}
\newcommand \twoc {c_{\textnormal{ii}}}

\newcommand \tcnp{\widetilde{c}_{n,p}}

\newcommand \onegamma {\gamma_n^{(1)}}
\newcommand \twogamma {\gamma_n^{(2)}}
\newcommand \threegamma {\gamma_n^{(3)}}

\newcommand \bnotreH {b^\notreH}

\newcommand\gxedaux{\slashed{g}}

\declareslashed{}{\backslash}{.1}{0}{\gxedaux}
\newcommand\nablaxedaux{\slashed{\nabla}}

\declareslashed{}{\backslash}{.1}{0}{\nablaxedaux}
\newcommand\Deltaxedaux{\slashed{\Delta}}

\declareslashed{}{\backslash}{.1}{0}{\Deltaxedaux}
\newcommand\uxedaux{\slashed{u}}

\declareslashed{}{\backslash}{.1}{0}{\uxedaux}
\newcommand\vxedaux{\slashed{v}}

\declareslashed{}{\backslash}{.1}{0}{\vxedaux}
\newcommand\kxedaux{\slashed{k}}

\declareslashed{}{\backslash}{.1}{0}{\kxedaux}

\newcommand \projP {P}

\newcommand \CzeroKP{C_{0\Korn\Poin}^\lambda}
\newcommand \ConeKP{C_{1\Korn\Poin}^\lambda}

\newcommand \CPoin{C_{\Poin}^{\lambda}}
\newcommand \CPoinTwo{C_{\Poin 2}^{\lambda}}

\newcommand \detbf {\operatorname{\bf det}}

\newcommand \PLap {P^{\textnormal{Lap}}}
\newcommand \PHess {P^{\textnormal{Hess}}}

\newcommand \Id  I    

\newcommand \diam {{\textnormal{\textbf{diam}}}}

\newcommand\Cstar{C_\star}
\newcommand\xistar{\xi^{\star}}
\newcommand\xistarpar{\xi^{\star\parallel}}
\newcommand\xistarperp{\xi^{\star\perp}}

\newcommand\cstun             {\varpi_1}  
\newcommand\cstdeux          {\varpi_2} 

\newcommand\unL{\underline{L}}
\newcommand\unH{\underline{H}}
 
\newcommand \wvect {w}

\newcommand{\matV}{V}
\newcommand{\matG}{G}

\newcommand \matQ {\Theta}
\newcommand{\zeroXi}{\Sigma}
\newcommand{\fluc}{\mathrm{fl}}


%% file: BLF-PLF-localization-inequalities.bbl
\begin{thebibliography}{99}

 \bibitem{BrendleHuisken:2017}
{\sc S. Brendle and G. Huisken,}
A fully nonlinear flow for two-convex hypersurfaces in Riemannian manifolds,
Invent. Math. 210 (2017), 559--613.

\bibitem{Carlotto-Review}
{\sc A. Carlotto,}
The general relativistic constraint equations, 
Living Reviews in Relativity 24 (2021), 2. 

\bibitem{CarlottoSchoen}
{\sc A. Carlotto and R. Schoen,} 
Localizing solutions of the Einstein constraint equations, 
Invent. Math. 205 (2016), 559--615.

\bibitem{CDHS}
{\sc O. Chodosh, J.M. Daniels-Holgate, and F. Schulze,}
Mean curvature flow from conical singularities, 
Invent. Math. 238 (2024), 1041--1066.

\bibitem{Choquet-book} 
{\sc Y. Choquet-Bruhat,}
\emph{General relativity and the Einstein equations,}
Oxford Math. Monographs, Oxford University Press,Oxford, 2009.

\bibitem{Chrusciel-bourbaki}
{\sc P.T. Chru\'sciel,}
Anti-gravit\'e \`a la Carlotto-Schoen [after Carlotto and Schoen],
S\'eminaire Bourbaki, Vol. 2016/2017, Exp. 1120,  
Ast\'erisque 407 (2019), 1-- 25. 

\bibitem{ChruscielDelay-memoir}
{\sc P.T. Chru\'sciel and E. Delay,}
On mapping properties of the general relativistic constraints operator in weighted function spaces with applications, 
M\'em. Soc. Math. Fr. (N.S.), no.~94 (2003), 109~pp.

\bibitem{ChruscielDelay-2021}
{\sc P.T. Chru\'sciel and E. Delay,}
On Carlotto-Schoen-type scalar-curvature gluings, 
Bull. Soc. Math. France 149 (2021), 641--662. 
%
\bibitem{ColdingMinicozzi}
{\sc T.H. Colding and W.P. Minicozzi II,}
Generic mean curvature flow I: generic singularities,
Ann. of Math. 175 (2012), 755--833. 
%
\bibitem{Corvino-2000} 
{\sc J. Corvino,}
Scalar curvature deformation and a gluing construction for the Einstein constraint equations, 
Comm. Math. Phys. 214 (2000), 137--189. 

\bibitem{CorvinoEM} 
{\sc J. Corvino, M. Eichmair, and P. Miao,}
Deformation of scalar curvature and volume, 
Math. Ann. 357 (2013), 551--584.

\bibitem{CorvinoHuang} 
{\sc J. Corvino and L.-H. Huang,} 
Localized deformation for initial data sets with the dominant energy condition, 
Calc. Var. Partial Differential Equations 59 (2020), 42.

\bibitem{CorvinoSchoen} 
{\sc J. Corvino and R. Schoen,}
On the asymptotics for the vacuum Einstein constraint equations, 
J. Differential Geom. 73 (2006), 185-- 217.

\bibitem{Delay}
{\sc E. Delay,}
Localized gluing of Riemannian metrics in interpolating their scalar curvature, 
Differential Geom. Appl. 29 (2011), 433--439.  

\bibitem{FischerMarsden-1973} 
{\sc A.E. Fischer and J.E. Marsden,}
Linearization stability of the Einstein equations,
Bull. Amer. Math. Soc. 79 (1973), 997--1003.

\bibitem{GallowayMiaoSchoen} 
{\sc G.J. Galloway, P. Miao, and R. Schoen,}
Initial data and the Einstein constraint equations,
in \emph{General Relativity and Gravitation: A Centennial Perspective},
edited by A.~Ashtekar, B.~Berger, J.~Isenberg, and M.~MacCallum,
Cambridge University Press, Cambridge, 2015, pp.~412--448. 

\bibitem{LL-Letter}
{\sc B. Le Floch and P.G. LeFloch,} 
Optimal shielding for Einstein gravity, 
Class. Quantum Grav. 41 (2024), 13LT02.  

\bibitem{LL-optimal-main}
{\sc B. Le Floch and P.G. LeFloch,}
Optimal localization for the Einstein constraints.
First version in 2023 available at \url{https://arxiv.org/abs/2312.17706v1}. Second version in 2026 available at \url{https://arxiv.org/abs/2312.17706v2}.

\bibitem{LL-next}
{\sc B. Le Floch and P.G. LeFloch,} 
The seeds and silhouettes of optimal gravitational shielding, 
in preparation.

\bibitem{LeFlochNguyen}
{\sc P.G. LeFloch and T.-C. Nguyen,}
The seed-to-solution method for the Einstein constraints and the asymptotic localization problem, 
J. Funct. Anal. 285 (2023), 110106.

\bibitem{Moncrief-1975}
{\sc V. Moncrief,}
Spacetime symmetries and linearization stability of the Einstein equations. I, 
J. Math. Phys. 16 (1975), 493--498. 

\end{thebibliography}
